\RequirePackage{fix-cm}
\documentclass[]{article}
\usepackage{arxiv}
\usepackage[english]{babel}
\usepackage{lipsum}
\usepackage{graphicx}
\usepackage{placeins}
\usepackage{bm}
\usepackage{algorithm}
\usepackage[italicComments=true,indLines=true,noEnd=false]{algpseudocodex}
\usepackage{tabularx}
\usepackage{multirow}
\usepackage{xcolor}
\usepackage{subfigure}
\usepackage{float}
\usepackage{color, colortbl}
\definecolor{ao(english)}{rgb}{0.0, 0.5, 0.0}
\usepackage{xr-hyper}
\usepackage{hyperref}
\hypersetup{colorlinks,linkcolor={ao(english)},citecolor={ao(english)},urlcolor={red}} 
\usepackage{stackrel,amssymb}
\usepackage{amsfonts,amsmath,amstext,amsbsy,amsthm}
\theoremstyle{definition}
\newtheorem{prop}{Proposition}

\newtheorem{theorem}{Theorem}[section]

\usepackage{cleveref}
\definecolor{LightCyan}{rgb}{0.88,1,1}
\definecolor{LightRed}{rgb}{1,0.7,0.7}

\def\bp{\boldsymbol{p}}

\def\bw{\boldsymbol{w}}

\newcommand{\R}{\mathbb{R}}

\title{Parsimonious Physics-Informed Random Projection Neural Networks for Initial-Value Problems of ODEs and index-1 DAEs}

\author{
  Gianluca Fabiani\\
  Scuola Superiore Meridionale\\
  Universit\`a degli Studi di Napoli Federico II\\
  Italy\\
  \texttt{gianluca.fabiani@unina.it} \\
   \And
 Evangelos Galaris\\
 Dept. of Mathematics and Applications\\ Universit\`a degli Studi di Napoli Federico II,\\
 Italy\\
 \texttt{evangelos.galaris@unina.it}
  \And
  Lucia Russo\\
Scienze e Tecnologie per l'Energia e la Mobilità Sostenibili\\
Consiglio Nazionale delle Ricerche\\
Italy\\
\texttt{lucia.russo@stems.cnr.it }
 \And
Constantinos Siettos\thanks{Corresponding author}\\
Dept. of Mathematics and Applications,\\
Scuola Superiore Meridionale\\
Universit\`a degli Studi di Napoli ``Federico II"\\
Italy\\
\texttt{constantinos.siettos@unina.it}
  }

\usepackage{amsopn}

\begin{document}

\maketitle

\begin{abstract}
We address a physics-informed neural network based on the concept of random projections for the numerical solution of initial value problems of nonlinear ODEs in linear-implicit form and index-1 DAEs, which may also arise from the spatial discretization of PDEs. The proposed scheme has a single hidden layer with appropriately randomly parametrized Gaussian kernels and a linear output layer, while the internal weights are fixed to ones. The unknown weights between the hidden and output layer are computed by Newton's iterations, using the Moore-Penrose pseudoinverse for low to medium scale, and sparse QR decomposition with $L^2$ regularization for medium to large scale systems. To deal with stiffness and sharp gradients, we thus propose an variable step size scheme based on the elementary local error control algorithm for adjusting the step size of integration and address a natural continuation method for providing good initial guesses for the Newton iterations. Building on previous works on random projections, we prove the approximation capability of the scheme for ODEs in the canonical form and index-1 DAEs in the semiexplicit form. The ``optimal" bounds of the uniform distribution from which the values of the shape parameters of the Gaussian kernels are sampled are ``parsimoniously" chosen based on the bias-variance trade-off decomposition, thus using the stiff van der Pol model as the reference solution for this task. The optimal bounds are fixed once and for all the problems studied here. In particular, the performance of the scheme is assessed through seven benchmark problems. Namely, we considered four index-1 DAEs, the Robertson model, a no autonomous model of five DAEs describing the motion of a bead on a rotating needle, a non autonomous model of six DAEs describing a power discharge control problem, the chemical so-called Akzo Nobel problem and three stiff problems, the Belousov-Zhabotinsky model, the Allen-Cahn PDE phase-field model and the Kuramoto-Sivashinsky PDE giving rise to chaotic dynamics. The efficiency of the scheme in terms of both numerical accuracy and computational cost is compared with three stiff/DAE solvers (\texttt{ode23t}, \texttt{ode23s}, \texttt{ode15s}) of the MATLAB ODE suite. Our results show that, the proposed scheme outperforms the aforementioned stiff solvers in several cases, especially in regimes where high stiffness and/or sharp gradients arise, in terms of numerical accuracy, while the computational costs are for any practical purposes comparable.
\end{abstract}



\section{Introduction\label{sec:intro}}
The interest in using machine learning as an alternative to the classical numerical analysis methods \cite{gear1971numerical,brenan1995numerical,gear1990introduction,shampine1997matlab} for the solution of the inverse \cite{Krischer1993,Masri,chen1995universal,gonzalez1998identification,siettos2002truncated,siettos2002semiglobal,alexandridis2002modelling}, and forward problems \cite{lee1990neural,dissanayake1994neural,meade1994numerical,gerstberger1997feedforward,lagaris1998artificial} in differential equations modelling dynamical systems can be traced back three decades ago.
Today, this interest has been boosted together with our need to better understand and analyse the emergent dynamics of complex multiphysics/ multiscale dynamical systems of fundamental theoretical and technological importance \cite{karniadakis2021physics}. The objectives are mainly two. First, that of the solution of the inverse problem, i.e. that of identifying/discovering the hidden macroscopic laws, thus learning nonlinear operators and constructing coarse-scale dynamical models of ODEs and PDEs and their closures, from microscopic large-scale simulations and/or from multi-fidelity observations \cite{bongard2007automated,raissi2017machine,raissi2018numerical,raissi2019physics,rudy2019data,bertalan2019learning,arbabi2020linking,Lee2020,vlachas2020backpropagation,chen2021physics,chen2021solving,lu2021learning}. Second, based on the constructed coarse-scale models, to systematically investigate their dynamics by efficiently solving the corresponding differential equations, especially when dealing with (high-dimensional) PDEs \cite{fabiani2021numerical,calabro2021extreme,chen2021physics,chen2021solving,dong2021local,dong2021modified,ji2021stiff,lu2021deepxde,raissi2019physics,schiassi2021extreme}.
Towards this aim, physics-informed machine learning \cite{raissi2017machine,raissi2018numerical,raissi2019physics,lu2021learning,meng2020ppinn,chen2021physics,chen2021solving,karniadakis2021physics} has been addressed to integrate available/incomplete information from the underlying physics, thus relaxing the ``curse of dimensionality''. However, failures may arise at the training phase especially in deep learning formulations, while there is still the issue of the corresponding computational cost ~\cite{larochelle2009exploring,wang2021understanding,wang2022and,karniadakis2021physics}. Thus, a bet and a challenge is to develop physics-informed machine learning methods that can achieve high approximation accuracy at a low computational cost.

Within this framework, and towards this aim, we propose a physics-informed machine learning (PIRPNN) scheme based on the concept of random projections \cite{johnson1984extensions,rahimi2008weighted,gorban2016approximation}, for the numerical solution of initial-value problems of nonlinear ODEs and index-1 DAEs as these may also arise from the spatial discretization of PDEs. Our scheme consists of a single hidden layer, with Gaussian kernels, in which the weights between the input and hidden layer are fixed to ones. The shape parameters of the Gaussian kernels are random variables drawn from a uniform distribution which bounds are ``parsimoniously'' chosen based  on the expected bias-variance trade-off \cite{belkin2019reconciling}. The unknown parameters, i.e. the weights between the hidden and the output layer are estimated by solving with Newton-type iterations a  system of nonlinear algebraic equations. For low to medium scale systems this is achieved using SVD decomposition, while for large scale systems, we exploit a sparse QR factorization algorithm with $L^2$ regularization \cite{davis2009user}. Furthermore, to facilitate the convergence of Newton's iterations, especially at very stiff regimes and regimes with very sharp gradients, we (a) propose a variable step size scheme for adjusting the interval of integration based on the elementary local error control algorithm \cite{soderlind2002automatic}, and (b) address a natural continuation method for providing good initial guesses for the unknown weights.
To demonstrate the performance of the proposed method in terms of both approximation accuracy and computational cost, we have chosen seven benchmark problems, four index-1 DAEs and three stiff problems of ODEs, thus comparing it with the \texttt{ode23s}, \texttt{ode23t} and \texttt{ode15s} solvers of the MATLAB suite ODE \cite{shampine1997matlab}. In particular, we considered the index-1 DAE Robertson model describing the kinetics of an autocatalytic reaction \cite{Robertson1966,shampine1999solving}, an index-1 DAEs model describing the motion of a bead on a rotating needle \cite{shampine1999solving}, an index-1 DAEs model describing the dynamics of a power discharge control problem \cite{shampine1999solving}, the index-1 DAE chemical Akzo Nobel problem \cite{mazzia2012test,stortelder1998parameter}, the  Belousov-Zabotinsky chemical kinetics stiff ODEs \cite{belousov1951periodic,zhabotinsky1964periodical}, the Allen-Chan phase-field PDE describing the process of phase separation for generic interfaces \cite{allen1979microscopic} and Kuramoto-Sivashinsky PDE \cite{kuramoto1978diffusion,sivashinsky1977nonlinear,trefethen2000spectral}. The PDEs are discretized in space with central finite differences, thus resulting to a system of stiff ODEs \cite{trefethen2000spectral}. The results show that the proposed scheme outperforms the aforementioned solvers in several cases in terms of numerical approximation accuracy, especially in problems where stiffness and sharp gradients arise, while the required computational times are comparable for any practical purposes.

\section{Methods}
First, we describe the problem, and present some preliminaries on the use of machine learning for the solution of differential equations and on the concept of random projections for the approximation of continuous functions. We then present the proposed physics-informed random projection neural network scheme, and building on previous works \cite{rahimi2008weighted}, we prove that in principle the proposed PIRPNN can approximate with any given accuracy any unique continuously differentiable function that satisfies the Picard-Lindel\"of Theorem. Finally, we address (a) a variable step size scheme for adjusting the interval of integration and (b) a natural continuation method, to facilitate the convergence of Newton iterations, especially in regimes where stiffness and very sharp gradients arise.
\subsection{Description of the Problem\label{sec:preliminaries}}
We focus on IVPs of ODEs and index-1 DAEs that may also arise from the spatial discretization of PDEs using for example finite differences, finite elements and spectral methods. In particular, we consider IVPs in the linear implicit form of:
\begin{equation}\label{eq:DAE}
  \begin{array}{lll}
       \boldsymbol{M}\dfrac{d\boldsymbol{u}(t)}{dt}  =  \boldsymbol{f}(t,\boldsymbol{u}(t)), \qquad
      \boldsymbol{u}(0)  =  \boldsymbol{z}.
   \end{array} 
\end{equation}
$\boldsymbol{u}\in \mathbb{R}^{m}$ denotes the set of the states $\{u_1, u_2, \dots, u_i,\dots,  u_m\}$, $\boldsymbol{M}\in \mathbb{R}^{m\times m}$ is the so-called mass matrix with elements $M_{ij}$, $\boldsymbol{f}: \R \times \mathbb{R}^m \to \mathbb{R}^m$ denotes the set of Lipschitz continuous multivariate functions, say $f_i(t, u_1,u_2,...,u_m)$ in some closed domain $D$, and $\boldsymbol{z}\in \mathbb{R}^{m}$ are the initial conditions. When $\boldsymbol{M}=\boldsymbol{I}$, the system reduces to the canonical form.
The above formulation includes problems of DAEs when $\boldsymbol{M}$ is a singular matrix, including semiexplicit DAEs in the form \cite{shampine1999solving}:
\begin{equation}\label{eq:semiexplicitDAE}
  \begin{array}{lll}
       \dfrac{d\boldsymbol{u}(t)}{dt}  =  \boldsymbol{f}(t, \boldsymbol{u}(t),\boldsymbol{v}(t)), \qquad \boldsymbol{u}(0)  =  \boldsymbol{z},\\ 
       \boldsymbol{0}=\boldsymbol{g}(t,\boldsymbol{u}(t),\boldsymbol{v}(t)),
   \end{array} 
\end{equation}
where, we assume that the Jacobian $\nabla_{\boldsymbol{v}} \boldsymbol{g}$ is nonsingular, $\boldsymbol{f}: \R \times \mathbb{R}^{m-l} \times \mathbb{R}^{l}\to \mathbb{R}^{m-l}, \boldsymbol{g}: \R \times \mathbb{R}^{m-l} \times \mathbb{R}^{l}\to \mathbb{R}^{l}$.

In this work, we use physics-informed random projection neural networks for the numerical solution of the above type of IVPs which solutions are characterized both by sharp gradients and stiffness \cite{shampine1979user,shampine1999solving}.

Stiff problems are the ones which integration ``with a code that aims at non stiff problems proves conspicuously inefficient for no obvious reason (such a severe lack of smoothness in the equation or the presence of singularities)''\cite{shampine1979user}. 
At this point, it is worthy to note that stiffness is not connected to the presence of sharp gradients \cite{shampine1979user}. For example, at the regimes where the relaxation oscillations of the van der Pol model exhibit very sharp changes resembling discontinuities, the equations are not stiff. 

\subsection{Physics-informed machine learning for the solution of differential equations\label{sec:fnn}}
Let's assume a set of $n_{x}$ points $\boldsymbol{x}_i\in \Omega \subset \R^{d}$ of the independent (spatial) variables, thus defining the size grid in the domain $\Omega$, $n_{\partial \Omega}$ points along the boundary $\partial \Omega$ of the domain and $n_t$ points in the time interval, where the solution is sought. For our illustrations, let's consider a time-dependent PDE in the form of
\begin{equation}
  \dfrac{\partial u}{\partial t} = L(\boldsymbol{x}, u, \nabla u, \nabla^2 u, \dots),
  \label{PDE}
\end{equation}
where $L$ is the partial differential operator acting on $u$ satisfying the boundary conditions $B u =g, \mbox{ in } \partial \Omega$, where $B$ is the boundary differential operator.
Then, the solution with machine learning of the above PDE involves the solution of a minimization problem of the form:
\begin{align}
     \label{eq:cost}
    \min_{\boldsymbol{P},\boldsymbol{Q}} E(\boldsymbol{P},\boldsymbol{Q}) := 
    \sum_{i=1}^{n_x}\sum_{j=1}^{n_t} \left\| \dfrac{d \Psi}{dt} (\cdot) - L (\boldsymbol{x}_i,\Psi(\cdot), \nabla \Psi(\cdot), \nabla^2 \Psi(\cdot), \dots) \right\|^2 + \\\nonumber
    \sum_{j=1}^{n_{\partial \Omega}} \left\| B \Psi(\cdot) -g \right\|^2,
\end{align}
where $\Psi(\cdot):=\Psi(\boldsymbol{x}_i, t_j,  \mathcal{N}( \boldsymbol{x}_i, t_j, \boldsymbol{P},\boldsymbol{Q}))$ represents a machine learning constructed function approximating the solution $u$ at $\boldsymbol{x}_i$ at time $t_j$ and $\mathcal{N}( \boldsymbol{x}_i, t_j,  \boldsymbol{P},\boldsymbol{Q})$ is a machine learning algorithm; $\boldsymbol{P}$ contains the parameters of the machine learning scheme (e.g. for a FNN the internal weights $\boldsymbol{W}$, the biases $\boldsymbol{B}$, the weights between the last hidden and the output layer $\boldsymbol{W}^{o}$), $\boldsymbol{Q}$ contains the hyper-parameters (e.g. the parameters of the activation functions for a FNN, the learning rate, etc.). In order to solve the optimization problem~\eqref{eq:cost}, one usually needs quantities such as the derivatives of $\mathcal{N}(\boldsymbol{x}, \boldsymbol{P}, \boldsymbol{Q})$ with respect to $t,\boldsymbol{x}$ and the parameters of the machine learning scheme (e.g. the weights and biases for a FNN). These can be obtained in several ways, numerically using finite differences or other approximation schemes, or by symbolic or automatic differentiation \cite{baydin2018automatic,lu2021deepxde}). The above approach can be directly implemented also for solving systems of ODEs as these may also arise by discretizing in space PDEs.
Yet, for large scale problems even for the simple case of single layer networks, when the number of hidden nodes is large enough, the solution of the above optimization problem is far from trivial. Due to the ``curse of dimensionality" the computational cost is high, while when dealing with differential equations which solutions exhibit sharp gradients and/or stiffness several difficulties or even failures in the convergence have been reported ~\cite{larochelle2009exploring,wang2021understanding,wang2022and}.

\subsection{Random Projection Neural Networks\label{sec:rand_proj_rbfnn}}
Random projection neural networks (RPNN) including Random Vector Functional Link Networks (RVFLNs) \cite{igelnik1995stochastic,husmeier1999random}, Echo-State Neural Networks and Reservoir Computing \cite{jaeger2002adaptive}, and Extreme Learning Machines \cite{huang2006extreme,huang2014insight} have been introduced to tackle the ``curse of dimensionality'' encountered at the training phase. One of the fundamental works on random projections is the celebrated Johnson and Lindenstrauss Lemma \cite{johnson1984extensions} stating that for a matrix $\boldsymbol{W} \in \R^{d \times n}$ containing $n$, $\boldsymbol{w}$ points in $\R^{d}$, there exists a projection $\boldsymbol{F}: \mathbb{R}^d \rightarrow \mathbb{R}^k$ defined as:
\begin{equation}\label{JL}
\boldsymbol{F} (\boldsymbol{w}) = \dfrac{1}{\sqrt{k}} \boldsymbol{R} \boldsymbol{w},
\end{equation}
where $\boldsymbol{R} = [r_{ij}] \in \R^{k \times d}$ has components which are i.i.d.~random variables sampled from a normal distribution, which maps $\boldsymbol{W}$ into a random subspace of dimension $k \geq O \biggl(\dfrac{\ln{n}}{\epsilon^2}\biggr)$, where the distance between any pair of points in the embedded space $\boldsymbol{F}(\boldsymbol{W})$ is bounded in the interval $[1-\epsilon \quad 1+\epsilon]$.

Regarding single layer feedforward neural networks (SLFNNs), in order to improve the approximation accuracy, Rosenblatt \cite{rosenblatt1962perceptions} suggested the use of randomly parametrized activation functions for single layer structures. Thus, the approximation of a sufficiently smooth function $f(\boldsymbol{x}):\R^d\rightarrow \R$ is written as a linear combination of appropriately randomly parametrized family of $N$ basis functions $\phi_i: \R\times \R^p \rightarrow \R$  as
\begin{equation}\label{eq:randomexpansion}
    f(\boldsymbol{x})\simeq  f_{N}(\boldsymbol{x})= \sum_{i=1}^{N}w_{i}^{o}\phi_i(\boldsymbol{w}_{i}^T\boldsymbol{x}+b_i,\boldsymbol{p}_i),
\end{equation}
where $\boldsymbol{w}_{i} \in \R^d$ are the weighting coefficients of the inputs, $b \in \R$ are the biases and $ \boldsymbol{p}_{i} \in \R^p$ are the shape parameters of the basis functions.

More generally, for SLFNNs with $d$ inputs, $k$ outputs and $N$ neurons in the hidden layer, the random projection of $n$ samples in the $d$-dimensional input space $\boldsymbol{X}$ can be written in a matrix-vector notation as:
\begin{equation} \label{eq:lin_map}
    \boldsymbol{Y}_{N} = \boldsymbol{\Phi}_N {\boldsymbol{W}^{o}}, \quad \boldsymbol{Y} \in \mathbb{R}^{n\times k},
\end{equation}
where, $\boldsymbol{\Phi}_N \in \mathbb{R}^{n \times N}$ is a random matrix containing the outputs of the hidden layer as shaped by the $n$ samples in the $d$-dimensional space, the randomly parametrized internal weights $\boldsymbol{W}\in \R^{d\times N}$, the biases $\boldsymbol{b} \in \R^N$ and shape parameters of the $N$ activation functions; $\boldsymbol{W}^o \in \R^{N \times k}$ is the matrix containing the weights $w^{o}_{ij}$ between the hidden and the output layer.

In the early '90s, Barron \cite{barron1993universal} proved that for functions with integrable Fourier transformations, a random sample of the parameters of sigmoidal basis functions from an appropriately chosen distribution  results to an approximation error of the order of $1/N$. Igelnik and Pao \cite{igelnik1995stochastic} extending Barron's proof \cite{barron1993universal} for any family of $L^2$ integrable basis functions $\phi_i$ proved the following theorem for RVFLNs:
\begin{theorem}
\label{thm:SLFNN_universal}
SLFNNs as defined by Eq.(\ref{eq:randomexpansion}) with weights $\boldsymbol{w}_{i}$ and biases $b_i$ selected randomly from a uniform distribution and for any family of $L^2$ integrable basis functions $\phi_i$, are universal approximators of any Lipschitz continuous function $f$ defined in the standard hypercube $I^d$ and the expected rate of convergence of the approximation error, i.e., the distance between $f(\boldsymbol{x})$ and $f_N(\boldsymbol{x})$ on any compact set $K \subset I^d$ defined as:
\begin{equation}
\rho^2(f(\boldsymbol{x}), f_N(\boldsymbol{x}))=\mathbb{E}\biggl[\int_K (f(\boldsymbol{x})- f_N(\boldsymbol{x}))^2 d \boldsymbol{x}\biggr]
\end{equation}
is of the order of $(C/\sqrt{N})$, where $C=C(f,\phi_i, \beta, \Omega, \alpha, d)$; $\mathbb{E}$ denotes expectation over a probabilistic space $S(\mathcal{U}, \alpha)$, $\beta$ is the support of $\phi_i$ in $\prod_{i=1}^{d}(-\beta \boldsymbol{w}_i, \beta \boldsymbol{w}_i)$. 
\end{theorem}
Similar results for one-layer schemes have been also reported in other studies (see e.g. \cite{rahimi2008weighted,gorban2016approximation}). Rahimi and Recht \cite{rahimi2008weighted} proved the following Theorem:
\begin{theorem}
Let $p$ be a probability distribution drawn i.i.d. on $\mathcal{U}$ and consider the basis functions $\phi(x,\alpha):\mathcal{X}\times \mathcal{U} \rightarrow R$ that satisfy $sup_{x,\alpha}\phi(x,\alpha)\leq 1$. Define the set of functions
\begin{equation}\label{infsolution}
\mathcal{S}_p\equiv \biggl\{f(x)=\int w(\alpha)\phi(x;\alpha) d\alpha, |w(\alpha)| \le C p(\alpha)\biggr\},
\end{equation}
and let $\mu$ be a measure on $\mathcal{X}$. Then, if one takes a function $f$ in $S_p$, and $N$ values $\alpha_1,\alpha_2,\dots,\alpha_N$ of the shape parameter $\alpha$ are drawn i.i.d. from $p$, for any $\delta >0$ with probability at least $1-\delta $ over $\alpha_1,\alpha_2,\dots,\alpha_N$, there exists a function defined as $f_N=\sum_{j=1}^N w_j \phi_j(x;\alpha_j), |\alpha_j|\leq \dfrac{C}{N}$ so that
\begin{equation}\label{uniformconv}
    \int_{\mathcal{X}}(f(\boldsymbol{x}), f_N(\boldsymbol{x}))^2d\mu(x)\leq\dfrac{C}{\sqrt{N}}(1+\sqrt{2\log \dfrac{1}{\delta}}).
\end{equation}
\end{theorem}
For the so-called Extreme-Learning machines (ELMs), which similarly to RVFLNs are FNNs with randomly assigned internal weights and biases of the hidden layers, Huang et al. \cite{huang2006extreme,huang2014insight} has proved the following theorem:
\begin{theorem}
\label{thm:ELM_universal}
For any set of $n$ input-output pairs $(\boldsymbol{x}_i, \boldsymbol{y}_i), i=1,2,..N$ ($\boldsymbol{x}_i\in \R^d, \boldsymbol{y}_i\in \R^k$), the projection matrix  $\boldsymbol{\Phi}_N \in \mathbb{R}^{N \times N}$ in Eq.(\ref{eq:lin_map}) is invertible and
\begin{equation}
   \lVert {\boldsymbol{\Phi}_N \boldsymbol{W}^{o}} - \boldsymbol{Y}_N\rVert = 0,
\end{equation}
with probability 1 for $N=n$, $\boldsymbol{W}$, $\boldsymbol{b}$ randomly chosen from any probability distribution, and an activation function that is infinitely differentiable. 
\end{theorem}

\subsection{The Proposed Physics-Informed RPNN-based method for the solution of ODES and index-1 DAEs\label{sec:our_method}}
Here, we propose a physics-informed machine learning method based on random projections for the solution of IVPs of a system given by Eq.(\ref{eq:DAE})/(\ref{eq:semiexplicitDAE}) in $n$ collocation points in an interval, say $[t_0 \quad t_{end}]$. According to the previous notation, for this problem we have $d=k=1$. Thus, Eq.(\ref{eq:lin_map}) reads:
\begin{equation} \label{eq:lin_map3}
    \boldsymbol{Y}_N = \boldsymbol{\Phi}_N {\boldsymbol{w}^{o}}, \quad \boldsymbol{Y}_N \in \mathbb{R}^{n}, \quad \boldsymbol{\Phi}_N \in\mathbb{R}^{n\times N}, \quad \boldsymbol{w}^{o} \in\mathbb{R}^{N}
\end{equation}
Thus, the output of the RPNN is spanned by the range $\mathcal{R}(\boldsymbol{\Phi})$, i.e. the column vectors of $\boldsymbol{\Phi}_N$, say $\boldsymbol{\phi}_i \in \R^n$. Hence, the output of the RPNN can be written as:
\begin{equation} \label{eq:lin_map4}
    \boldsymbol{Y}_N = \sum_{i=1}^{N}w_{i}^{o} \boldsymbol{\phi}_i
\end{equation}
For an IVP of $m$ variables, we construct $m$ such PIRPNNs. We denote by $\boldsymbol{\Psi}(t,\boldsymbol{W},\boldsymbol{W}^{o},\boldsymbol{P})$ the set of functions $\Psi_{Ni}(t,\bw^{o}_i,\boldsymbol{p}_i), i=1, 2, \dots ,m$ that approximate the solution profile $u_i$ at time $t$, defined as:
\begin{equation}
    \Psi_{Ni}(t,\boldsymbol{w}_i,\bw^{o}_i,\boldsymbol{p}_i) = z_i+(t-t_0) {\bw^{o}_{i}}^{T}\boldsymbol{\Phi}_{Ni}(t,\boldsymbol{w}_i,\boldsymbol{p}_i),
    \label{eq:trsolsys}
\end{equation}
where  $\boldsymbol{\Phi}_{Ni}(t,\boldsymbol{w}_i,\boldsymbol{p}_i) \in \R^{N}$ is the column vector containing the values of the $N$ basis functions at time $t$ as shaped by $\bw_i$ and $\boldsymbol{p}_i$ containing the values of the parameters of the $N$ basis functions and  $\bw^{o}_i=[w^{o}_{1i} \; w^{o}_{2i} \; \ldots \; w^{o}_{Ni}]^T \in \R^{N}$ is the vector containing the values of the output weights of the $i$-th PIRPNN network. Note that the above set of functions are continuous functions of $t$ and satisfy explicitly the initial conditions.

For index-1 DAEs, with say $M_{ij}=0, \forall i\ge l,$ $j=1,2,\dots, m$, or in the semiexplicit form of (\ref{eq:semiexplicitDAE}), there are no explicit initial conditions $z_i$ for the variables $u_i,$ $i=l,l+1,\dots, m$, or the variables $\boldsymbol{v}$ in (\ref{eq:semiexplicitDAE}): these values have to satisfy the constraints $f_i(t,\boldsymbol{u})=0, i\ge l,$ (equivalently $\boldsymbol{0}=\boldsymbol{g}(t,\boldsymbol{u},\boldsymbol{v})$) $\forall t$, thus one has to start with consistent initial conditions.  Assuming that the corresponding Jacobian matrix of the $f_i(t,\boldsymbol{u})=0, i=l,l+1,\dots, m$ with respect to $u_i$, and for the semiexplicit form (\ref{eq:semiexplicitDAE}), $\nabla_{\boldsymbol{v}} \boldsymbol{g}$, is not singular, one has to solve initially at $t=0$, using for example Newton-Raphson iterations, the above nonlinear system of $m-l$ algebraic equations in order to find a consistent set of initial values.
Then, one can write the approximation functions of the $u_i,  i=k,k+1,\dots, m$/$\boldsymbol{v}$ variables as in Eq.\eqref{eq:trsolsys}.

With $n$ collocation points in $[t_0 \quad t_{end}]$, by fixing the values of the interval weights $\boldsymbol{w}_i$ and the shape parameters $\boldsymbol{p}_i$, the loss function that we seek to minimize with respect to the unknown coefficients $\bw^{o}_i$ is given by:
\begin{equation}
\mathcal{L}(\boldsymbol{W}^o)=\sum_{j=1}^{n} 
\left(\boldsymbol{M}\dfrac{d\boldsymbol{\Psi}}{dt}(t_j,\boldsymbol{W},\boldsymbol{W}^{o},\boldsymbol{P})- \boldsymbol{f}(t_j,\boldsymbol{\Psi}(t_j,\boldsymbol{W}, \boldsymbol{W}^{o},\boldsymbol{P})) \right)^2,
\label{costfun}
\end{equation}

%
When the system of ODEs/DAEs results from the spatial discretization of PDEs, we assume that the corresponding boundary conditions have been appropriately incorporated into the resulting algebraic equations explicitly or otherwise can be added in the loss function as algebraic constraints. 

Based on the above notation, we construct $m$ PIRPNNs, taking, for each network $\mathcal{N}_i$, $N$ Gaussian RBFs which, for $j=1,\dots,N,$ $i=1,\dots,m$, are given by:
\begin{equation}
 g_{ji}(t,w_{ji},b_{ji},\alpha_{ji},c_{j}) =e^{ -\alpha_{ji}( w_{ji}t+b_{ji}-c_{j})^2}.
\label{eq:gaussian_RBF}
\end{equation}
The values of the (hyper) parameters, namely $w_{ji}$, $b_{ji}, c_{j}$ are set as:
\[w_{ji}=1, \quad b_{ij}=0, \quad  c_{j}=t_j= t_0 + (j-1)\frac{t_{end}-t_0}{N-1},\]
while the values of the shape parameters $\alpha_{ji}>0$ are sampled from an appropriately chosen uniform distribution.
Under the above assumptions, the time derivative of $\Psi_{Ni}$ is given by:
\begin{equation}
    \dfrac{d \Psi_{Ni}}{d t} =
    \sum_{j=1}^N w^{o}_{ji} e^{ -\alpha_{ji}(t-t_{j})^2}-
    2(t-t_0)\sum_{j=1}^{N} \alpha_{ji} w^{o}_{ji} (t-t_{j}) e^{ -\alpha_{ji}(t-t_{j})^2}.
\label{eq:dirGaussian_RBF}
\end{equation}  

\subsubsection{Approximation with the PIRPNN}\label{sec:convergence}
At this point, we note that in Theorems \ref{thm:SLFNN_universal} and \ref{thm:ELM_universal}, the universal approximation property is based on the random base expansion given by  Eq.(\ref{eq:randomexpansion}), while in our case, we have a slightly different expansion given by Eq.(\ref{eq:trsolsys}). Here, we show that the PIRPNN given by Eq.(\ref{eq:trsolsys}) is a universal approximator of the solution $\boldsymbol{u}$ of the ODEs in canonical form or of the index-1 DAEs in the semiexplicit form (\ref{eq:semiexplicitDAE}).
\begin{prop}
For the IVP problem (\ref{eq:DAE}) in the canonical form or in the semiexplicit form (\ref{eq:semiexplicitDAE}), the PIRPNN solution $\Psi_{Ni}$ given by Eq.(\ref{eq:trsolsys}) with $N$ Gaussian basis functions defined by Eq.(\ref{eq:gaussian_RBF}), and the values of the shape parameters $\alpha_{ji}$ drawn i.i.d. from a uniform distribution converges uniformly to the solution profile $\boldsymbol{u}(t)$ in a closed time interval $[t_0 \quad t_{end}]$.
\end{prop}
\begin{proof}
Assuming that the system in Eq.\eqref{eq:DAE} can be written in the canonical form, the Picard-Lindel\"of Theorem \cite{collins1988differential}  holds true, then it exists a unique continuously differentiable function defined on a closed time interval $[t_0 \quad t_{end}]$ given by:
\begin{equation}\label{solintegral}
 u_i(t)=z_i+\int_{t_0}^{t}f_i(s,\boldsymbol{u}(s))ds, \quad i=1,2,\dots m
\end{equation}
From Eq.(\ref{eq:trsolsys}) we have:
\begin{equation}\label{approxsolintegral}
\Psi_{Ni}(t)=z_i+(t-t_0)\sum_{j=1}^{N}w_{j}^{o}e^{ -\alpha_{j}( t-t_{j})^2}.
\end{equation}
By the change of variables, $\tau=\dfrac{s-t_0}{t-t_0}$, the integral in Eq.(\ref{solintegral}) becomes
\begin{equation}\label{normalizedint}
\int_{t_0}^{t}f_i(s,\boldsymbol{u}(s))ds=(t-t_0)\int_{0}^{1}f_i(\tau(t-t_0)+t_0,\boldsymbol{u}(\tau(t-t_0)+t_0))d\tau.
\end{equation}
Hence, by Eqs.(\ref{solintegral}),(\ref{approxsolintegral}),(\ref{normalizedint}), we have:
\begin{equation}
    I_n(t)\equiv\int_{0}^{1}f_i(\tau(t-t_0)+t_0,\boldsymbol{u}(\tau(t-t_0)+t_0))d\tau\approx \sum_{j=1}^{N}w_{j}^{o}e^{ -\alpha_{j}( t-t_{j})^2}.
\end{equation}
Thus, in fact, upon convergence, the PIRPNN provides an approximation of the normalized integral. By Theorem 3, we have that in the interval $[t_0 \quad t_{end}]$, the PIRPNN with the shape parameter of the Gaussian kernel drawn i.i.d. from a uniform distribution provides, a uniform approximation of the integral in  Eq.(\ref{solintegral}) in terms of a Monte Carlo integration method as also described in \cite{igelnik1995stochastic}. 

Hence, as the initial conditions are explicitly satisfied by $\Psi_{Ni}(t)$, we have from Eq.({\ref{uniformconv}}) an upper bound for the uniform approximation of the solution profile $u_i$ with probability $1-\delta$.

For index-1 DAEs in the semiexplicit form of (\ref{eq:semiexplicitDAE}), by the implicit function theorem, we have that the DAE system is \textit{in principle} equivalent with the ODE system in the canonical form:
\begin{equation}
\dfrac{d\boldsymbol{u}(t)}{dt}  =  \boldsymbol{f}(t, \boldsymbol{u}(t),\boldsymbol{\mathcal{H}}(t,\boldsymbol{u})),
\end{equation}
where $\boldsymbol{v}(t)=\boldsymbol{\mathcal{H}}(t,\boldsymbol{u}(t))$ is the unique solution of  $\boldsymbol{0}=\boldsymbol{g}(t,\boldsymbol{u}(t),\boldsymbol{v}(t))$. Hence, in that case, the proof of convergence reduces to the one above for the ODE system in the canonical form.
\end{proof}
\subsubsection{Computation of the unknown weights}
For $n$ collocation points, the outputs of each network $\mathcal{N}_i\equiv \mathcal{N}_i(t_1,t_2, \dots t_n, \bw^{o}_{i}, \boldsymbol{p}_i) \in \R^{n}$, $i=1,2,\dots m$, are given by:
\begin{equation}
\begin{aligned}
    \mathcal{N}_i = \boldsymbol{R}_i \bw^{o}_{i}, \quad \quad  \boldsymbol{R}_i\equiv \boldsymbol{R}_{i}(t_1,\ldots,t_n,\boldsymbol{p}_i)=
   \begin{bmatrix}
       g_{1i}(t_1) & \cdots & g_{Ni}(t_1)\\
       \vdots        & \vdots  & \vdots \\
       g_{1i}(t_n) & \cdots & g_{Ni}(t_n)
    \end{bmatrix}.
\end{aligned}
\label{eq:rpnn}
\end{equation}
%
The minimization of the loss function \eqref{costfun} is performed over the 
$nm$ nonlinear residuals $F_q$:
\begin{equation}
    F_q(\boldsymbol{W}^o)= \sum_{j=1}^m M_{ij}\dfrac{d\Psi_{Nj}}{dt_l} (t_l,\bw^{o}_j) - f_i(t_l, \Psi_{N1}(t_l,\bw^{o}_1), \ldots, \Psi_{Nm}(t_l,\bw^{o}_m)),
\label{eq:Fq}
\end{equation}
where $q=l + (i-1)n$, $i=1,2,\dots m$, $l=1,2,\dots n$, $\boldsymbol{W}^0 \in \R^{m N}$ is the column vector obtained by collecting the values of all $m$ vectors $\boldsymbol{w}^{o}_i \in \mathbb{R}^N$, $\boldsymbol{W}^o = [W^o_{k}] = [\boldsymbol{w}^{o}_1,\boldsymbol{w}^{o}_2
\dots,\boldsymbol{w}^{o}_m]^T,$ $k=1,2,...mN$.
%
Thus, the solution to the above non-linear least squares problem can be obtained, e.g. with Newton-type iterations such as Newton-Raphson, Quasi-Newton and Gauss-Newton methods (see e.g. \cite{de1981continuation}). For example, by setting $\boldsymbol{F}(\boldsymbol{W}^o) = [F_1(\boldsymbol{W}^o) \cdots F_q(\boldsymbol{W}^o) \cdots F_{(nm)}(\boldsymbol{W}^o)]^T$, the new update $d\boldsymbol{W}^{o(\nu)}$ at the $(\nu)$-th Gauss-Newton iteration is computed by the solution of the linearized system:
\begin{equation}
\begin{aligned}
 (\nabla_{\boldsymbol{W}^{o(\nu)}}\boldsymbol{F})^T \nabla_{\boldsymbol{W}^{o(\nu)}}\boldsymbol{F} \quad d\boldsymbol{W}^{o(\nu)}=-  (\nabla_{\boldsymbol{W}^{o(\nu)}}\boldsymbol{F})^T\boldsymbol{F}(\boldsymbol{W}^{o(\nu)}),
\end{aligned}
\label{gauss-newton}
\end{equation}
where $\nabla_{\boldsymbol{W}^{o(\nu)}}\boldsymbol{F} \in \R^{n m \times m N}$ is the Jacobian matrix of $\boldsymbol{F}$ with respect to $\boldsymbol{W}^{o (\nu)}$.
Note that the residuals depend on the derivatives $\frac{\partial \Psi_{Ni}(\cdot)}{\partial t_l}$ and the approximation functions $\Psi_{Ni}(\cdot)$, while the elements of the Jacobian matrix depend on the derivatives of $\frac{\partial \Psi_{Ni}(\cdot)}{\partial w^{o}_{ji}}$ as well as on the mixed derivatives $\frac{\partial^2 \Psi_{Ni}(\cdot)}{\partial t_l \partial w^{o}_{ji}}$. Based on \eqref{eq:dirGaussian_RBF}, the latter are given by
\begin{equation}
     \dfrac{\partial^2 \Psi_{Ni}}{\partial t_l \partial w^{o}_{ji}}=\dfrac{\partial \mathcal{N}_i(t_l,\bw^{o}_i, \boldsymbol{p}_i)}{\partial w^{o}_{ji}}-2(t_l-t_0) \alpha_{ji} (t_l+b_{ji}-c_{j}) e^{\left( -\alpha_{ji}(t_l+b_{ji}-c_{j})^2 \right)},
    \label{mixedGaussian_RBF}
\end{equation}  
Based on the above, the elements of the Jacobian matrix $\nabla_{\boldsymbol{W}^{o(\nu)}}\boldsymbol{F}$ can be computed analytically as:
\begin{equation}
    \dfrac{\partial F_p}{\partial W^o_q}=\dfrac{\sum_{j=1}^m M_{ij}\partial^2 \Psi_{Ni}(\cdot)}{\partial t_l\partial w^{o}_{jk}}-\dfrac{\partial f_i(t_l)}{\partial w^{o}_{jk}}
    \label{eq:jac_J}
\end{equation}
where, as before, $q=l+(i-1)n$ and $p=j+(k-1)h$.

However, in general, even when $N\ge n$, the Jacobian matrix is expected to be rank deficient, or nearly rank deficient, since some of the rows due to the random construction of the basis functions can be nearly linear dependent. Thus, the solution of the corresponding system, and depending on the size of the problem, can be solved using for example truncated SVD decomposition or QR factorization with regularization.
The truncated SVD decomposition scheme results to the Moore-Penrose pseudoinverse and the updates $\boldsymbol{dW}^{o(\nu)}$ are given by:
\begin{equation*}
    \boldsymbol{dW}^{o(\nu)} = -(\nabla_{\boldsymbol{W}^{o(\nu)}}\boldsymbol{F})^{\dagger} \boldsymbol{F}(\boldsymbol{W}^{o(\nu)}),  (\nabla_{\boldsymbol{W}^{o(\nu)}}\boldsymbol{F})^{\dagger} = \boldsymbol{V}_{\epsilon} \boldsymbol{\Sigma}_{\epsilon}^{\dagger} \boldsymbol{U}_{\epsilon}^T,
\end{equation*}
where  $\boldsymbol{\Sigma}_{\epsilon}^{\dagger}$ is the inverse of the diagonal matrix with singular values of $\nabla_{\boldsymbol{W}^o}\boldsymbol{F}$ above a certain value $\epsilon$, and $\boldsymbol{U}_{\epsilon}$, $\boldsymbol{V}_{\epsilon}$ are the matrices with columns the corresponding left and right eigenvectors, respectively. If we have already factorized $\nabla_{\boldsymbol{W}^{o(\tilde{\nu})}}\boldsymbol{F}$ at an iteration $\tilde{\nu}<\nu$, then in order to decrease the computational cost, one can proceed with a Quasi-Newton scheme, thus using the same pseudoinverse of the Jacobian for the next iterations until convergence, e.g., computing the SVD decomposition for only the first and second Newton-iterations and keep the same $(\nabla_{\boldsymbol{W}^{o(1)}}\boldsymbol{F})^{\dagger}$ for the next iterations.

For large-scale sparse Jacobian matrices, as those arising for example from the discretization of PDEs, one can solve the regularization problem using other methods such as sparse QR factorization.
Here, to account for the ill-posed Jacobian, we have used a sparse QR factorization with regularization as implemented by SuiteSparseQR, a multifrontal multithreaded sparse QR factorization package \cite{davis2009user,davis2011algorithm}. 
\subsubsection{Parsimonious construction of the PIRPNN}
\paragraph{\textbf{The variable step
size scheme.}} \label{Sec:adapt}

In order to deal with the presence of sharp gradients that resemble singularities at the time interval of interest, and stiffness, we propose an adaptive scheme for adjusting the step size of time integration as follows.
The full time interval of integration $[t_0 \quad t_{end}]$ is divided into sub-intervals, i.e., $[t_0 \quad t_{end}] = [t_0 \quad t_1] \cup [t_1 \quad t_2] \cup \dots,\cup[t_k \quad t_{k+1}]\cup \dots \cup [t_{end-1} \quad t_{end}]$, where $
t_1,t_2,\dots,t_k,\dots,t_{end-1}$ are determined in an adaptive way. This decomposition of the interval leads to the solution of consecutive IVPs.
In order to describe the variable step
size scheme, let us assume that we have solved the problem up to interval $[t_{k-1} \quad t_{k}]$, hence we have found $u_i^{(k-1)}$ and we are seeking $u_i^{(k)}$ in the current interval $[t_k \quad t_{k+1}]$ with a width of $\Delta t_k = t_{k+1}-t_k$.
Moreover suppose that the Quasi-Newton iterations after a certain number of iterations, say, $\nu\le \nu_{max}$ (here $\nu_{max}=5$) the resulting approximation error is \cite{gear1971numerical,soderlind2002automatic}:
\begin{equation}
    err=\biggl\|\dfrac{\bm{F}(\bm{W}^{o(\nu)})}{AbsTol+RelTol \cdot \frac{d \bm{\Psi}^{(k)}}{dt}}\biggr\|_{l^2},
    \label{eq:err_adapt_rpnn}
\end{equation}
where $RelTol$ is the tolerance relative to the size of each derivative component $\frac{d \Psi_{Ni}^{(k)}}{dt}$ and $AbsTol$ is a threshold tolerance.
Now, if $err<1$ the solution is accepted, otherwise the solution is rejected. 

In both cases, the size of the time interval will be updated according to the elementary local error control algorithm \cite{soderlind2002automatic}:
\begin{equation}
    \begin{aligned}
    &\Delta t_k^*=0.8 \gamma \cdot \Delta t_k, \qquad \text{with} \qquad \gamma=\biggl(\dfrac{1}{err}\biggr)^{\frac{1}{\nu+1}},
    \end{aligned}
    \label{eq:adapt_step_size}
\end{equation}
where $\gamma$ is a scaling factor and $0.8$ is a safe/conservative factor. Also $\Delta t_k^*$ should not be allowed to increase or decrease too much, so $\gamma$ should not be higher than a $\gamma_{max}$ (here set to $4$) and to be smaller than a $\gamma_{min}$ (here set to $0.1$).\\
Thus, if the Quasi-Newton scheme does not converge to a specific tolerance within a number of iterations $\nu_{max}$, then the interval width is decreased, thus redefining a new guess $t^*_{k+1}=t_{k}+\Delta t_k^*$ for $t_{k+1}$ and the Quasi-Newton scheme is repeated in the interval $[t_k \quad t^*_{k+1}]$. 

Finally, the choice of the first subinterval $[t_0 \quad t_1]$ was estimated using an automatic detection code for selecting the starting step as described in  \cite{gladwell1987automatic,hairer2000solving}.

Regarding the interplay between the number of collocation points $n$ and the time interval of integration as shaped by the variable step
size scheme, we note that for band-limited profile solutions, according to the Nyquist sampling theorem, 
we can reconstruct the solution if the sampling rate is at least $2\nu_c$ samples per second; for  $n$ sampling points in a time interval say $\Delta T$, the maximum allowed frequency for band-limited signals should not exceed $\nu_c= \dfrac{n}{2\Delta T}$. For any practical purposes, a sampling rate of at least 4 or even 10 times the critical frequency  is required in order to deal with phenomena such as aliasing and response spectrum shocks.
For time-limited signals, thus, assuming that the energy of the signal in the time domain is bounded i.e., $||f_{\infty}||{_{L^1}}<C_f$, then for all practical purposes, the critical frequency $\nu_c$ can be set as the frequency beyond which the amplitude of the Fourier transform can be considered negligible, i.e., lower than a certain threshold, say $\epsilon_{\nu} \ll 1$.
\\
\paragraph{\textbf{A continuation method for Newton's iterations.}}

For Newton-type schemes, the speed of the convergence to (or the divergence from) the solution depends on the choice of the initial guess, here for the unknown weights. Thus, we address a numerical natural continuation method for providing ``good'' initial guesses for the weights of the PIRPNN.
Suppose that we have already converged to the solution in the interval $[t_{k-1} \quad t_{k}]$; we want to provide for the next time interval $[t_{k} \quad t_{k+1}]$, as computed from the proposed adaptation scheme described above, a good initial guess for the weights of the PIRPNN. We state the following proposition.
\begin{prop}
Let $\boldsymbol{\Psi}(t_k)\in \mathbb{R}^m$ be the solution found with PIRPNN at the end of the time interval $[t_{k-1} \quad t_{k}]$. Then, an initial guess for the weights of the PIRPNN for the time interval $[t_{k} \quad t_{k+1}]$ is given by:
\begin{equation}
  \hat{\boldsymbol{W}}^o=\dfrac{d\boldsymbol{\Psi}(t_k)}{dt}\dfrac{\boldsymbol{\Phi}^T}{||\boldsymbol{\Phi}||^2_{l_2}},
  \label{eq:initial_guess_continuation}
\end{equation}
where $\hat{\boldsymbol{W}}^o\in \mathbb{R}^{m\times N}$ is the matrix with the initial guess of the output weights of the $m$ PIRPNNs and $\boldsymbol{\Phi}\in \mathbb{R}^N$ is the vector containing the values of the random basis functions in the interval $[t_{k} \quad t_{k+1}]$.
\end{prop}
\begin{proof}
At time $t_{k}$, a first-order estimation of the solution $\boldsymbol{\Psi}(t_{k+1})\in \mathbb{R}^m$ is given by:
\begin{equation}
   \hat{\boldsymbol{\Psi}}(t_{k+1}) =\boldsymbol{\Psi}(t_{k})+\dfrac{d\boldsymbol{\Psi}(t_k)}{dt}(t_{k+1}-t_{k}),
   \label{approx}
\end{equation}
where $\frac{d\boldsymbol{\Psi}(t_k)}{dt}$ is known. 
For the next time interval $[t_{k} \quad t_{k+1}]$, the approximation of the solution with the PIRPNNs reads:
\begin{equation}
   {\boldsymbol{\Psi}}(t_{k+1}) =\boldsymbol{\Psi}(t_{k})+(t_{k+1}-t_{k})\boldsymbol{W}^o\boldsymbol{\Phi}.
    \label{approxRPNN}
\end{equation}
By Eqs.(\ref{approx}), (\ref{approxRPNN}), we get:
\begin{equation}
  \hat{\boldsymbol{W}}^o\boldsymbol{\Phi}=\dfrac{d\boldsymbol{\Psi}(t_k)}{dt}.
 \label{approxRPNN2}
\end{equation}
It can be easily seen, that the economy SVD decomposition of the $N$-dimensional vector $\boldsymbol{\Phi}$ is given by:
\begin{equation}
\boldsymbol{\Phi}_{N\times 1}=\boldsymbol{U}_{N\times 1}\sigma_1,\quad \boldsymbol{U}_{N\times 1}=\dfrac{\boldsymbol{\Phi}_{N\times 1}}{||\boldsymbol{\Phi}||_{l_2}}, \sigma_1=||\boldsymbol{\Phi}||_{l_2}.
\end{equation}
Thus, the pseudo-inverse of $\boldsymbol{\Phi}$, is $\boldsymbol{\Phi}^{\dagger}=\dfrac{\boldsymbol{\Phi}^T}{||\boldsymbol{\Phi}||^2_{l_2}}$.
Hence, by Eq.(\ref{approxRPNN2}), an initial guess for the weights for the time interval $[t_{k} \quad t_{k+1}]$ is given by:
\begin{equation}
  \hat{\boldsymbol{W}}^o=\dfrac{d\boldsymbol{\Psi}(t_k)}{dt}\boldsymbol{\Phi}^{\dagger}=\dfrac{d\boldsymbol{\Psi}(t_k)}{dt}\dfrac{\boldsymbol{\Phi}^T}{||\boldsymbol{\Phi}||^2_{l_2}}.
    \label{approxRPNN3}
\end{equation}
\end{proof}

\paragraph{\textbf{Estimation of the interval of uniform distribution based on the variance/bias trade-off decomposition}}
Based on the choice of Gaussian basis functions, from Eqs.(\ref{eq:randomexpansion}), (\ref{eq:lin_map4}) one has to choose the number $N$ of the basis functions, and the interval of the uniform distribution say $\mathcal{U}=[0 \quad \alpha_{max}]$, from which the values of the shape parameters $\alpha_i$ are drawn.
The theorems of uniform convergence (sections \ref{sec:rand_proj_rbfnn} and \ref{sec:convergence}) consider the problem from the function approximation point of view. Regarding the approximation of a discrete set of data points, it has been proved 
that a set of $N$ randomly and independently constructed vectors in the hypercube $[0, 1]^n$ will be pair-wise $\epsilon$-orthogonal (i.e., $|\boldsymbol{\phi}_{i}^T \boldsymbol{\phi}_i|<\epsilon, \forall i,j, i\neq j$) with probability $1-\theta$, where $\theta$ is sufficiently small for  $N<\exp{(\epsilon^2n/4)}\sqrt{\theta}$ \cite{gorban2016approximation}.

Here, we construct $N$ random vectors by parsimoniously sampling the values of the shape parameter from an appropriately bounded uniform interval for minimizing the two sources of error approximation, i.e., the bias and the variance in order to get good generalization properties. In our scheme, these, over all possible values of the shape parameter $\alpha$ are given by (see Eq.\eqref{approxsolintegral}):
\begin{equation}
\begin{aligned}
    &\mathcal{B}=\mathbb{E}\biggl[\sum_{j=1}^{N}w_{j}^{o}e^{ -\alpha_{j}( t-t_{j})^2}\biggr]-I_n(t),\\ &Var=\mathbb{E}\biggl[(\sum_{j=1}^{N}w_{j}^{o}e^{ -\alpha_{j}( t-t_{j})^2})^2\biggr]-\mathbb{E}\biggl[\sum_{j=1}^{N}w_{j}^{o}e^{ -\alpha_{j}( t-t_{j})^2}\biggr]^2,
    \end{aligned}
\end{equation}
where $\mathbb{E}$ denotes expectation operator. Overfitting, i.e., a high variance occurs for large values of $\alpha$ and underfitting, i.e., a high bias occurs for small values of $\alpha$. \\
The expected value of the kernel $\phi(t-t_j;\alpha)=e^{-a(t-t_j)^2}, t\neq t_j$ with respect to the probability density function of the uniform distribution of the random variable $\alpha$ reads:
\begin{equation}
\begin{aligned}
    \mathbb{E}[\phi(t-t_j;\alpha)]=\int\displaylimits_{0}^{\alpha_{max}} f_{\alpha}(\alpha) e^{-\alpha(t-t_j)^2} d\alpha
   =\dfrac{1-e^{-\alpha_{max}(t-t_j)^2}}{\alpha_{max}(t-t_j)^2}.
    \end{aligned}
    \label{expectval}
\end{equation}
Similarly, the variance is given by:
\begin{equation}
\begin{aligned}
 Var[\phi(t-t_j;\alpha)]=\int\displaylimits_{e^{-\alpha_{max}(t-t_j)^2}}^{1}\phi^2 \dfrac{1}{\alpha_{max}(t-t_j)^2}\dfrac{1}{\phi} d\phi-\mathbb{E}[\phi]^2=\\
 \dfrac{1-e^{-2\alpha_{max}(t-t_j)^2}}{2\alpha_{max}(t-t_j)^2}-\mathbb{E}[\phi]^2.
  \end{aligned}
  \label{expectvar}
\end{equation}
At the limits of $t-t_j=dt=\dfrac{t_{end}-t_0}{N}$, 
from Eqs.(\ref{expectval}),(\ref{expectvar}), we get:
\begin{equation}\label{expectedminmax}
\begin{aligned}
&\mathbb{E}[\phi(dt;\alpha)]=\dfrac{N^2}{(t_{end}-t_0)^2}\dfrac{1-e^{-\alpha_{max}\dfrac{(t_{end}-t_0)^2}{N^2}}}{\alpha_{max}},\\ 
&Var[\phi(dt;\alpha)]=\dfrac{N^2}{(t_{end}-t_0)^2}\dfrac{1-e^{-2\alpha_{max}\dfrac{(t_{end}-t_0)^2}{N^2}}}{2\alpha_{max}} -\mathbb{E}[\phi(dt;\alpha)]^2, \\
\end{aligned}
\end{equation}
The above expressions suggest that $\alpha_{max}=\dfrac{ N^2}{c^2(N)}\dfrac{1}{(t_{end}-t_0)^2},$ $c(N)>0$. This leaves us with only one parameter $c=c(N)$ to be determined for the ``optimal" estimation of the upper bound of the uniform interval.
Here, the value of $c(N)$ is found  based on a reference solution, say $\boldsymbol{u}_{ref}$ resulting from the integration of a stiff problem, which solution profiles contain also sharp gradients. For our computations, we have chosen as reference solution the one resulting from the van der Pol (vdP) ODEs given by:
\begin{equation}
\dfrac{du_1}{dt}  = u_2, \quad \dfrac{du_2}{dt} = \mu(1-{u_1}^2)u_2-u_1,
\label{vdpsys0}
\end{equation}
for $\mu=100$ and $u_1(0)=2$, $u_2(0)=0$ as initial conditions; the time interval was set to $[0 \quad 3\mu]$, i.e., approximately three times the period of the relaxation oscillations, which for $\mu\gg 1$, is $T\approx \mu (3-2\ln{2})$. The particular choice of $\mu=100$ results to a stiff problem, thus containing very sharp gradients resembling approximately a discontinuity in the solution profile within the integration interval.
The reference solution was obtained using  the \texttt{ode15s} of the MATLAB ODE suite with absolute and relative error tolerances equal to 1e$-$14.\\
Here, in order to estimate the ``optimal" values of $c, N$ for the vdP problem, we computed the bias-variance loss function  using $60,000$ points in $[0 \quad 3\mu]$ fixing the number of collocation points to $n=20$.
In Figure \ref{fig:optimization_vdp}, we show the contour plots of the bias and variance errors and the computational times required for convergence, with both absolute tolerance and relative tolerance set to 1e$-$06, for the vdP model.
\begin{figure}
    \centering
    \subfigure[]{\includegraphics[width=0.45 \textwidth]{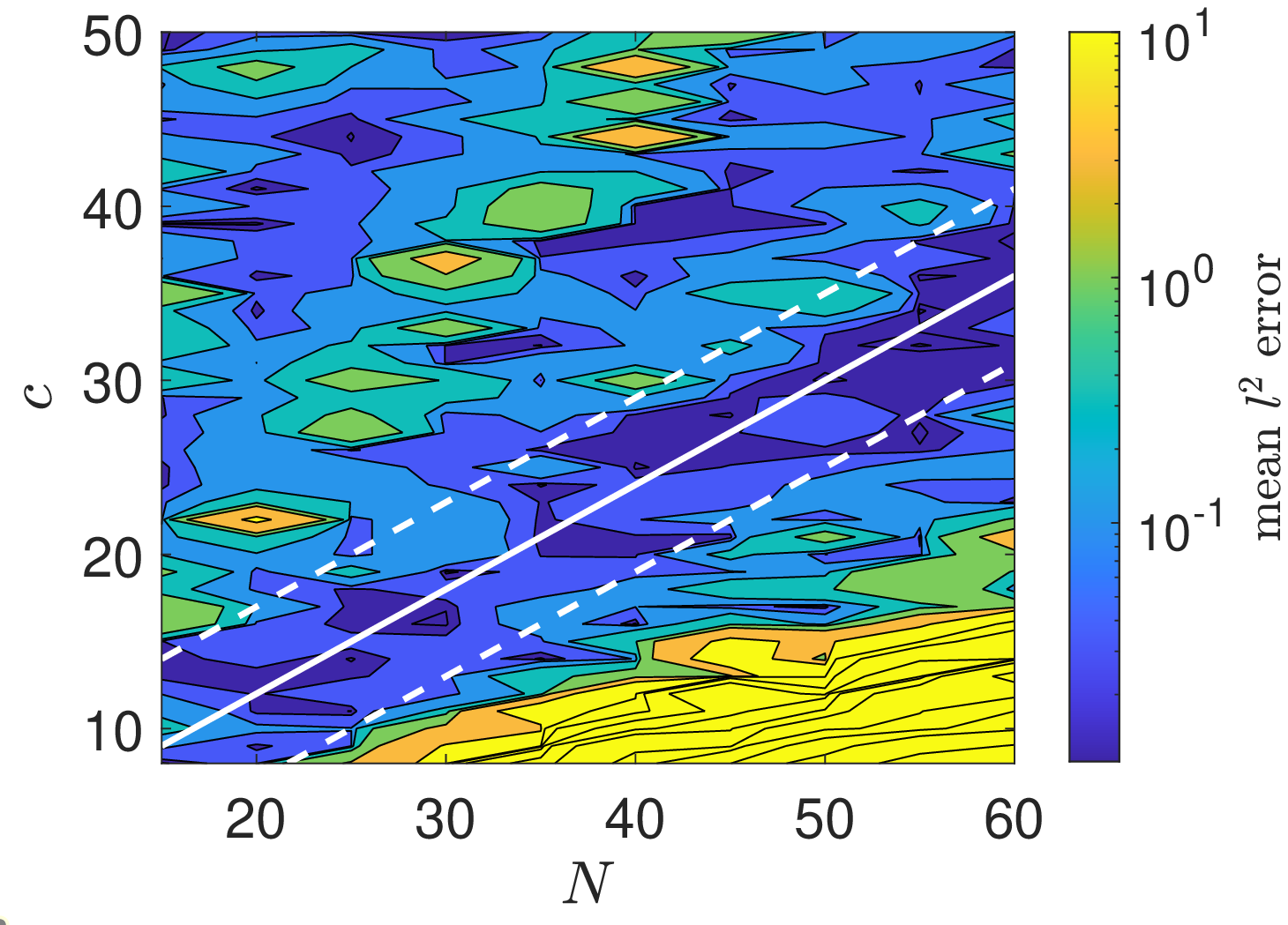}}
    \subfigure[]{\includegraphics[width=0.45 \textwidth]{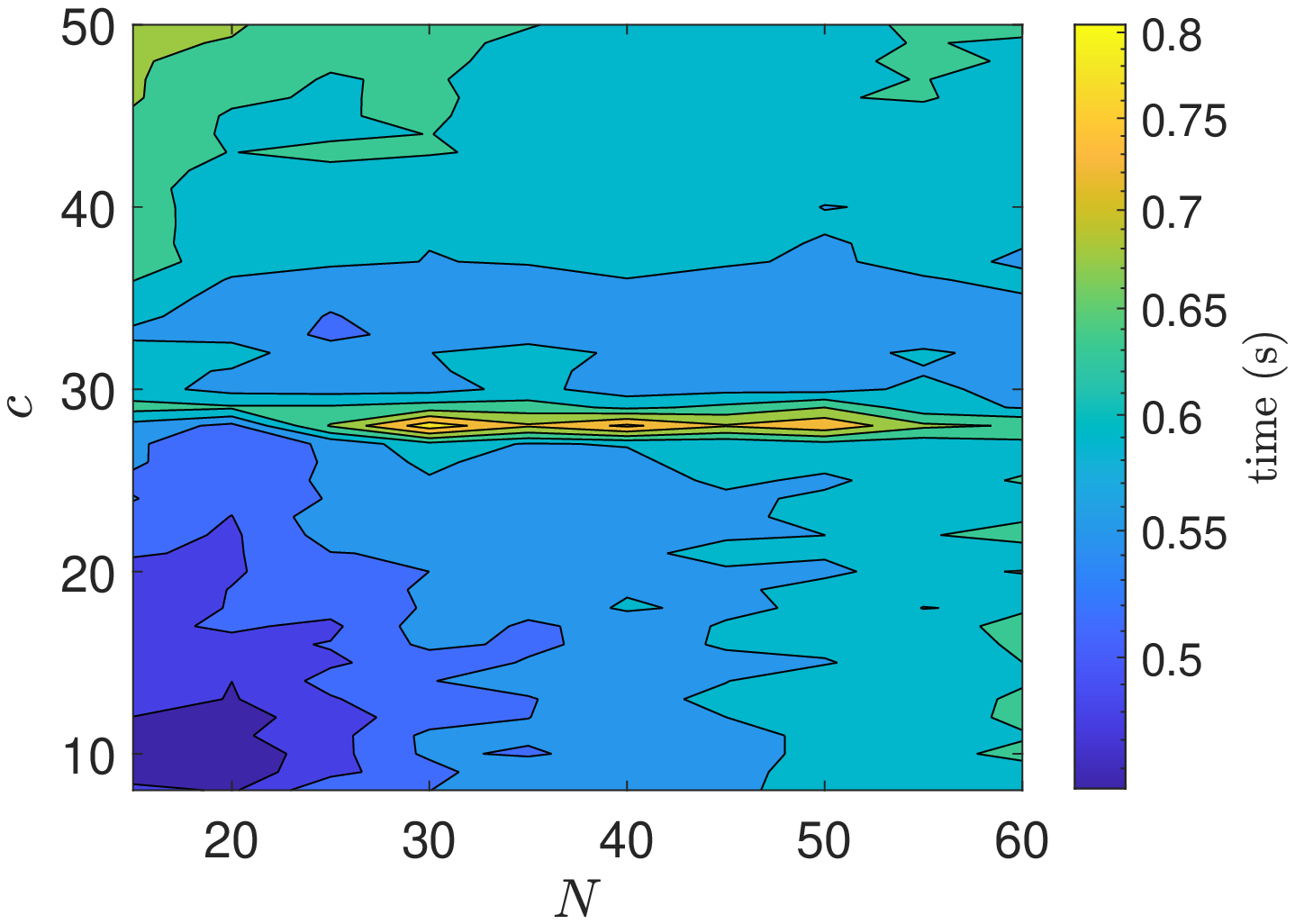}}
    
    \caption{The numerical solution of the van der Pol problem with the proposed PIRPNN with $\mu=100$ in the interval $[0 \quad 3\mu]$ with respect to $c$ and $N$ for $n=20$; $RelTol$ and $AbsTol$ were set to 1e$-$06; (a) Sum of the Bias and Variance of the approximation with respect to the reference solution as obtained with \texttt{ode15s} with $AbsTol$ and $RelTol$ set at 1e$-$14, (b) Computational Times (s).
    }
    \label{fig:optimization_vdp}
\end{figure}
As can be seen in Figure \ref{fig:optimization_vdp}(a) there is a linear band of ``optimal" combinations of $c,N$ within which the bias-variance error is for any practical purposes the same of the order of $\sim$ 1e$-$2. In particular, this band can be coarsely described by a linear relation of the form $c \approx \frac{3}{5}N,$ $N \ge 20$. Furthermore, as shown in Figure\ref{fig:optimization_vdp}(b), within this band of ``optimal'' combinations, the computational cost is minimum for $c\sim 12,$ $N=n=20$. 
Based on the above, the parsimonious of the variables $N$ and $c$ are $N=20,$ $c=12$ giving the best numerical accuracy and minimum computational cost. We note that the above parsimonious optimal values are set fixed once and for all for all the benchmark problems that we consider here.
\begin{algorithm}[htp!]
\small
\caption{Solving IVPs of ODEs and index-1 DAEs using PIRPNNs\label{alg:rpnn}}
\smallskip
\begin{algorithmic}[1] 
\Require $\displaystyle \boldsymbol{M}\dfrac{d\boldsymbol{u}}{dx}=\boldsymbol{f}(t,\boldsymbol{u})$ in $[0 \quad  t_{end}]$\; \Comment{THe ODE/ DAE system}
\Require $u_i(0)=z_i,$ for $i=1,2,\ldots,m$  \Comment{Set/Find Consistent Initial conditions}
\Require Abstol, RelTol, FlagSparse\;
\State $N \gets 20; \, n \gets 20$; \Comment{Set \# Neurons and \# Collocation Points}
\State $Der_i=0;$ \Comment{Initialization of Derivatives}
\State $\nu_{max}\gets 5$; \Comment{Set max \# iters for Newton's iterations}
\State $t^{*} \gets 0$; \, $\Delta t \gets set\_step\_size(\bm{f},\bm{z},AbsTol,RelTol)$; \Comment{set initial interval, see \cite{gladwell1987automatic,hairer2000solving}}
\Repeat
\State Select $t_l \in [t^{*} \quad t^{*}+\Delta t]; \, l=1,\ldots,n$ \Comment{Set \# Collocation Points}
\State $c_{j} \gets t^{*} + j \dfrac{\Delta t}{N-1}; \, j=1,\dots,N$\Comment{Set the Centers of Gaussian kernels}
\State $\displaystyle \alpha_{ji} \sim \mathcal{U} \left( 0,\dfrac{25}{9\Delta t^2} \right);$ \Comment{Uniformly Distributed Shape Parameters of Gaussian kernels}
\State $\bm{w}_i^o=Der_i \cdot \frac{\Phi^T}{||\Phi||^2_{l_2}};$ \Comment{Natural Continuation for an Initial Guess (see Eq.\eqref{eq:initial_guess_continuation})}
\State $\mathcal{N}_i(t,\boldsymbol{w}^0_i,\bp_i) \gets \sum_{j=1}^N w^{o}_{ji}\text{exp}\left(-\alpha_{ji}(t-c_j)^2\right)$\, \Comment{Construct the RPNN}
\State $\Psi_{Ni}(t,\boldsymbol{w}^0_i)=z_i+(t-t^{*})\mathcal{N}_i(t,\boldsymbol{w}^0_i,\bp_i)$\,\Comment{(see Eqs.\eqref{eq:gaussian_RBF},\eqref{eq:rpnn} and \eqref{eq:rpnn})}
\State $\nu \gets 0$; \Comment{Counter for Newton Iterations}
\Repeat
\For{$l=1,\dots,n$ and $i=1,\dots,m$}
\State $q \gets l + (i-1)n$ \Comment{Construct Residuals (see Eq.\eqref{eq:Fq})}
\State $\displaystyle F_q(\boldsymbol{W}^o) \gets \sum_{j=1}^m M_{ij}\dfrac{d\Psi_{Nj}}{dt_l} (t_l,\bw^{o}_j) - f_i(t_l, \Psi_{N1}(t_l,\bw^{o}_1), \ldots, \Psi_{Nm}(t_l,\bw^{o}_m))$;
\EndFor
\State Set $\boldsymbol{F}(\bm{W}^o)=[F_1(\bm{W}^o) , \, F_2(\bm{W}^o) , \, \dots \, , \, F_{(nm)}(\bm{W}^o)]^T$
\If{$\nu \le 1$} \Comment{Apply Quasi-Newton scheme}
\State  $J\gets\nabla_{\bm{W}^o} \boldsymbol{F}(\bm{W}^o)$\; \Comment{Compute Jacobian Matrix}
\If{FlagSparse is `False'}
\State $J^{\dagger} \gets V_{{\epsilon}} \Sigma_{{\epsilon}}^{\dagger} U_{{\epsilon}}^T$\; \Comment{Compute Pseudo-Inverse with SVD}
\EndIf
\EndIf
\If{FlagSparse is `False'}
\State $d\bm{W}^o \gets - J^{\dagger} \boldsymbol{F}(\bm{W}^o)$; 
\Else
\State $d\bm{W}^o \gets \texttt{spqr\_solve}(J,-F,'solution','min2norm')$; \Comment{SuiteSparseQR \cite{davis2009user,davis2011algorithm}}
\EndIf
\State $\bm{W}^o \gets \bm{W}^o+d\bm{W}^o$;

\State $err \gets \biggl\|\dfrac{\bm{F}(\bm{W}^{o(\nu)})}{AbsTol+RelTol \cdot \dfrac{d \bm{\Psi}^{(k)}}{dt}}\biggr\|_{l^2}$; \Comment{Compute Approximation Error \eqref{eq:err_adapt_rpnn}}
\State $\nu \gets \nu + 1$;
\Until{{$(err < 1)$ or $(\nu \ge \nu_{max})$}}
\If{$err<1$}
\State $t^{*} \gets t^{*}+\Delta t$;
\State $Der_i=\frac{d\Psi_{N_i}(t^*)}{dt};$ \Comment{Derivative at End Point}
\EndIf
\State $\Delta t \gets 0.8 \cdot min\biggl(0.1,max\biggl(4,\dfrac{1}{err}^{\frac{1}{\nu+1}}\biggr)\biggr) \cdot \Delta t$; \, \Comment{Adapt Step Size \eqref{eq:adapt_step_size}}
\Until{{$t^{*}=t_{end}$}}
\end{algorithmic}
\end{algorithm}

\subsubsection{The algorithm}
We summarize the proposed method in the pseudo-code shown in Algorithm~\ref{alg:rpnn}, where $\mathcal{U}(a,b)$ denotes the uniform random distribution in the interval $(a,b)$ and $FlagSparse$ is a logic variable for choosing either the SVD or the sparse QR factorization.

\section{Numerical Results\label{sec:results}}
We implemented the proposed algorithm~\ref{alg:rpnn} using MATLAB 2020b on an Intel Core i7-10750H CPU @ 2.60GHz with up to 3.9 GHz frequency and a memory of 16 GBs. In all our computations, we have used a fixed number of collocation points $n=20$ and number of basis functions $N=20$ with $c=12$ as discussed above. 
The Moore-Penrose pseudoinverse of $\nabla_{\boldsymbol{w}^o}\boldsymbol{F}$ was computed with the MATLAB built-in function \texttt{pinv}, with the default tolerance.

As stated, for assessing the performance of the proposed scheme, we considered seven benchnark problems of stiff ODEs and index-1 DAEs. In particular, we considered 
four index-1 DAEs, namely, the Robertson \cite{Robertson1966,shampine1999solving} model of chemical kinetics, a problem of mechanics \cite{shampine1999solving}, a power discharge control problem \cite{shampine1999solving}, the index-1 DAE chemical Akzo Nobel problem \cite{mazzia2012test,stortelder1998parameter}, and three stiff systems of ODEs, namely, the Belousov-Zhabotinsky chemical model \cite{belousov1951periodic,zhabotinsky1964periodical}, the Allen-Chan metastable PDE phase-field model \cite{allen1979microscopic,trefethen2000spectral} and the Kuramoto-Sivashinsky PDE \cite{sivashinsky1977nonlinear,trefethen2000spectral}.
For comparison purposes, we used three solvers of the MATLAB ODE suite \cite{shampine1997matlab}, namely \texttt{ode15s} and \texttt{ode23t} for DAEs and also \texttt{ode23s} for stiff ODEs, thus using the analytical Jacobian.  In order to estimate the numerical approximation error, we used as reference solution the one computed with \texttt{ode15s} setting the absolute and relative error tolerances to 1e$-$14. To this aim, we computed the $l^2$ and $l^{\infty}$ norms of the differences between the computed and the reference solutions, as well as the mean absolute error (MAE). Finally, we ran each solver 10 times and computed the median, maximum and minimum computational times. For the PIRPNN the reported accuracy is the mean value of 10 estimates. The initial time interval was selected according to the code in Chapter II.4 in \cite{hairer2000solving} (for further details see \cite{gladwell1987automatic}).

\subsection{Case Study 1: The Robertson index-1 DAEs}
The Robertson model describes the kinetics of an autocatalytic reaction \cite{Robertson1966}. This system of three DAEs is part of the benchmark problems considered in \cite{shampine1999solving}. The set of the reactions reads:
\begin{equation}
\begin{aligned}
&A\xrightarrow{k_1}B,\\
&B+C\xrightarrow{k_2}A+C,\\
&2B\xrightarrow{k_3}B+C,
\end{aligned}
\label{eq:strrober}
\end{equation}
where $A$, $B$, $C$ are chemical species and $k_1=0.04$, $k_2=10^4$ and $k_3=3\times10^7$ are reaction rate constants. Assuming that the total mass of the system is conserved, we have the following system of index-1 DAEs:
\begin{equation}
\begin{aligned}
&\dfrac{dA}{dt}  = -k_1A+k_2BC,\\
&\dfrac{dB}{dt}  = +k_1A-k_2BC-k_3B^2,\\
&  A+B+C=1,
\end{aligned}
\label{eq:robertson}
\end{equation}
where $A$, $B$ and $C$ denote the concentrations of $[A]$, $[B]$ and $[C]$, respectively. In our simulations, we set $A(0)=1$, $B(0)=0$ as initial conditions of the concentrations and a very large time interval $[0 \quad 4\cdot 10^{11}]$ as proposed in \cite{metivier2012strategies}.
\begin{table}[ht!]
\begin{center}
\caption{The Robertson index-1 DAEs \eqref{eq:robertson} \cite{shampine1999solving} in the time interval $[0 \quad 4\cdot 10^{11}]$. $l^{2}$, $l^{\infty}$ and mean absolute (MAE) approximation errors obtained with both absolute and relative tolerances set to 1e$-$03 and 1e$-$06.
\label{tab:robertson_accuracy}}
{\footnotesize
    \begin{tabular}{|l|l |l l l |l l l|}
    \hline
\multicolumn{2}{|c|}{} & \multicolumn{3}{c|}{$tol=$ 1e$-$03} & \multicolumn{3}{c|}{$tol=$ 1e$-$06} \\
\cline{3-8}
        \multicolumn{2}{|c|}{} & $l^2$ & $l^{\infty}$ & MAE & $l^2$ & $l^{\infty}$ & MAE \\
        \hline
        \rowcolor{LightCyan}
        & PIRPNN & 2.08e$-$01 & 1.37e$-$02 & 2.47e$-$04 & 2.07e$-$04 & 6.69e$-$06 & 4.29e$-$07\\
        $A$ & \texttt{ode23t}  & Inf & Inf & Inf & 3.41e$+$00 & 2.13e$-$01 & 3.10e$-$03\\
        & \texttt{ode15s} & 3.85e$+$09 & 1.92e$+$08 & 3.86e$+$06 & 2.84e$+$09 & 1.62e$+$08 & 2.15e$+$06 \\
        \hline
        \rowcolor{LightCyan}
        & PIRPNN & 1.79e$-$06 & 1.33e$-$07 & 1.79e$-$09 & 9.86e$-$10 & 5.48e$-$11 & 2.02e$-$12\\
        $B$ & \texttt{ode23t}  &  Inf & Inf & Inf & 3.41e$+$00 & 2.13e$-$01 & 3.10e$-$03\\
        & \texttt{ode15s} & 4.57e$-$03 & 2.42e$-$04 & 4.81e$-$06 & 1.56e$-$04 & 4.00e$-$06 & 1.75e$-$07 \\
        \hline
       \rowcolor{LightCyan}
        & PIRPNN & 2.08e$-$01 & 1.37e$-$02 & 2.47e$-$04 & 2.07e$-$04 & 6.69e$-$06 & 4.29e$-$07\\
        $C$ & \texttt{ode23t}  & Inf & Inf & Inf & 1.10e$-$03 & 1.75e$-$05 & 2.92e$-$06\\
        & \texttt{ode15s} & 3.85e$+$09 & 1.92e$+$08 & 3.86e$+$06 & 2.84e$+$09 & 1.62e$+$08 & 2.15e$+$06 \\
    \hline
    \end{tabular}
}
\end{center}
\end{table}
Table~\ref{tab:robertson_accuracy} summarizes the $l^2$, $l^{\infty}$ and mean absolute (MAE) approximation errors with respect to the reference solution in $40,000$ logarithmically spaced grid points in the interval $[0 \quad 4\cdot 10^{11}]$. As shown, for the given tolerances, the proposed method outperforms \texttt{ode15s} and \texttt{ode23t} in all metrics.
\begin{figure}[ht!]
    \centering
    \subfigure[]{
    \includegraphics[width=0.3 \textwidth]{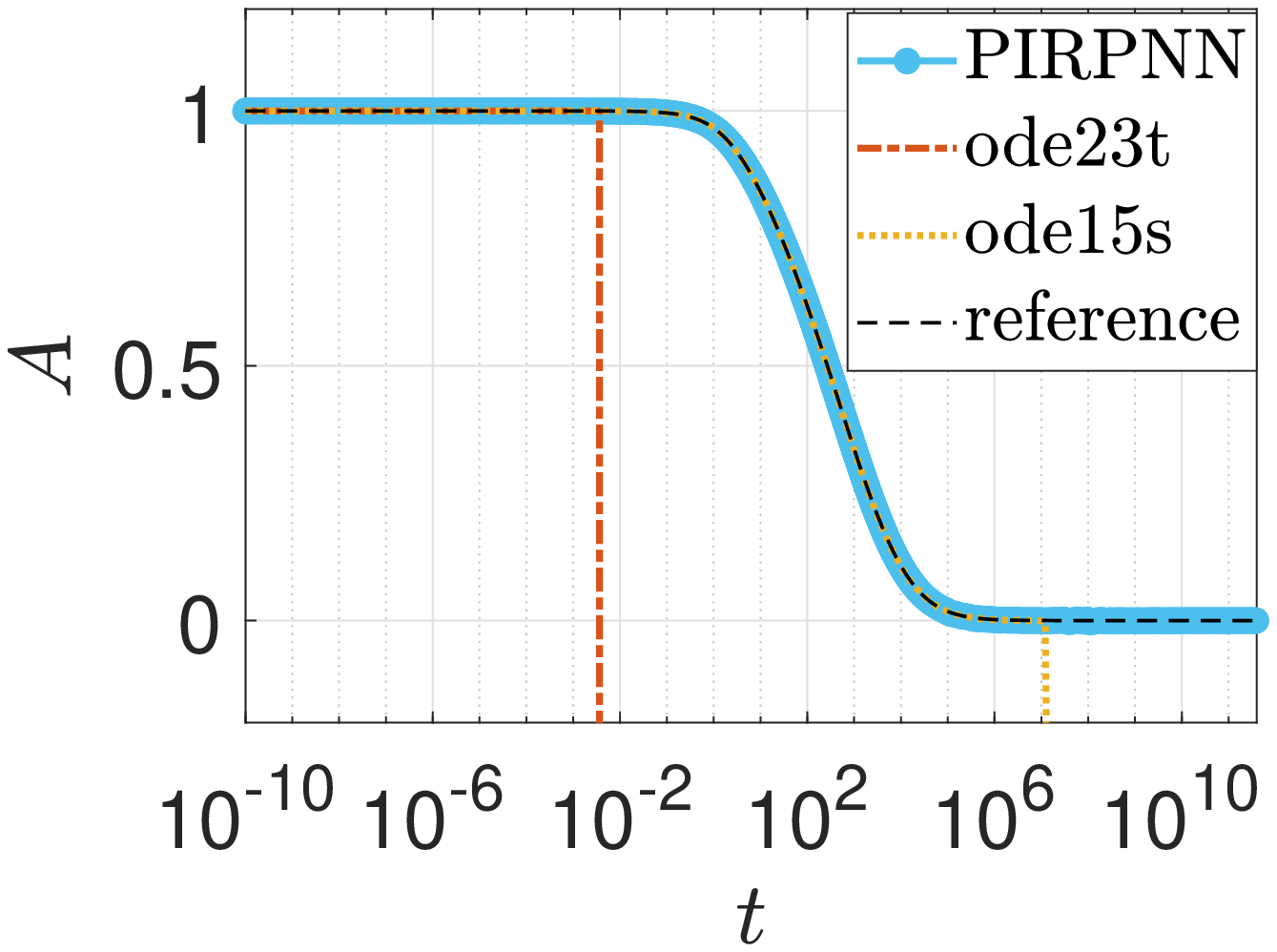}
    }
    \subfigure[]{
    \includegraphics[width=0.3 \textwidth]{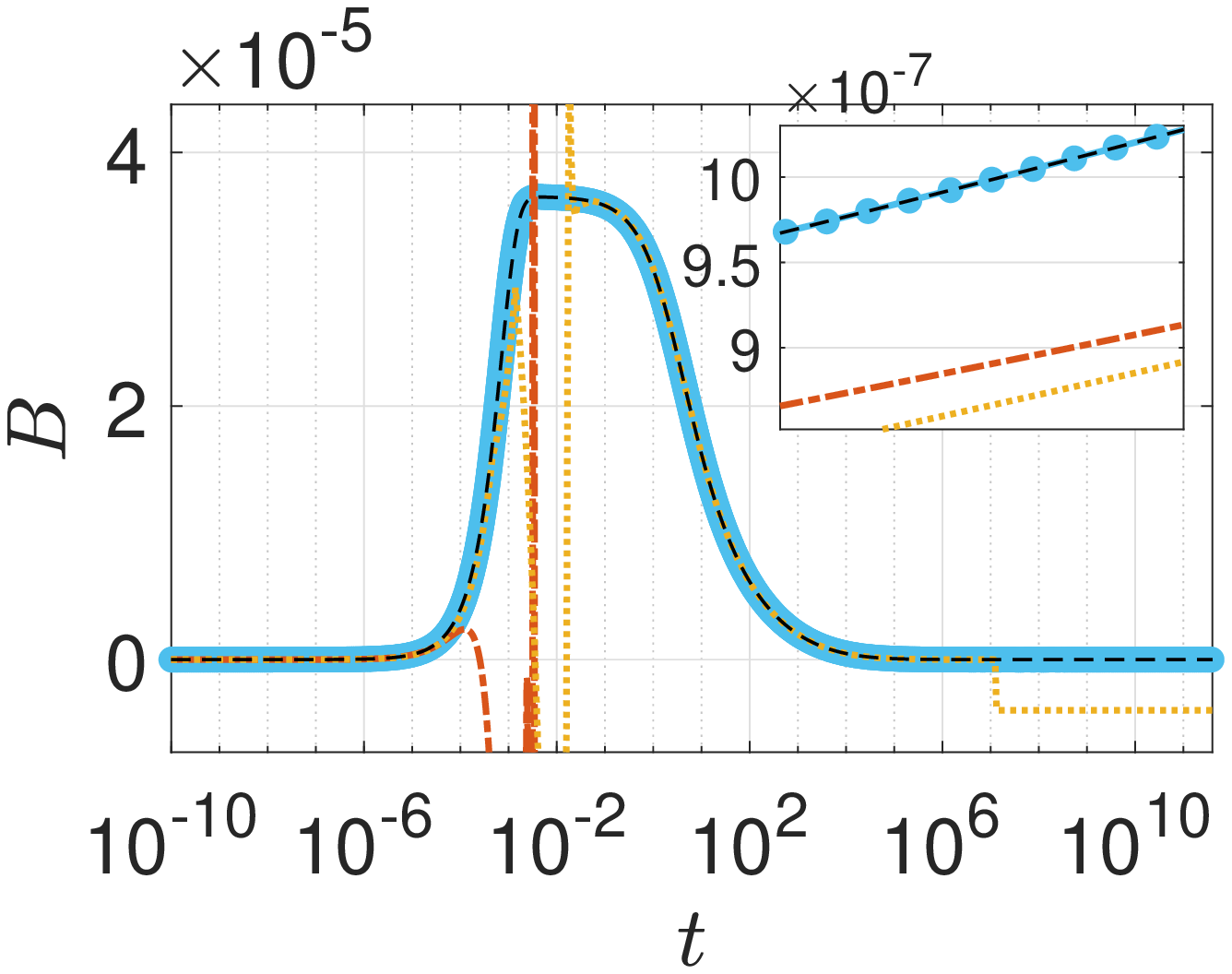}
    }
    \subfigure[]{
    \includegraphics[width=0.3 \textwidth]{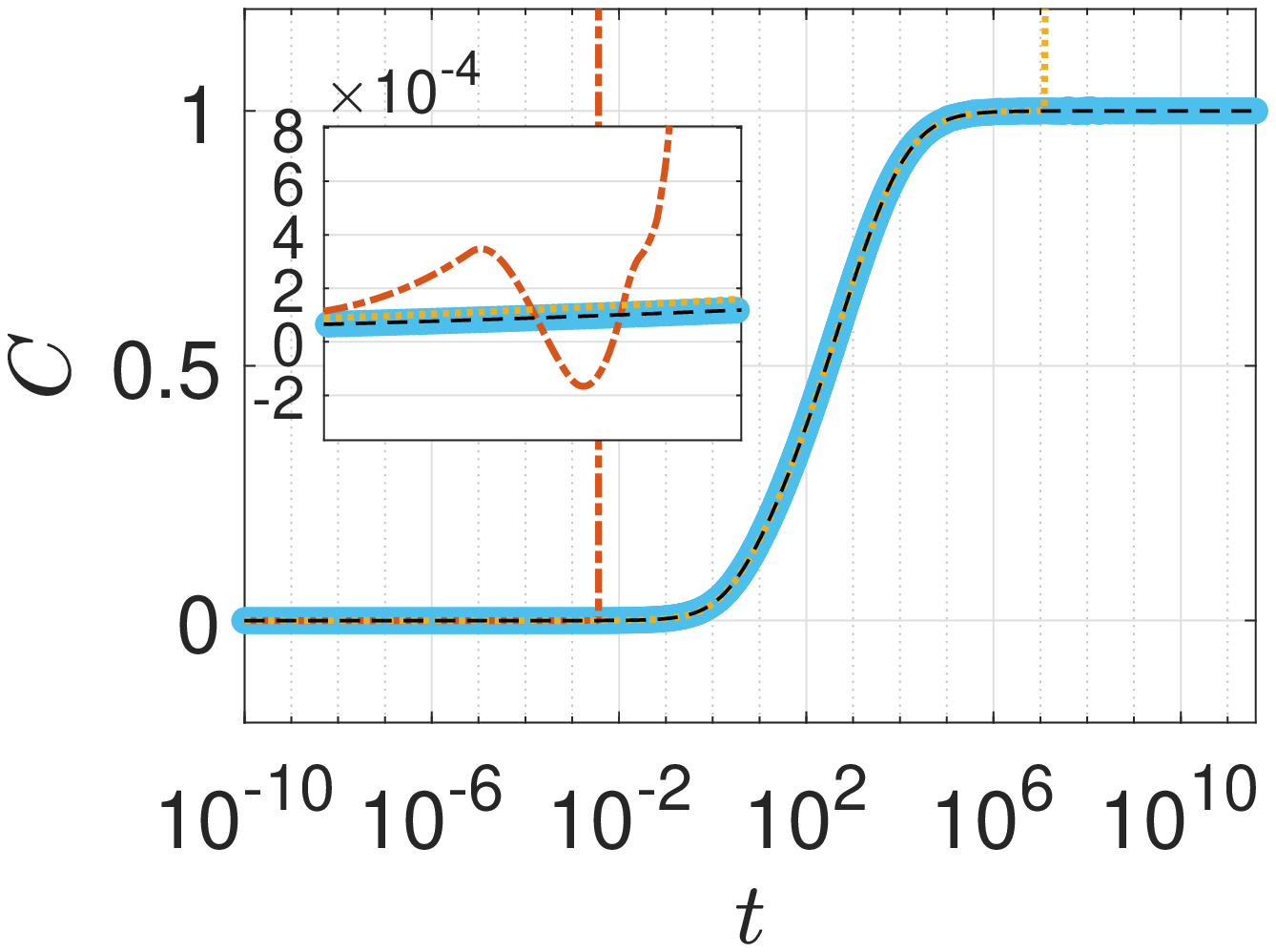}
    }
    \subfigure[]{
    \includegraphics[width=0.3 \textwidth]{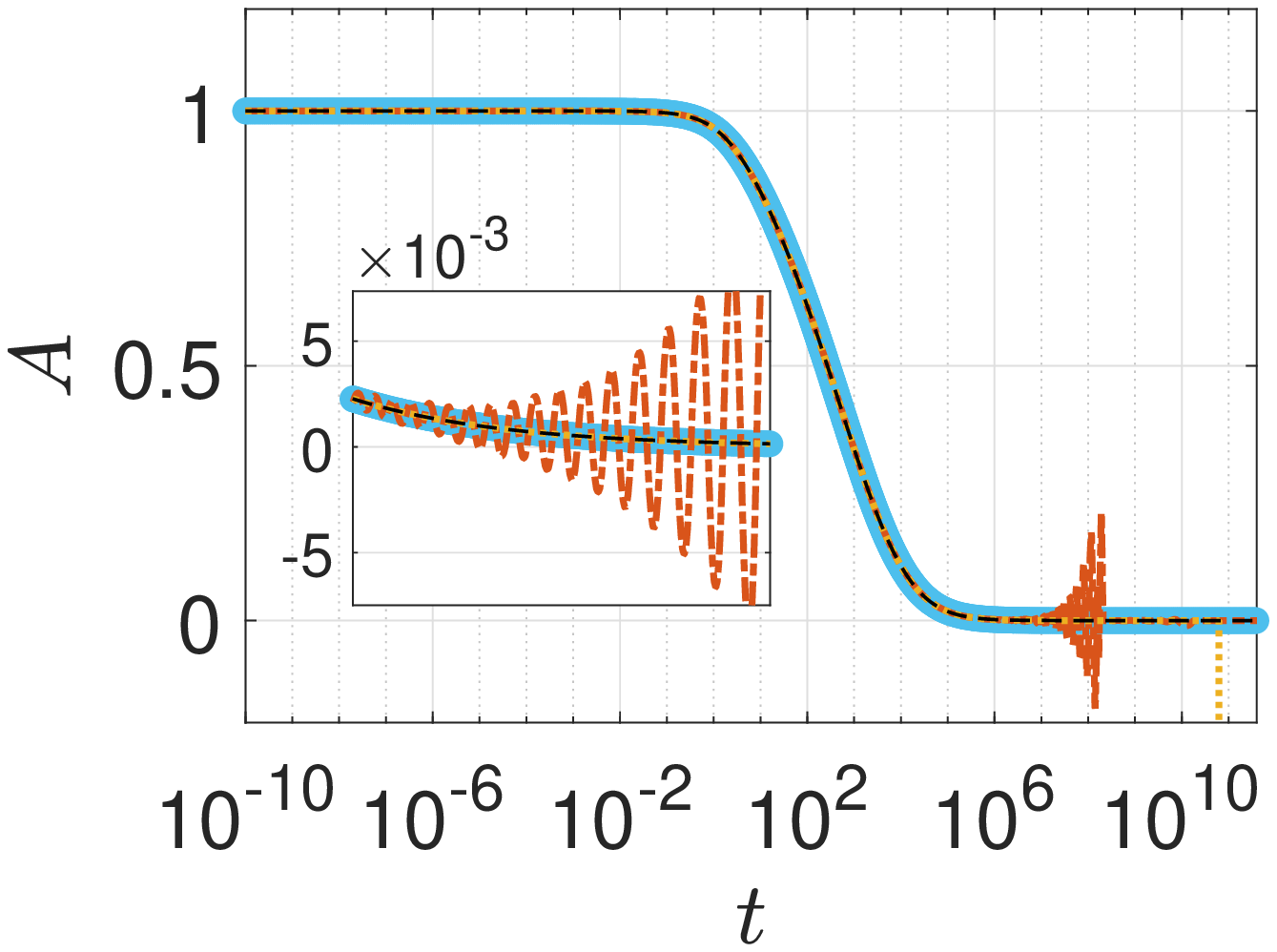}
    }
    \subfigure[]{
    \includegraphics[width=0.3 \textwidth]{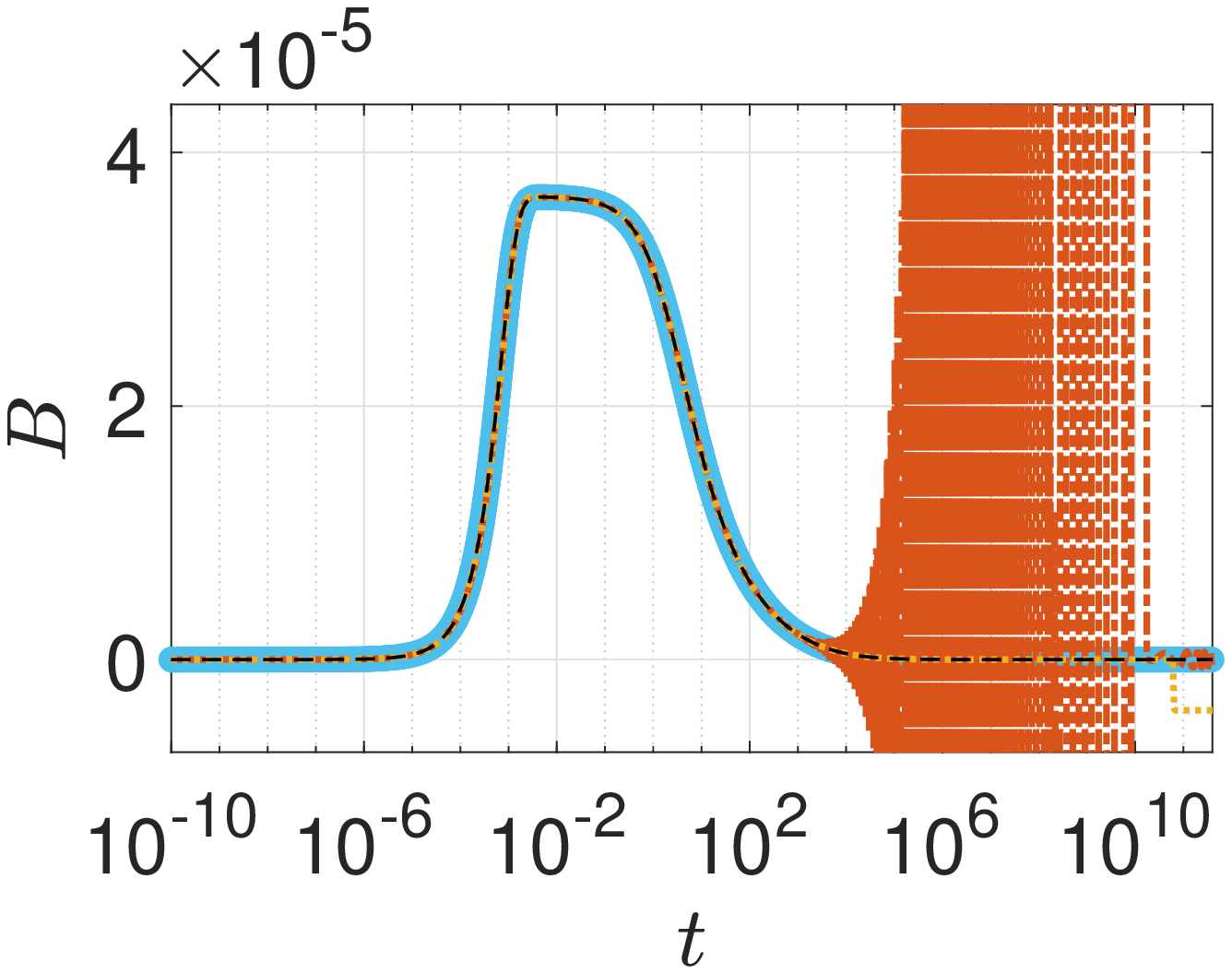}
    }
    \subfigure[]{
    \includegraphics[width=0.3 \textwidth]{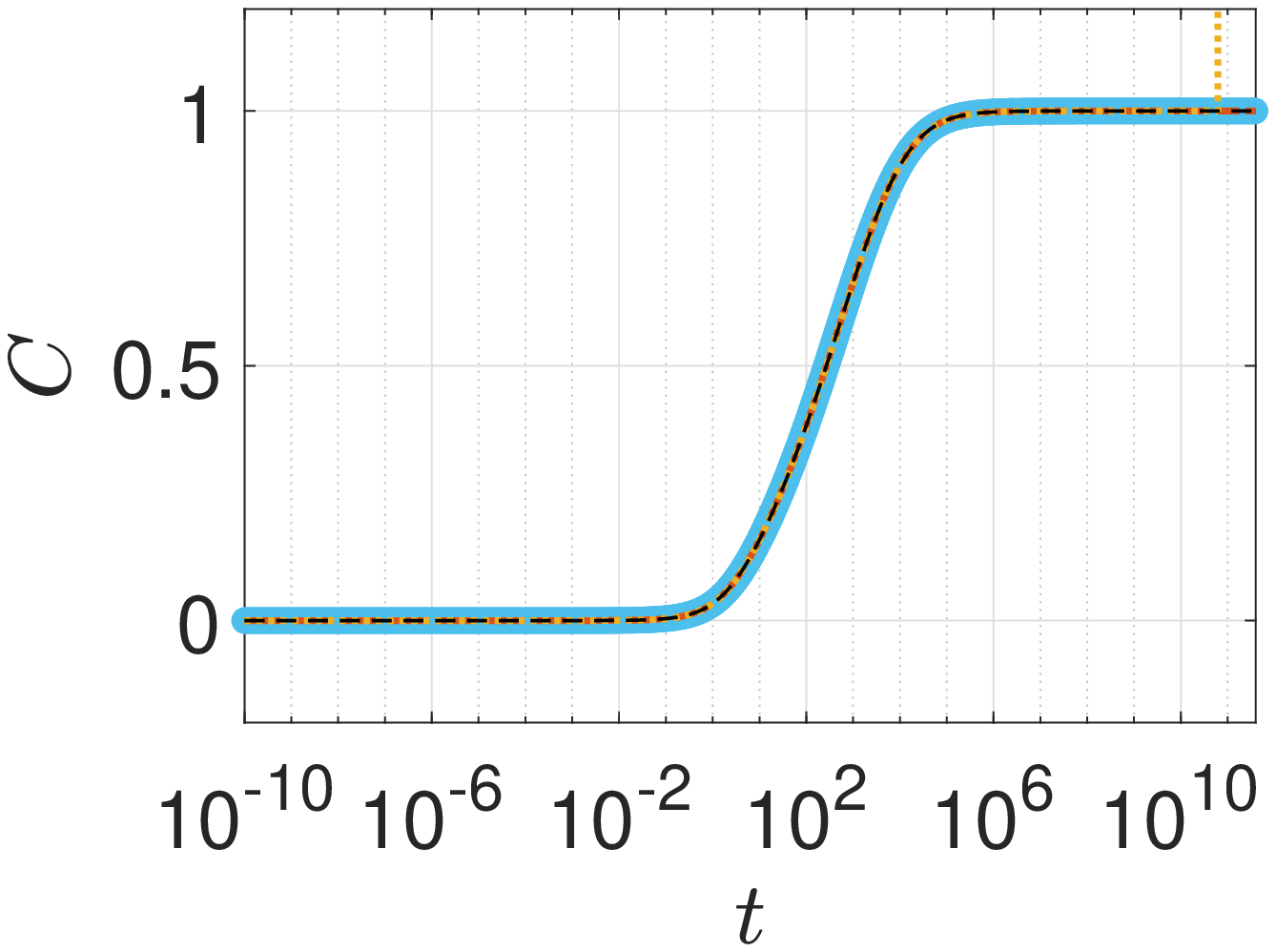}
    }
\caption{The Robertson index-1 DAEs \eqref{eq:robertson} \cite{shampine1999solving}. Approximate solutions computed in the interval $[0 \quad 4\cdot 10^{11}]$ with both absolute and relative tolerances set to~1e$-$03 ((a), (b) and (c)) and 1e$-$06 ((d), (e) and (f)). Insets depict zooms around the reference solution. \label{fig:robertson}}
\end{figure}
As it is shown in Figure~\ref{fig:robertson}, for tolerances 1e$-$03, the proposed scheme achieves more accurate solutions than \texttt{ode23t} and \texttt{ode15s}.
Actually, as shown, for the given tolerances, both \texttt{ode23t} and \texttt{ode15s} fail 
to approximate the solution up to the final time. 
\begin{table}[ht!]
\begin{center}
\caption{The Robertson index-1 DAEs \eqref{eq:robertson}. Computational times (s) (median, minimum and maximum over $10$ runs) and number of points required in the interval $[0 \quad 4\cdot 10^{11}]$ by the PIRPNN, \texttt{ode23t} and \texttt{ode15s} with both absolute and relative tolerances set to 1e$-$03 and 1e$-$06.}
{\footnotesize
\setlength{\tabcolsep}{3pt}
\begin{tabular}{|l|lllr|lllr|}
\hline
& \multicolumn{4}{c|}{$tol=$ 1e$-$03} & \multicolumn{4}{c|}{$tol=$ 1e$-$06} \\
\cline{2-9}
  & median & min & max & \multicolumn{1}{l|}{\# pts} & median & min & max & \multicolumn{1}{l|}{\# pts}\\
  \hline
  \rowcolor{LightCyan}
  PIRPNN & 2.66e$-$02 & 2.32e$-$02 & 3.97e$-$02 &  640 & 3.47e$-$02 & 3.15e$-$02 & 4.11e$-$02 & 682\\
  \texttt{ode23t}  & 7.09e$-$03 & 6.74e$-$03 & 1.16e$-$02 & 199 & 1.42e$-$02 & 1.38e$-$02 & 1.64e$-$02 & 613\\
  \texttt{ode15s} & 1.02e$-$02 & 9.69e$-$03 & 1.84e$-$02 & 271 & 2.11e$-$02 & 2.00e$-$02 & 2.18e$-$02 & 696\\
  reference & 1.15e$-$01 & 1.12e$-$01 & 1.22e$-$01 & 5194  & 1.15e$-$01 & 1.12e$-$01 & 1.22e$-$01 & 5194\\
  \hline
\end{tabular}
}
\end{center}
\label{tab:robertson_time_points}
\end{table}
In Table~\ref{tab:robertson_time_points}, we report the computational times and number of points required by each method, including the time for computing the reference solution.
As shown, the corresponding total number of points required by the proposed scheme is comparable with the ones required by \texttt{ode23t} and \texttt{ode15s} and significantly less than the number of points required by the reference solution.
Furthermore, the computational times of the proposed method are comparable with the ones required by the \texttt{ode23t} and \texttt{ode15s}, thus outperforming the ones required for computing the reference solution.

\subsection{Case Study 2: a mechanics non autonomous index-1 DAEs model}
This is a non autonomous system of five index-1 DAEs and it is part of the benchmark problems considered in \cite{shampine1999solving}. It describes the motion of a bead on a rotating needle subject to the forces of gravity, friction, and centrifugal force. The equations of motion are:
\begin{equation}
\begin{aligned}
    &\frac{du_1}{dt}=u_2, \\
    &\frac{du_2}{dt} = - 10u_2 + {sin}(t+pi/4)u_5, \\
    &\frac{du_3}{dt} =   u_4,\\ 
    &\frac{du_4}{dt} = - 10u_4 - {cos}(t+pi/4)u_5 + 1,\\
    & 0 = g_{pp} + 20g_p + 100g,
\end{aligned}
\label{eq:ex11}
\end{equation}
where 
\begin{equation}
\begin{aligned}
   &g ={cos}(t+pi/4)u_3 -{sin}(t+pi/4)u_1,\\
   &g_p    =  {cos}(t+pi/4)(u_4 - u_1) + {sin}(t+pi/4)( - u-2 - u_3),\\
   &g_{pp}   = {cos}(t+pi/4)(\frac{du_4}{dt} - \frac{du_1}{dt} - u_2 - u_3+ \text{sin}(t+pi/4)(- \frac{du_2}{dt} - \frac{du_3}{dt} - u_4 + u_1).
\end{aligned}
\label{eq:ex11b}
\end{equation}
The initial conditions are $u_1(0)=1,$ $u_2(0) = -6,$ $u_3(0) =  1,$ $u_4(0) = -6$ and a consistent initial condition for $u_5(0) = -10.60660171779820$ was found with Newton-Raphson at $t=0$, with a tolerance of 1e$-$16.
\begin{table}[ht]
\begin{center}
\caption{The mechanics problem of non autonomous index-1 DAEs \eqref{eq:ex11} in $[0\quad 15]$. $l^{2}$ $l^{\infty}$ and mean absolute (MAE) approximation errors obtained with both absolute and relative tolerances set to 1e$-$03 and 1e$-$06.\label{tab:ex11_accuracy}}
{\footnotesize
    \begin{tabular}{|l|l |l l l |l l l|}
    \hline
\multicolumn{2}{|c|}{} & \multicolumn{3}{c|}{$tol=$ 1e$-$03} & \multicolumn{3}{c|}{$tol=$ 1e$-$06} \\
\cline{3-8}
        \multicolumn{2}{|c|}{} & $l^2$ & $l^{\infty}$ & MAE & $l^2$ & $l^{\infty}$ & MAE \\
        \hline
        \rowcolor{LightCyan}
        & PIRPNN & 3.64e$-$05 & 7.77e$-$07 & 2.06e$-$07 & 9.70e$-$07 & 2.57e$-$08 & 4.71e$-$09\\
        $u_1$ & \texttt{ode23t}  & 6.26e$-$01 & 1.69e$-$02 & 3.34e$-$03 & 8.41e$-$03 & 2.19e$-$04 & 4.56e$-$05\\
        & \texttt{ode15s} & 6.38e$-$01 & 1.90e$-$02 & 3.12e$-$03 & 3.59e$-$04 & 8.89e$-$06 & 1.96e$-$06 \\
        \hline
        \rowcolor{LightCyan}
        & PIRPNN & 4.95e$-$05 & 1.42e$-$06 & 2.55e$-$07 & 3.08e$-$06 & 1.49e$-$07 & 1.22e$-$08\\
        $u_2$ & \texttt{ode23t}  &  4.38e$-$01 & 1.04e$-$02 & 2.25e$-$03 & 6.18e$-$03 & 1.35e$-$04 & 3.23e$-$05\\
        & \texttt{ode15s} & 6.57e$-$01 & 1.49e$-$02 & 3.60e$-$03 & 2.94e$-$04 & 1.40e$-$05 & 1.59e$-$06 \\
        \hline
       \rowcolor{LightCyan}
        & PIRPNN & 3.78e$-$05 & 8.05e$-$07 & 2.12e$-$07 & 9.90e$-$07 & 3.06e$-$08 & 4.78e$-$09\\
        $u_3$ & \texttt{ode23t}  & 5.79e$-$01 & 1.27e$-$02 & 3.15e$-$03 & 7.88e$-$03 & 1.68e$-$04 & 4.36e$-$05\\
        & \texttt{ode15s} & 5.72e$-$01 & 1.22e$-$02 & 2.91e$-$03 & 3.02e$-$04 & 7.36e$-$06 & 1.58e$-$06 \\
        \hline
       \rowcolor{LightCyan}
        & PIRPNN & 5.13e$-$05 & 1.93e$-$06 & 2.49e$-$07 & 2.98e$-$06 & 1.34e$-$07 & 1.18e$-$08\\
        $u_4$ & \texttt{ode23t}  & 4.50e$-$01 & 1.35e$-$02 & 2.30e$-$03 & 6.42e$-$03 & 1.82e$-$04 & 3.27e$-$05\\
        & \texttt{ode15s} & 6.89e$-$01 & 2.03e$-$02 & 3.53e$-$03 & 3.87e$-$04 & 2.38e$-$05 & 1.86e$-$06 \\
        \hline
       \rowcolor{LightCyan}
        & PIRPNN & 7.53e$-$04 & 3.12e$-$05 & 4.11e$-$06 & 4.63e$-$05 & 2.00e$-$06 & 2.01e$-$07\\
        $u_5$ & \texttt{ode23t}  & 6.39e$+$00 & 1.39e$-$01 & 3.53e$-$02 & 9.00e$-$02 & 1.95e$-$03 & 5.05e$-$04\\
        & \texttt{ode15s} & 9.61e$+$00 & 2.04e$-$01 & 5.65e$-$02 & 4.56e$-$03 & 1.10e$-$04 & 2.60e$-$05 \\
    \hline
    \end{tabular}
}
\end{center}
\end{table}
Table~\ref{tab:ex11_accuracy} summarizes the $l^2$, $l^{\infty}$ and mean absolute (MAE) errors with respect to the reference solution in $15,000$ equally spaced grid points in the interval $[0 \quad 15]$. As shown, for the given tolerances, the proposed method outperforms \texttt{ode15s} and \texttt{ode23t} in all metrics.
\begin{figure}[ht]
    \centering
    \subfigure[]{
    \includegraphics[width=0.3 \textwidth]{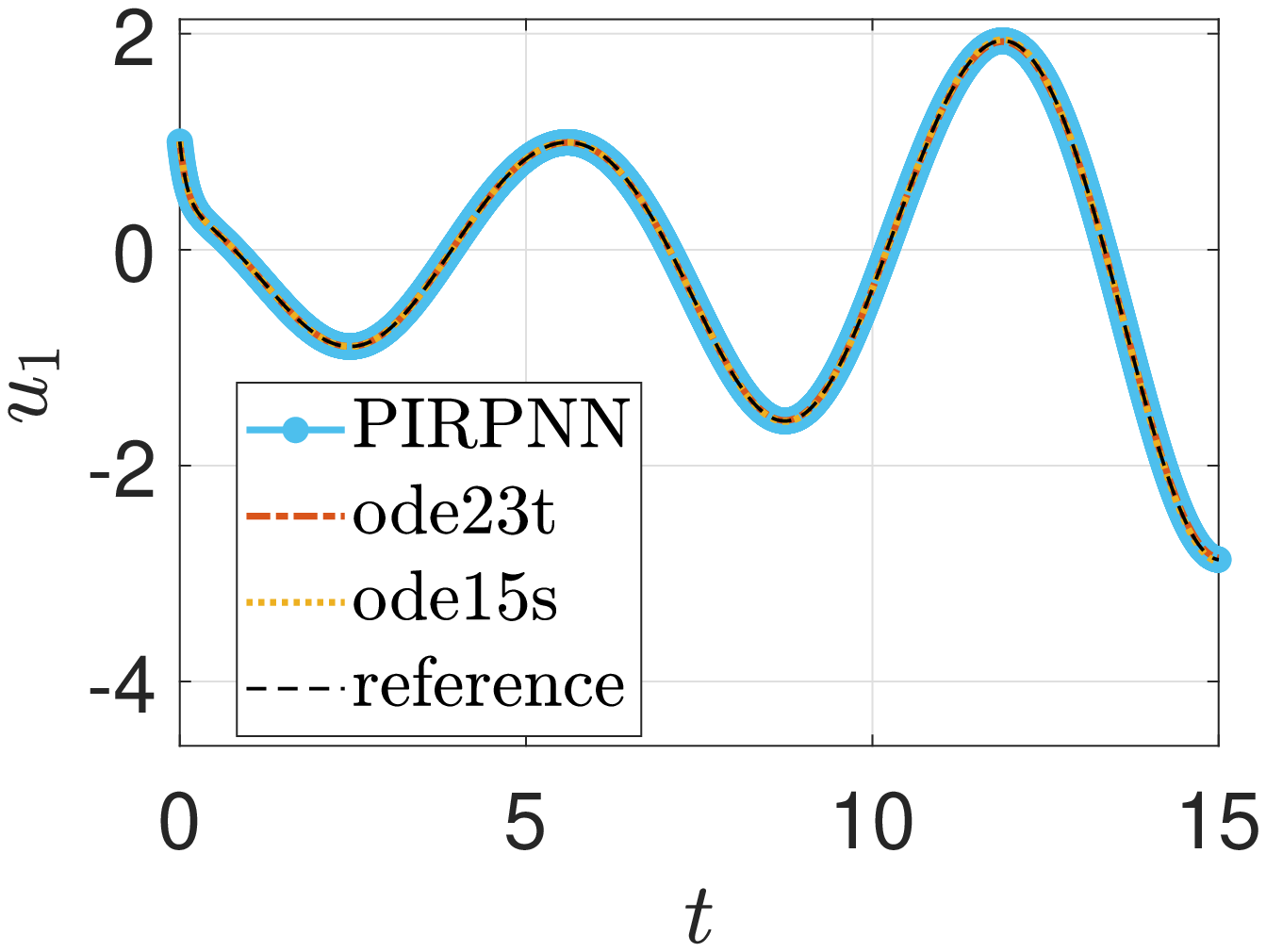}
    }
    \subfigure[]{
    \includegraphics[width=0.3 \textwidth]{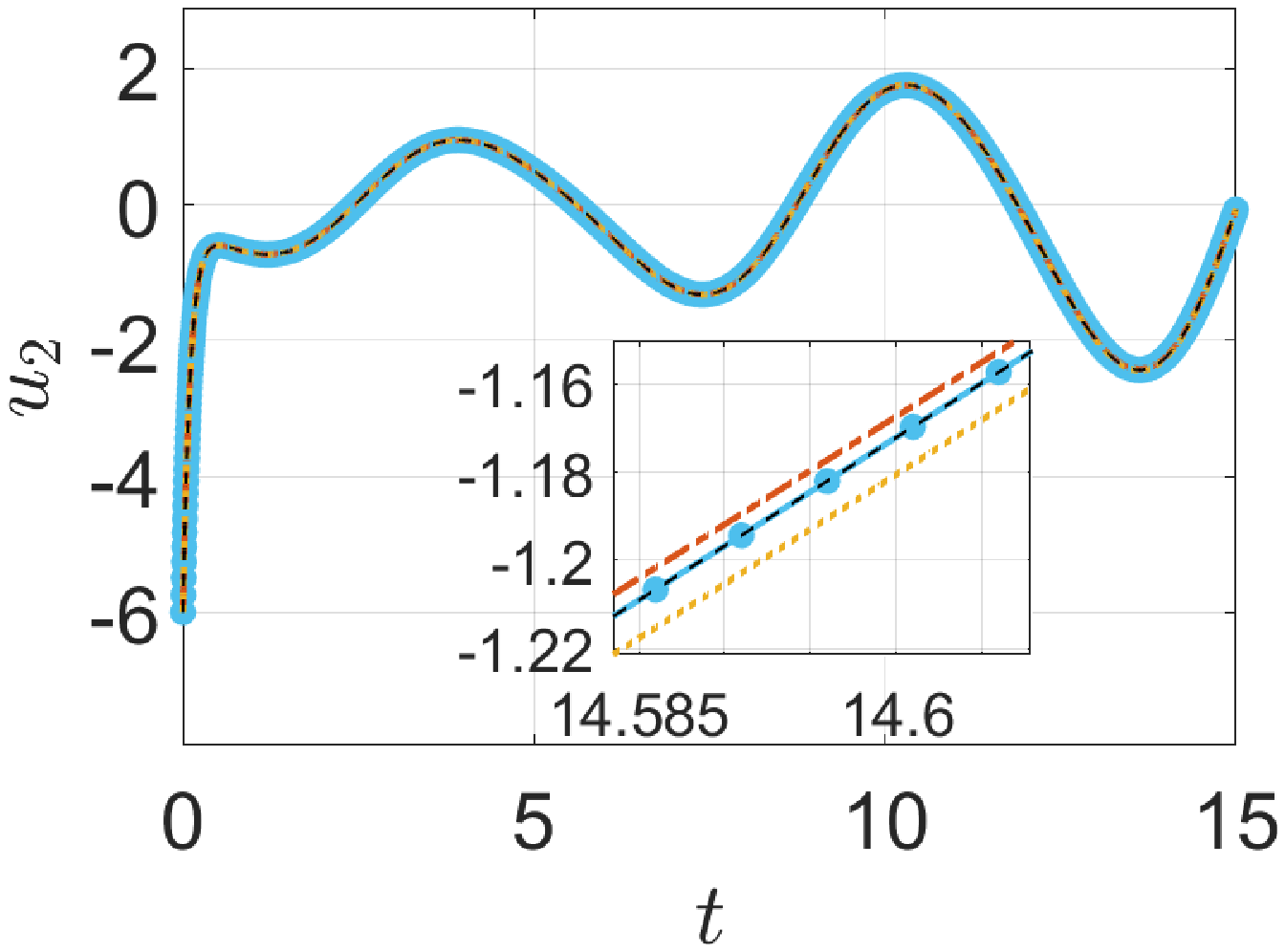}
    }
    \subfigure[]{
    \includegraphics[width=0.3 \textwidth]{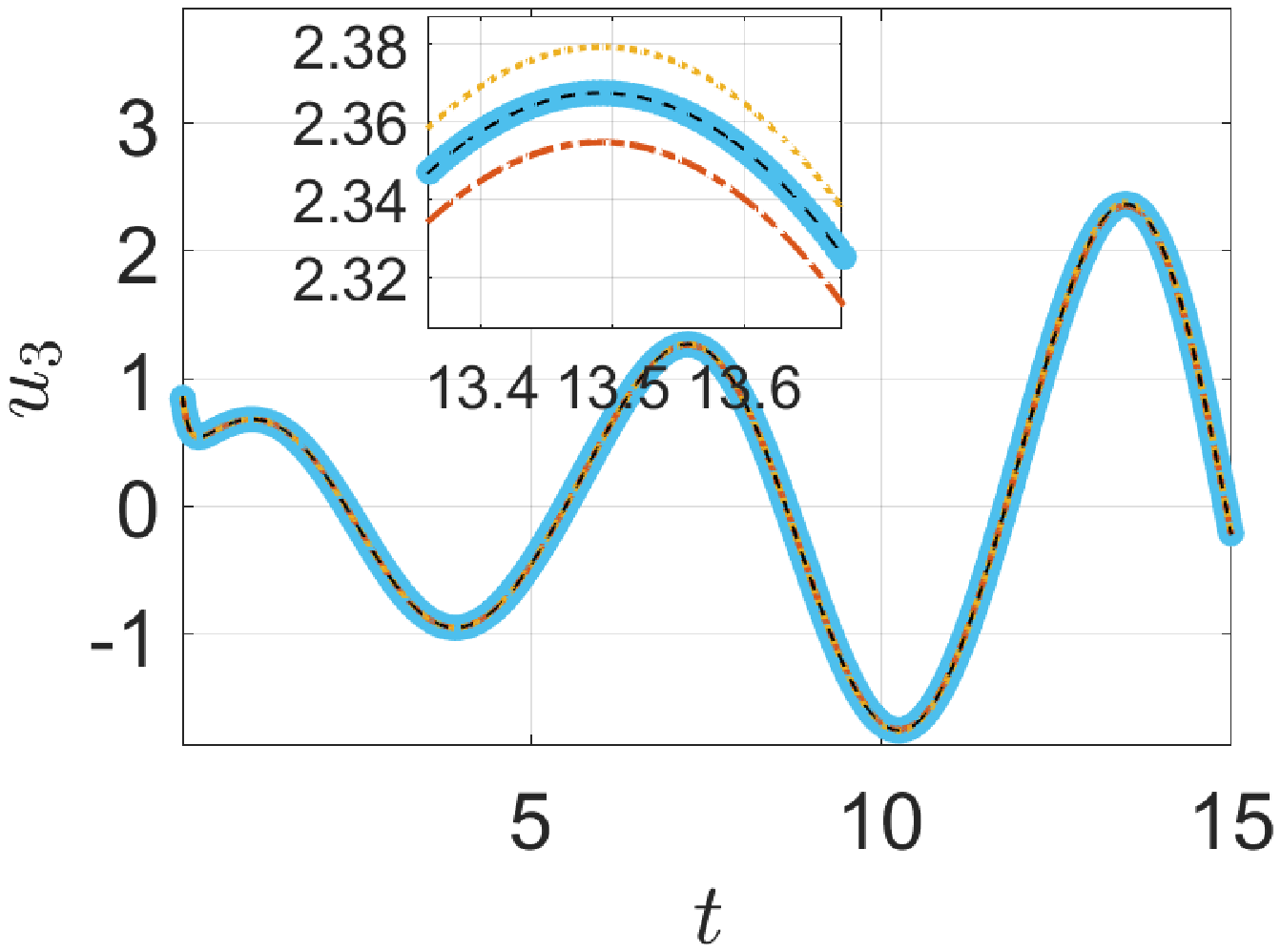}
    }
    \subfigure[]{
    \includegraphics[width=0.3 \textwidth]{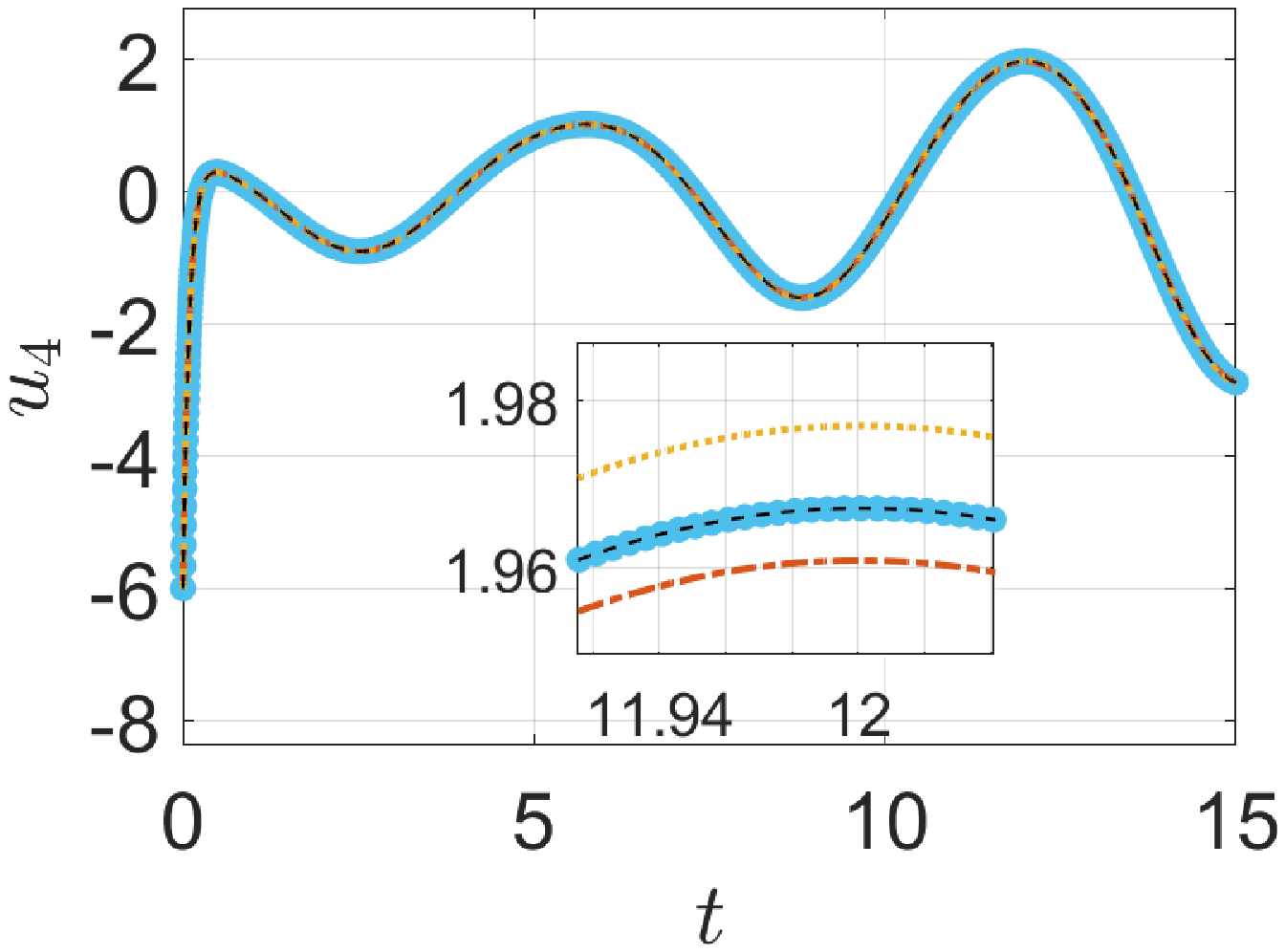}
    }
    \subfigure[]{
    \includegraphics[width=0.3 \textwidth]{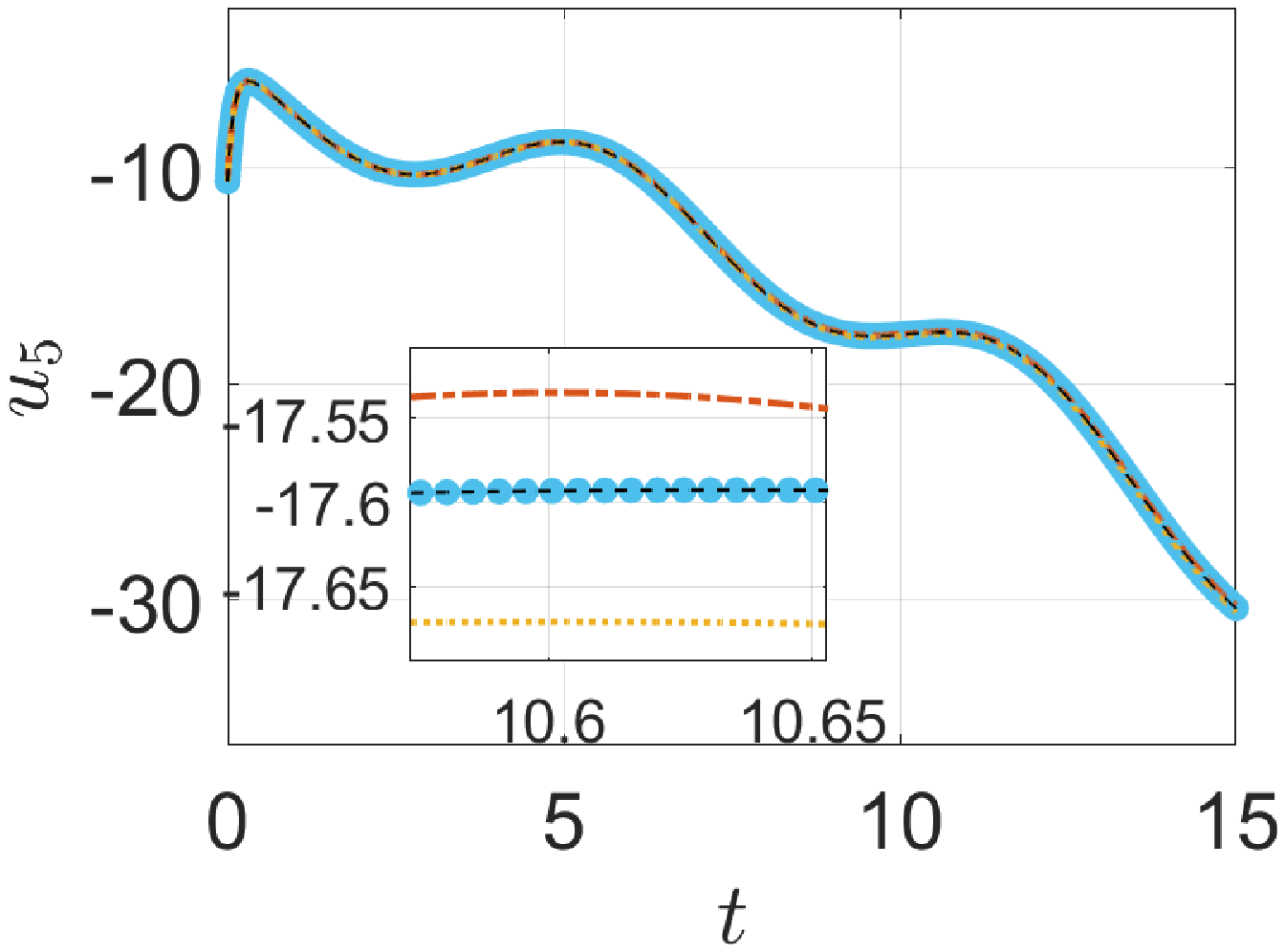}
    }
    \caption{The mechanics problem of non autonomous index-1 DAEs \eqref{eq:ex11} \cite{shampine1999solving}. Approximate solutions computed in the interval $[0 \quad 15]$ with both absolute and relative tolerances set to~1e$-$03. Insets depict zooms around the reference solution.\label{fig:ex11}}
\end{figure}
As it is shown in Figure~\ref{fig:ex11}, for tolerances 1e$-$03, the proposed scheme achieves more accurate solutions than \texttt{ode23t} and \texttt{ode15s}.
\begin{table}[ht]
\begin{center}
\caption{The mechanics problem of non autonomous index-1 DAEs \cite{shampine1999solving} \eqref{eq:ex11}. Computational times (s) (median, minimum and maximum over $10$ runs) and number of points required in the interval $[0 \quad 15]$ by the PIRPNN, \texttt{ode23t} and \texttt{ode15s} with both absolute and relative tolerances set to 1e$-$03 and 1e$-$06.}
{\footnotesize
\setlength{\tabcolsep}{3pt}
\begin{tabular}{|l|lllr|lllr|}
\hline
& \multicolumn{4}{c|}{$tol=$ 1e$-$03} & \multicolumn{4}{c|}{$tol=$ 1e$-$06} \\
\cline{2-9}
  & median & min & max & \multicolumn{1}{l|}{\# pts} & median & min & max & \multicolumn{1}{l|}{\# pts}\\
  \hline
  \rowcolor{LightCyan}
  PIRPNN & 1.72e$-$02 & 1.38e$-$02 & 3.09e$-$02 &  140 & 2.33e$-$02 & 1.96e$-$02 & 2.63e$-$02 & 200\\
  \texttt{ode23t}  & 4.97e$-$03 & 4.66e$-$03 & 3.23e$-$02 & 157 & 1.81e$-$02 & 1.72e$-$02 & 3.87e$-$02 & 1320\\
  \texttt{ode15s} & 4.71e$-$03 & 4.45e$-$03 & 1.08e$-$02 & 145 & 7.75e$-$03 & 7.57e$-$03 & 1.51e$-$02 & 295\\
  reference & 3.18e$-$01 & 3.06e$-$01 & 4.92e$-$01 & 20734 & 3.17e$-$01 & 3.05e$-$01 & 4.92e$-$01 & 20734\\
  \hline
\end{tabular}
}
\end{center}
\label{tab:ex11_time_points}
\end{table}
In Table~\ref{tab:ex11_time_points}, we report the computational times and number of points required by each method, including the ones required for computing the reference solution.
As shown, the corresponding total number of points required by the proposed scheme is comparable with the ones required by \texttt{ode15s} and significantly less than the number of points required by \texttt{ode23t} and the reference solution. Furthermore, the computational times of the proposed method are comparable with the ones required by the \texttt{ode23t} and \texttt{15s}, thus outperforming the ones required for computing the reference solution.

\subsection{Case Study 3: Power discharge control non autonomous index-1 DAEs model}
This is a non autonomous model of six index-1 DAEs and it is part of the benchmark problems considered in \cite{shampine1999solving}. The governing equations are:
\begin{equation}
\begin{aligned}
    &\frac{du_1}{dt} =\frac{(u_2-u_1)}{20},\\
    &\frac{du_2}{dt} = - \frac{(u_4-99.1)}{75},\\
    &\frac{du_3}{dt} = \mu-u_6,\\
    &0 = 20u_5-u_3 \\
     &0 = (3.35 - 0.075u_6 + 0.001u_6^2) - \frac{u_4}{u_5},\\
  &0 = \frac{u_3}{400}\frac{du_3}{dt}+\frac{\mu\mu_p}{(1.2u_1)^2}- \frac{du_1}{dt}\frac{\mu^2}{(1.44u_1)^3},\\
     &\mu =  15 + 5{tanh}(t - 10), \quad \mu_p = \frac{5}{{cosh}^2(t - 10)}.
\end{aligned}
\label{eq:ex13}
\end{equation}
The initial conditions are $u_1(0)=u_2(0)=0.25$, $u_3(0) =734,$ and consistent initial conditions for $u_4(0) = 99.08999492002$, $u_5(0)=36.7$ and $u_6(0)=10.00000251671$ were found at $t=0$ with Newton-Raphson with a tolerance of 1e$-$16.
\begin{table}[ht]
\begin{center}
\caption{Power discharge control non autonomous index-1 DAEs problem \eqref{eq:ex13} in $[0 \quad 40]$. $l^{2}$, $l^{\infty}$ and mean absolute (MAE) approximation errors obtained with both absolute and relative tolerances set to 1e$-$03 and 1e$-$06. \label{tab:ex13_accuracy}}
{\footnotesize
    \begin{tabular}{|l|l |l l l |l l l|}
    \hline
\multicolumn{2}{|c|}{} & \multicolumn{3}{c|}{$tol=$ 1e$-$03} & \multicolumn{3}{c|}{$tol=$ 1e$-$06} \\
\cline{3-8}
        \multicolumn{2}{|c|}{} & $l^2$ & $l^{\infty}$ & MAE & $l^2$ & $l^{\infty}$ & MAE \\
        \hline
        \rowcolor{LightCyan}
        & PIRPNN & 1.37e$-$05 & 1.46e$-$07 & 5.59e$-$08 & 2.69e$-$06 & 2.27e$-$08 & 1.13e$-$08\\
        $u_1$ & \texttt{ode23t}  & 1.98e$-$01 & 1.61e$-$03 & 8.40e$-$04 & 1.65e$-$03 & 1.21e$-$05 & 7.07e$-$06\\
        & \texttt{ode15s} & 1.83e$-$01 & 1.48e$-$03 & 7.96e$-$04 & 3.46e$-$04 & 2.32e$-$06 & 1.55e$-$06 \\
        \hline
        \rowcolor{LightCyan}
        & PIRPNN & 1.72e$-$04 & 3.27e$-$05 & 1.31e$-$07 & 4.49e$-$06 & 6.44e$-$08 & 1.96e$-$08\\
        $u_2$ & \texttt{ode23t}  &  8.31e$-$01 & 1.04e$-$02 & 3.05e$-$03 & 2.47e$-$03 & 1.91e$-$05 & 1.10e$-$05\\
        & \texttt{ode15s} & 3.10e$-$01 & 2.80e$-$03 & 1.39e$-$03 & 5.45e$-$04 & 5.56e$-$06 & 2.47e$-$06 \\
        \hline
       \rowcolor{LightCyan}
        & PIRPNN & 1.02e$-$02 & 1.18e$-$04 & 3.96e$-$05 & 1.32e$-$04 & 9.95e$-$07 & 5.67e$-$07\\
        $u_3$ & \texttt{ode23t}  & 9.36e$+$00 & 8.64e$-$02 & 3.74e$-$02 & 1.26e$-$01 & 1.21e$-$03 & 5.05e$-$04\\
        & \texttt{ode15s} & 7.05e$+$00 & 8.12e$-$02 & 2.78e$-$02 & 1.60e$-$02 & 1.50e$-$04 & 7.13e$-$05 \\
        \hline
       \rowcolor{LightCyan}
        & PIRPNN & 2.49e$-$03 & 4.87e$-$04 & 1.49e$-$06 & 1.96e$-$05 & 8.32e$-$07 & 6.50e$-$08\\
        $u_4$ & \texttt{ode23t}  & 1.23e$+$01 & 1.56e$-$01 & 4.32e$-$02 & 1.15e$-$02 & 1.37e$-$04 & 4.57e$-$05\\
        & \texttt{ode15s} & 1.23e$+$00 & 2.17e$-$02 & 4.02e$-$03 & 4.42e$-$03 & 4.93e$-$05 & 1.91e$-$05\\
        \hline
       \rowcolor{LightCyan}
        & PIRPNN & 5.12e$-$04 & 5.88e$-$06 & 1.98e$-$06 & 6.60e$-$06 & 4.98e$-$08 & 2.83e$-$08\\
        $u_5$ & \texttt{ode23t}  & 4.68e$-$01 & 4.32e$-$03 & 1.87e$-$03 & 6.28e$-$03 & 6.03e$-$05 & 2.52e$-$05\\
        & \texttt{ode15s} & 3.53e$-$01 & 4.06e$-$03 & 1.39e$-$03 & 8.00e$-$04 & 7.50e$-$06 & 3.57e$-$06 \\
        \hline
       \rowcolor{LightCyan}
        & PIRPNN & 1.36e$-$02 & 2.81e$-$03 & 5.58e$-$06 & 3.30e$-$05 & 5.35e$-$06 & 5.06e$-$08\\
        $u_6$ & \texttt{ode23t}  & 8.07e$+$00 & 1.08e$-$01 & 2.84e$-$02 & 5.71e$-$02 & 8.05e$-$04 & 2.04e$-$04\\
        & \texttt{ode15s} & 1.31e$+$00 & 5.72e$-$02 & 3.54e$-$03 & 3.74e$-$03 & 1.32e$-$04 & 1.37e$-$05 \\
    \hline
    \end{tabular}
}
\end{center}
\end{table}
Table~\ref{tab:ex13_accuracy} summarizes the $l^2$, $l^{\infty}$ and mean absolute (MAE) approximation errors with respect to the reference solution in $40,000$ equally spaced grid points in the interval $[0 \quad 40]$. As shown, for the given tolerances, the proposed method outperforms \texttt{ode15s} and \texttt{ode23t} in all metrics.
\begin{figure}[ht]
    \centering
    \subfigure[]{
    \includegraphics[width=0.3 \textwidth]{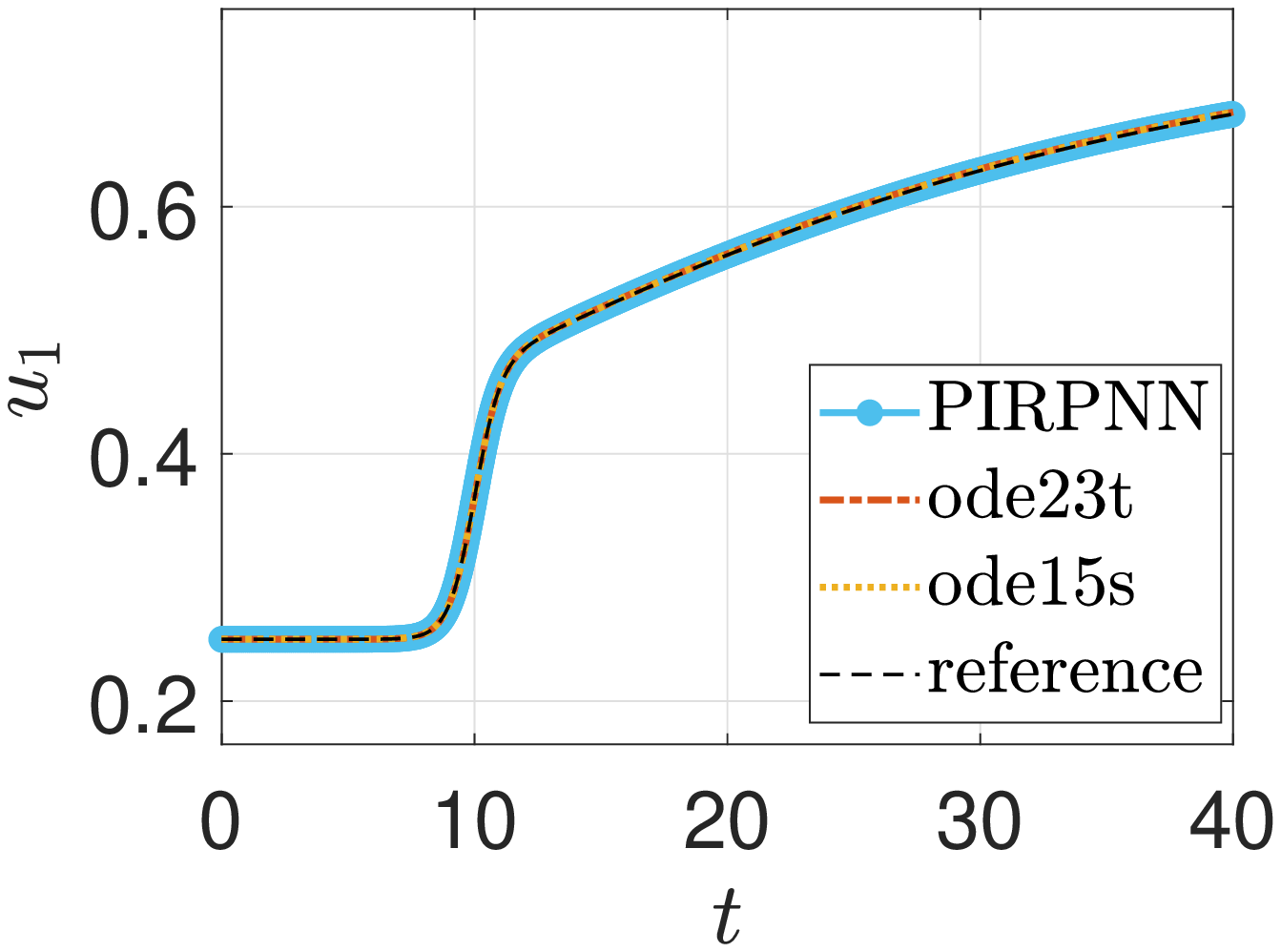}
    }
    \subfigure[]{
    \includegraphics[width=0.3 \textwidth]{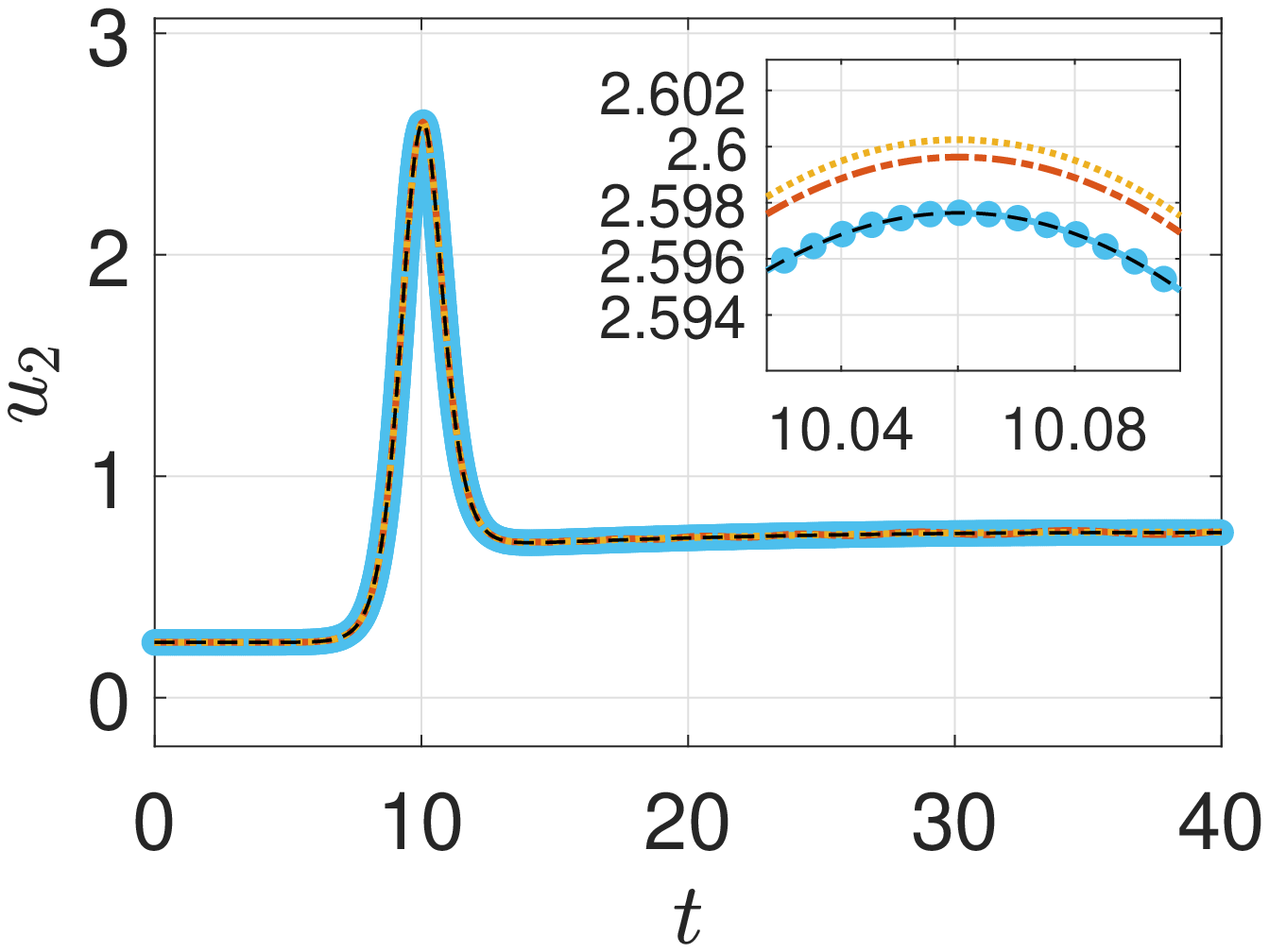}
    }
    \subfigure[]{
    \includegraphics[width=0.3 \textwidth]{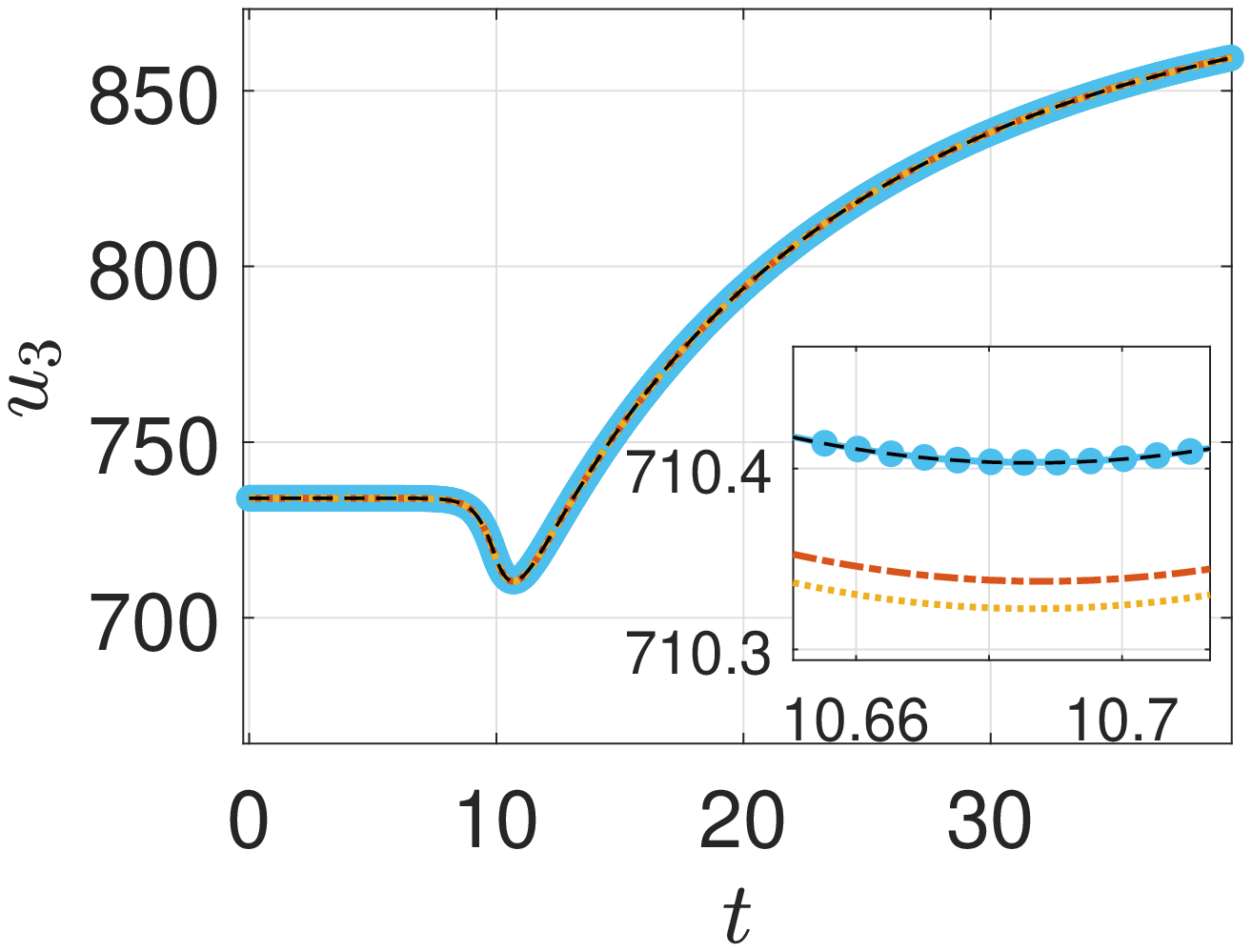}
    }
    \subfigure[]{
    \includegraphics[width=0.3 \textwidth]{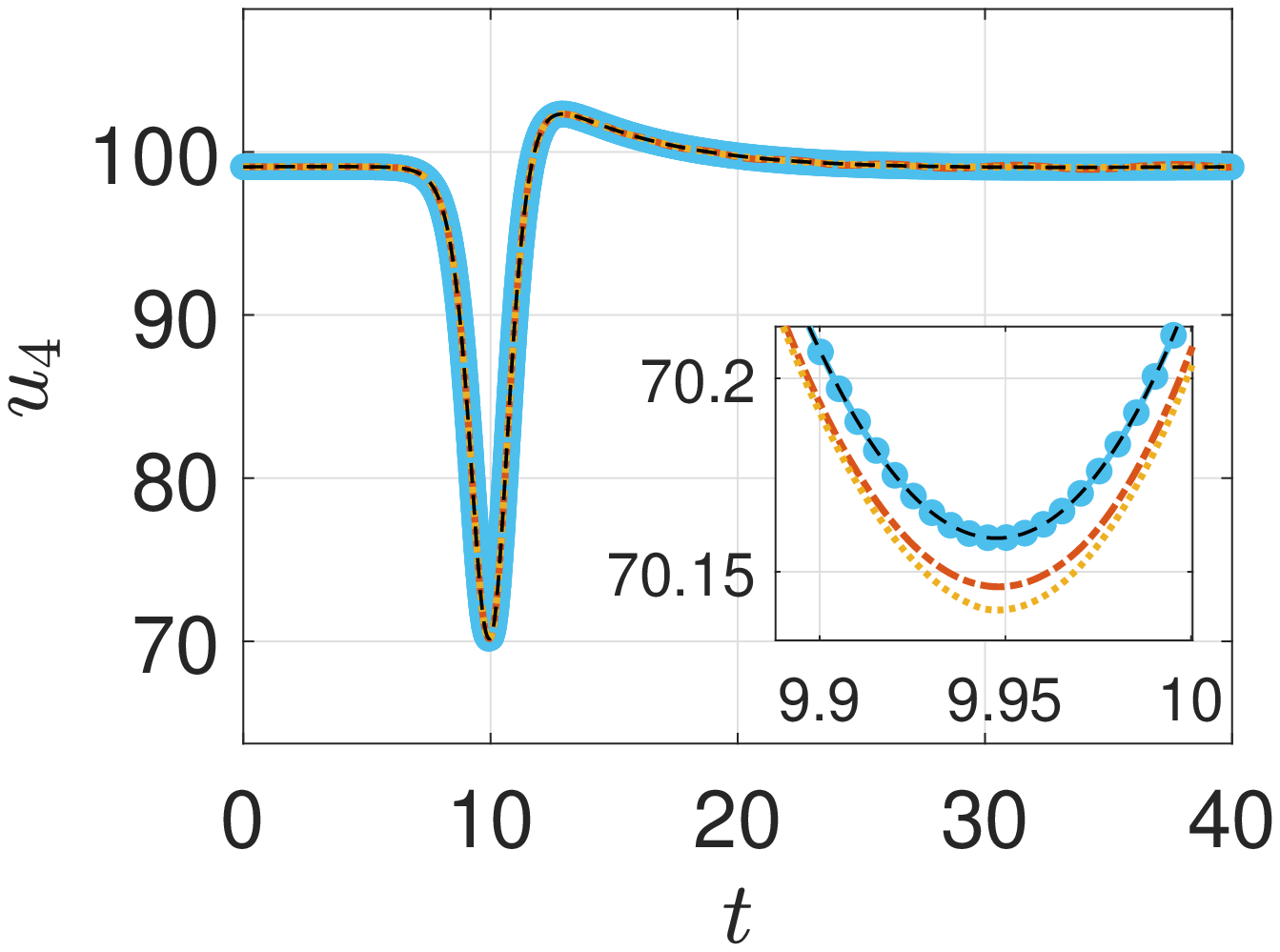}
    }
    \subfigure[]{
    \includegraphics[width=0.3 \textwidth]{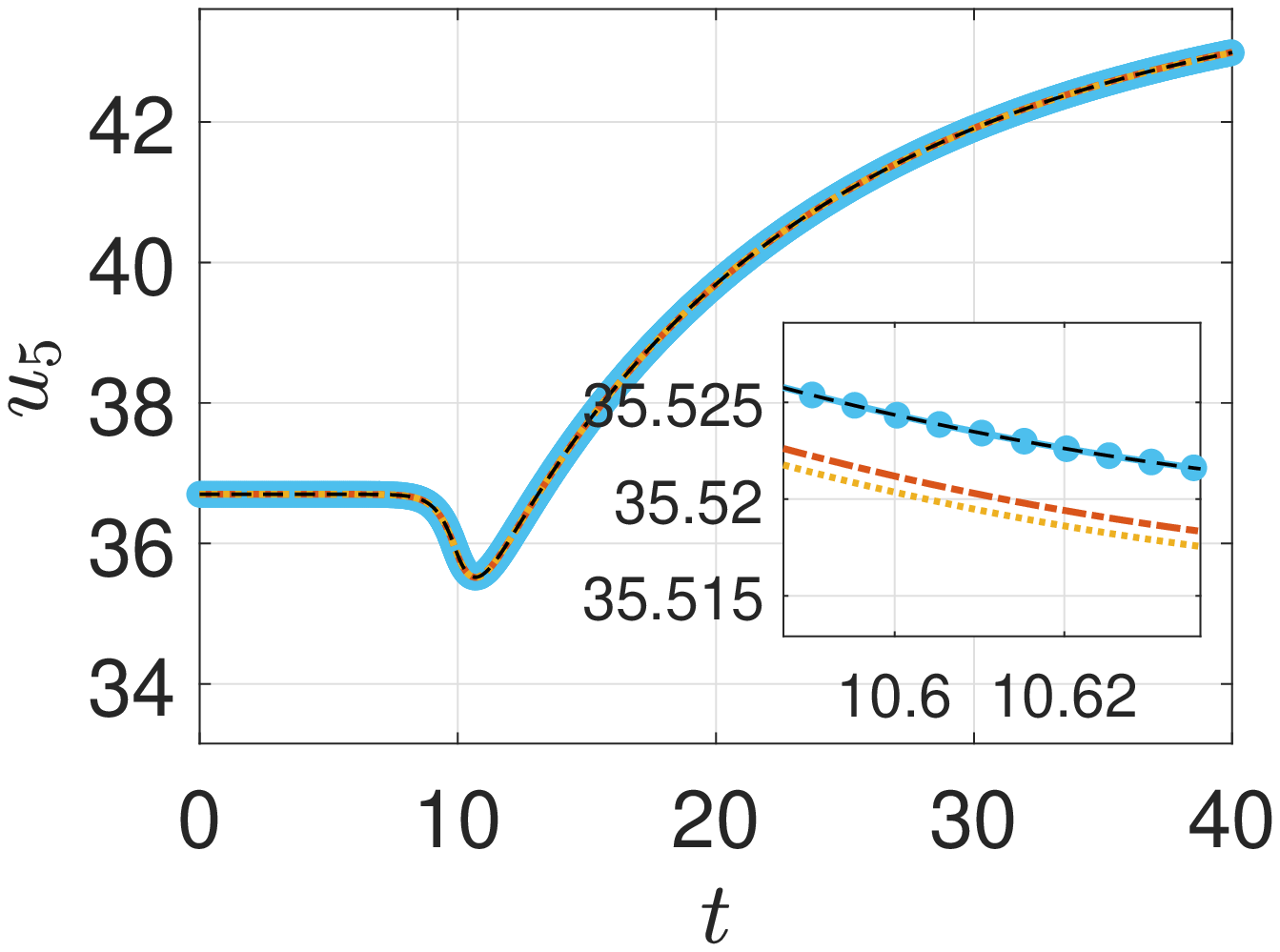}
    }
    \subfigure[]{
    \includegraphics[width=0.3 \textwidth]{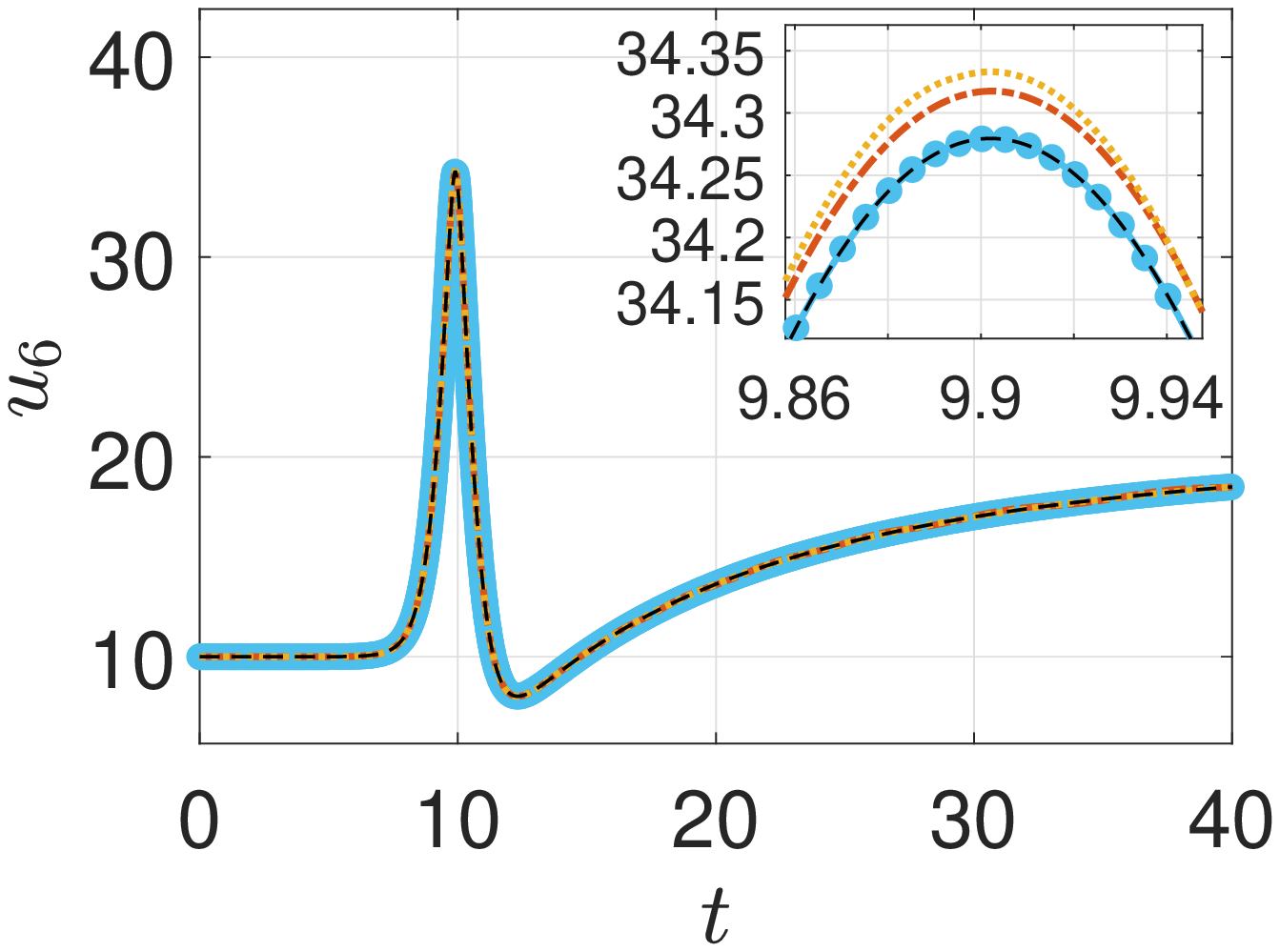}
    }
    \caption{Power discharge control non autonomous index-1 DAEs problem \eqref{eq:ex13} \cite{shampine1999solving}. Approximate solutions computed in the interval $[0 \quad 40]$ with both absolute and relative tolerances set to 1e$-$03. Insets depict zooms around the reference solution. \label{fig:ex13}}
\end{figure}
As it is shown in Figure~\ref{fig:ex13}, for tolerances 1e$-$03, the proposed scheme achieves more accurate solutions than \texttt{ode23t} and \texttt{ode15s}.
\begin{table}[ht]
\begin{center}
\caption{Power discharge control non autonomous index-1 DAEs problem \eqref{eq:ex13}. Computational times (median, minimum and maximum over $10$ runs) and number of points required in the interval $[0 40]$ by the PIRPNN, \texttt{ode23t} and \texttt{ode15s} with both absolute and relative tolerances set to 1e$-$03 and 1e$-$06.}
{\footnotesize
\setlength{\tabcolsep}{3pt}
\begin{tabular}{|l|lllr|lllr|}
\hline
& \multicolumn{4}{c|}{$tol=$ 1e$-$03} & \multicolumn{4}{c|}{$tol=$ 1e$-$06} \\
\cline{2-9}
  & median & min & max & \multicolumn{1}{l|}{\# pts} & median & min & max & \multicolumn{1}{l|}{\# pts}\\
  \hline
  \rowcolor{LightCyan}
  PIRPNN & 8.55e$-$02 & 7.75e$-$02 & 9.93e$-$02 &  394 & 9.26e$-$02 & 8.65e$-$02 & 1.16e$-$01 & 470\\
  \texttt{ode23t}  & 4.30e$-$03 & 2.65e$-$03 & 5.67e$-$02 & 69 & 1.06e$-$02 & 9.52e$-$03 & 4.97e$-$02 & 642\\
  \texttt{ode15s} & 2.93e$-$03 & 2.61e$-$03 & 1.49e$-$02 & 77 & 6.04e$-$03 & 5.68e$-$03 & 1.85e$-$02 & 229\\
  reference & 9.01e$-$02 & 8.88e$-$02 & 1.75e$-$01 & 5093 & 9.01e$-$02 & 8.88e$-$02 & 1.75e$-$01 & 5093\\
  \hline
\end{tabular}
}
\end{center}
\label{tab:ex13_time_points}
\end{table}
In Table~\ref{tab:ex13_time_points}, we report the computational times and number of points required by each method, including the ones required for computing the reference solution. As shown, the corresponding total corresponding number of points required by the proposed scheme is comparable with the ones required by \texttt{ode23t} and \texttt{ode15s} and significantly less than the number of points required by the reference solution.
Furthermore, the computational times of the proposed method are comparable with the ones required by the \texttt{ode23t} and \texttt{ode15s}, thus outperforming the ones required for computing the reference solution.

\subsection{Case Study 4: The Chemical Akzo Nobel index-1 DAE model}
The Chemical Akzo Nobel problem is a benchmark problem made up of six non-linear index-1 DAEs.
This problem originates from Akzo Nobel research center in Amsterdam and was described in \cite{mazzia2012test,stortelder1998parameter}.
The resulting system of index-1 DAEs is given by:
\begin{equation}
    \begin{aligned}
    &\dfrac{du_1}{dt} &= -2 k_1 u_1^4 u_2^{\frac{1}{2}} - \dfrac{k_2}{K} u_1u_5 + k_2 u_3 u_4 - k_3 u_1 u_4^2\\
    &\dfrac{ du_2 }{ dt } &= -\dfrac{1}{2} k_1 u_1^4 u_2^{\frac{1}{2}} - k_3 u_1 u_4^2 - \dfrac{1}{2} k_4 u_6^2 u_2^{\frac{1}{2}} + k_{in}\biggl(\dfrac{\rho}{H}-u_2\biggr)\\
    &\dfrac{ du_3 }{ dt } &= k_1 u_1^4 u_2^{\frac{1}{2}} + \dfrac{k_2}{K}u_1u_5 - k_2u_3u_4\\
    &\dfrac{ du_4 }{ dt } &= -\dfrac{k_2}{K} u_1u_5 + k_2 u_3 u_4 + k_4 u_6^2 u_2^{\frac{1}{2}}\\
    &\dfrac{ du_5 }{ dt } &= \dfrac{k_2}{K} u_1u_5 - k_2 u_3 u_4 - 2 k_3 u_1 u_4^2\\
    &0 &= K_s u_1 u_4 - u_6,
    \end{aligned}
    \label{eq:chemakzo}
\end{equation}
where $u_1,\, u_2,\dots,\, u_6$ denote the concentrations of 6 chemical species, $k_1=18.7, \, k_2=0.58, \, k_3=0.09, \, k_4=0.42$ and  $K=34.4$ are the reaction rate coefficients, $K_s=115.83$ is a coefficient of proportionality between $u_1u_4$ and $u_6$, and the costant injection of $u_2$ in the system is governed by a rate $k_{in}=3.3,\,$ and constants $\rho=0.9$ and $H=737$. The initial conditions are set as: $u_1=0.444, \, u_2= 0.0012,\, u_3=0,\, u_4=0.007,\, u_5=0$.

\begin{table}[ht]
\begin{center}
\caption{The chemical Akzo Nobel problem \eqref{eq:chemakzo}. Absolute error ($l^2$-norm, $l^{\infty}$-norm and MAE) for the solutions computed with tolerances set to 1e$-$03 and 1e$-$06. The reference solution was obtained with \texttt{ode15s} with tolerances set to 1e$-$14.\label{tab:chemakzo_accuracy}}
    {\footnotesize
    \begin{tabular}{|l|l |l l l |l l l|}
        \hline
\multicolumn{2}{|c|}{} & \multicolumn{3}{c|}{$tol=$ 1e$-$03} & \multicolumn{3}{c|}{$tol=$ 1e$-$06} \\
\cline{3-8}
        \multicolumn{2}{|c|}{} & $l^2$ & $l^{\infty}$ & MAE & $l^2$ & $l^{\infty}$ & MAE\\
        \hline
        \rowcolor{LightCyan}
        & PIRPNN   & 6.09e$-$04 & 3.84e$-$06 & 1.24e$-$06 & 1.35e$-$06 & 8.22e$-$09 & 2.42e$-$09\\
        $u_1$ & \texttt{ode23t}  & 7.17e$-$02 & 5.90e$-$04 & 1.19e$-$04 & 1.88e$-$03 & 5.48e$-$06 & 4.29e$-$06\\
        & \texttt{ode15s} &  1.36e$-$01 & 1.21e$-$03 & 1.62e$-$04 & 2.87e$-$04 & 1.80e$-$06 & 6.11e$-$07\\
        \hline
        \rowcolor{LightCyan}
        & PIRPNN   & 8.66e$-$06 & 5.43e$-$07 & 7.98e$-$09 & 2.74e$-$08 & 7.97e$-$10 & 2.64e$-$11\\
        $u_2$ & \texttt{ode23t}  & 8.24e$-$03 & 5.92e$-$05 & 1.74e$-$05  & 7.20e$-$06 & 5.41e$-$07 & 6.95e$-$09\\
        & \texttt{ode15s} & 1.33e$-$03 & 9.40e$-$05 & 9.48e$-$07 & 9.85e$-$06 & 9.88e$-$07 & 4.42e$-$09 \\
        \hline
        \rowcolor{LightCyan}
        & PIRPNN   & 3.01e$-$04 & 1.92e$-$06 & 6.10e$-$07 & 6.55e$-$07 & 4.04e$-$09 & 1.17e$-$09\\
        $u_3$ & \texttt{ode23t}  & 3.48e$-$02 & 2.86e$-$04 & 5.82e$-$05 & 9.38e$-$04 & 2.74e$-$06 & 2.14e$-$06\\
        & \texttt{ode15s} & 6.94e$-$02 & 6.15e$-$04 & 8.16e$-$05 & 1.44e$-$04 & 8.98e$-$07 & 3.07e$-$07 \\
        \hline
        \rowcolor{LightCyan}
        & PIRPNN   & 1.13e$-$05 & 5.57e$-$08 & 2.14e$-$08 & 4.93e$-$08 & 4.96e$-$10 & 7.16e$-$11\\
        $u_4$ & \texttt{ode23t}  & 2.93e$-$03 & 2.09e$-$05 & 3.93e$-$06 & 3.45e$-$05 & 2.40e$-$07 & 5.93e$-$08 \\
        & \texttt{ode15s} &  3.62e$-$03 & 4.11e$-$05 & 4.50e$-$06 & 7.45e$-$06 & 4.90e$-$08 & 1.25e$-$08 \\
        \hline
        \rowcolor{LightCyan}
        & PIRPNN   & 4.22e$-$04 & 1.54e$-$06 & 7.22e$-$07 &  3.26e$-$07 & 1.17e$-$09 & 6.48e$-$10\\
        $u_5$ & \texttt{ode23t}  & 1.97e$-$02 & 5.51e$-$05 & 4.44e$-$05 & 3.04e$-$04 & 8.23e$-$07 & 6.86e$-$07\\
        & \texttt{ode15s} & 2.75e$-$02 & 9.01e$-$05 & 6.25e$-$05 & 4.61e$-$05 & 1.53e$-$07 & 1.03e$-$07\\
        \hline
        \rowcolor{LightCyan}
        & PIRPNN & 2.59e$-$04 & 2.58e$-$06 & 5.09e$-$07 &  2.52e$-$06 & 4.25e$-$08 & 2.31e$-$09\\
        $u_6$ & \texttt{ode23s}  & 1.42e$-$01 & 1.03e$-$03 & 2.54e$-$04  & 1.06e$-$03 & 7.97e$-$06 & 1.77e$-$06\\
        & \texttt{ode15s} & 8.30e$-$02 & 9.17e$-$04 & 1.06e$-$04 & 1.71e$-$04 & 1.42e$-$06 & 2.80e$-$07\\
        \hline
    \end{tabular}
    }
\end{center}
\end{table}
Table~\ref{tab:chemakzo_accuracy} summarizes the approximation errors, in terms of $l^2$-norm and $l^{\infty}$-norm errors and MAE, with respect to the reference solution in $180,000$ equally spaced grid points in the interval $[0 \quad 180]$. As shown, for the given tolerances, the proposed method outperforms \texttt{ode15s} and \texttt{ode23t} in all metrics.

\begin{figure}
    \centering
    \subfigure[]{\includegraphics[width=0.3\textwidth]{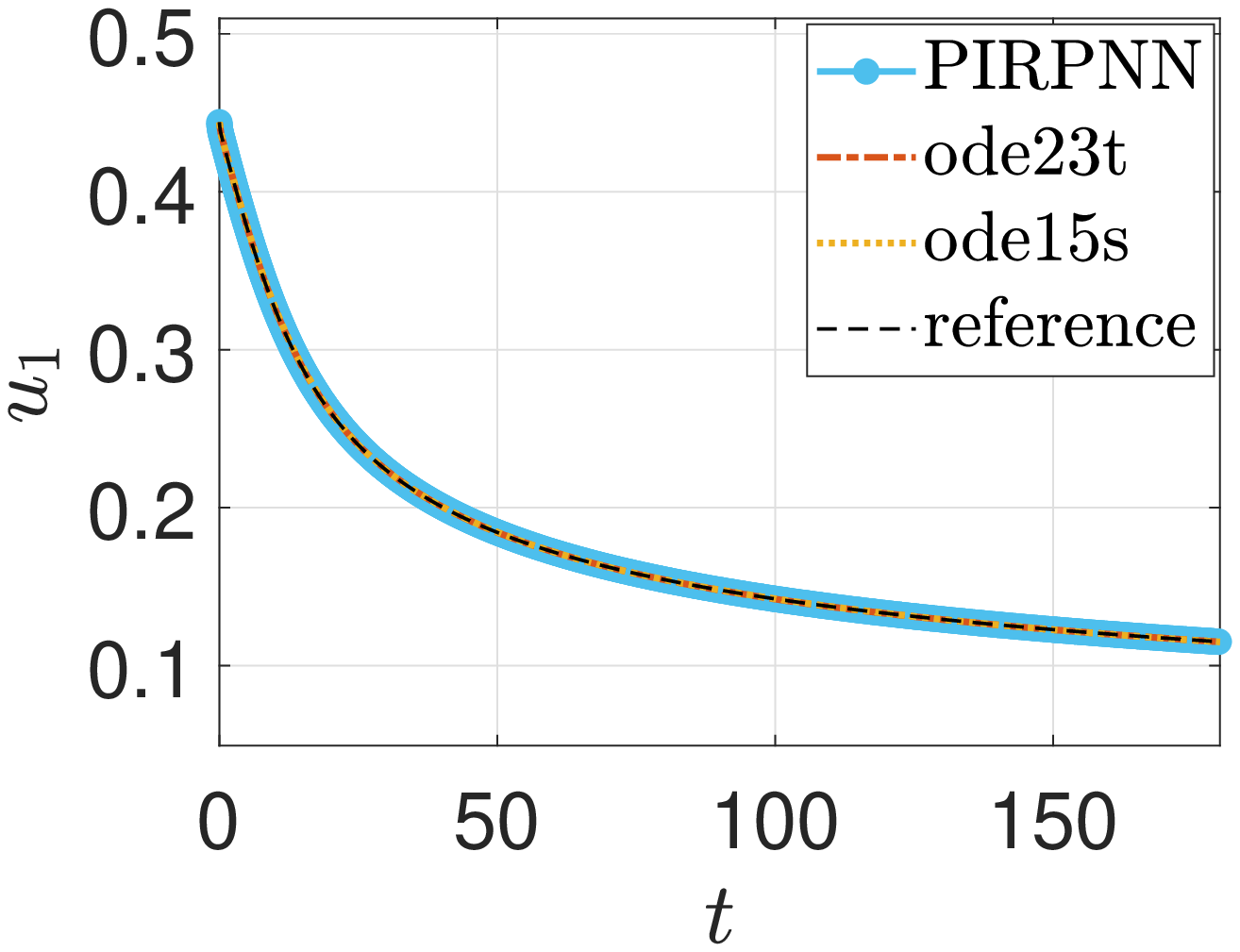}}
    \subfigure[]{\includegraphics[width=0.3\textwidth]{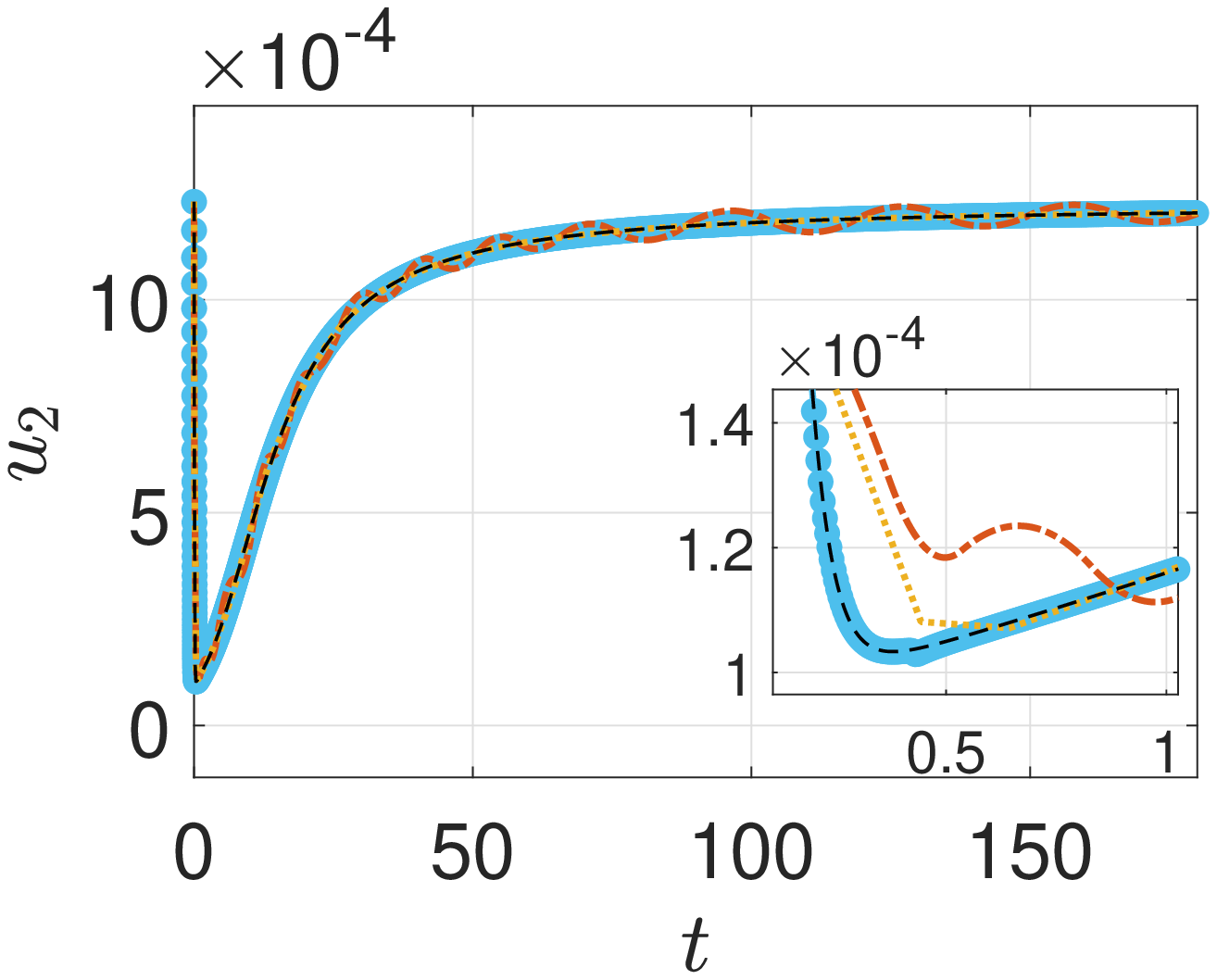}}
    \subfigure[]{\includegraphics[width=0.3\textwidth]{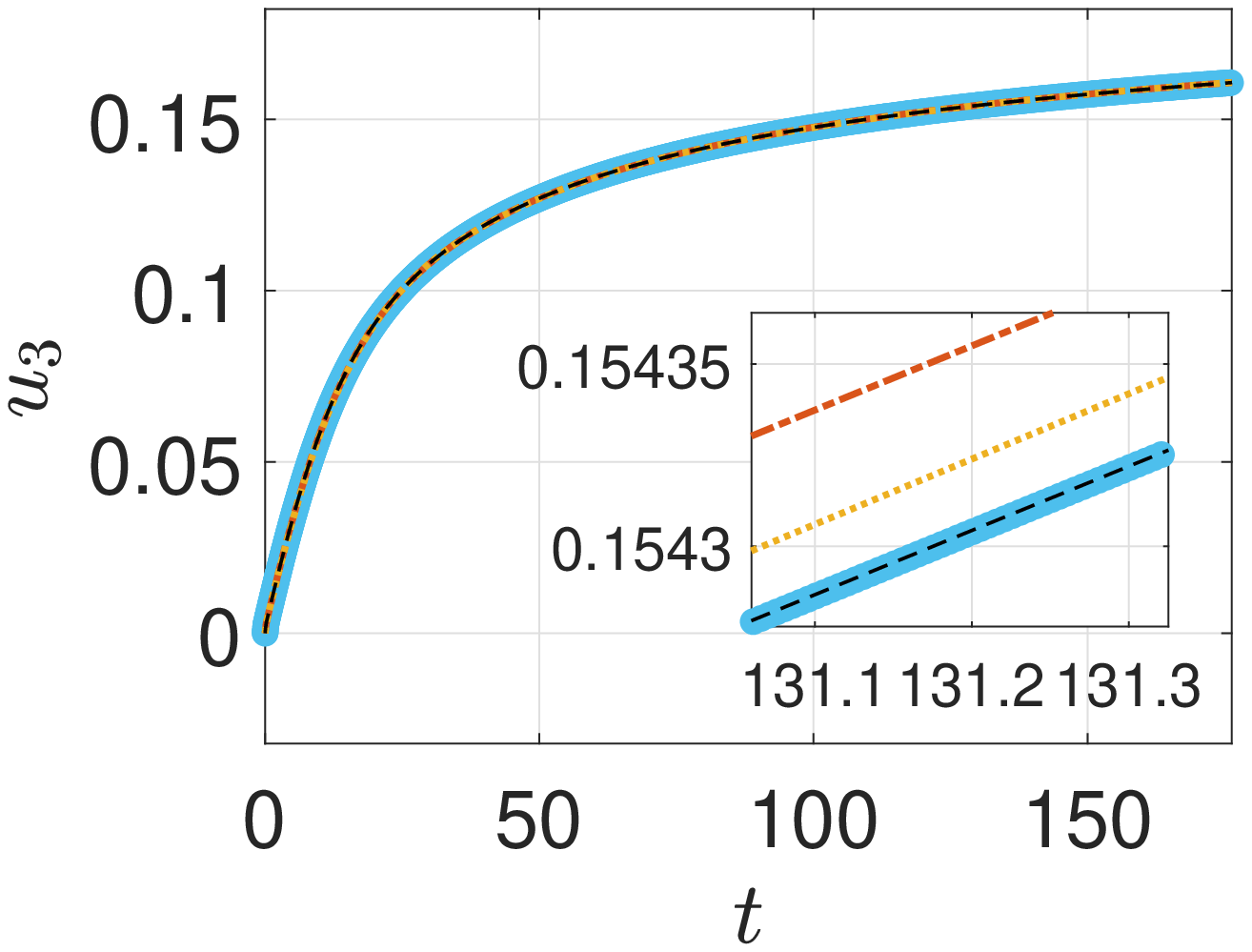}}
    \subfigure[]{\includegraphics[width=0.3\textwidth]{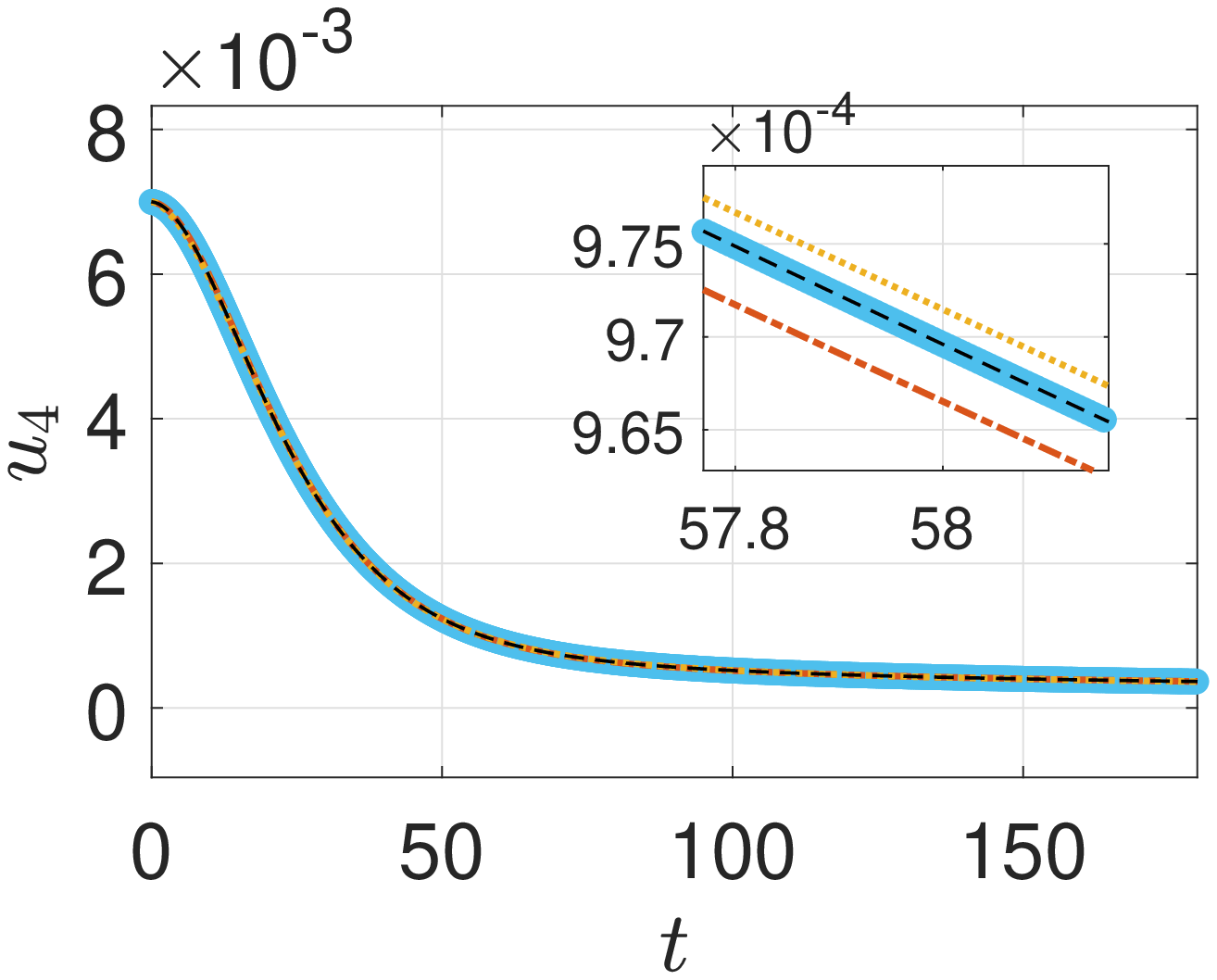}}
    \subfigure[]{\includegraphics[width=0.3\textwidth]{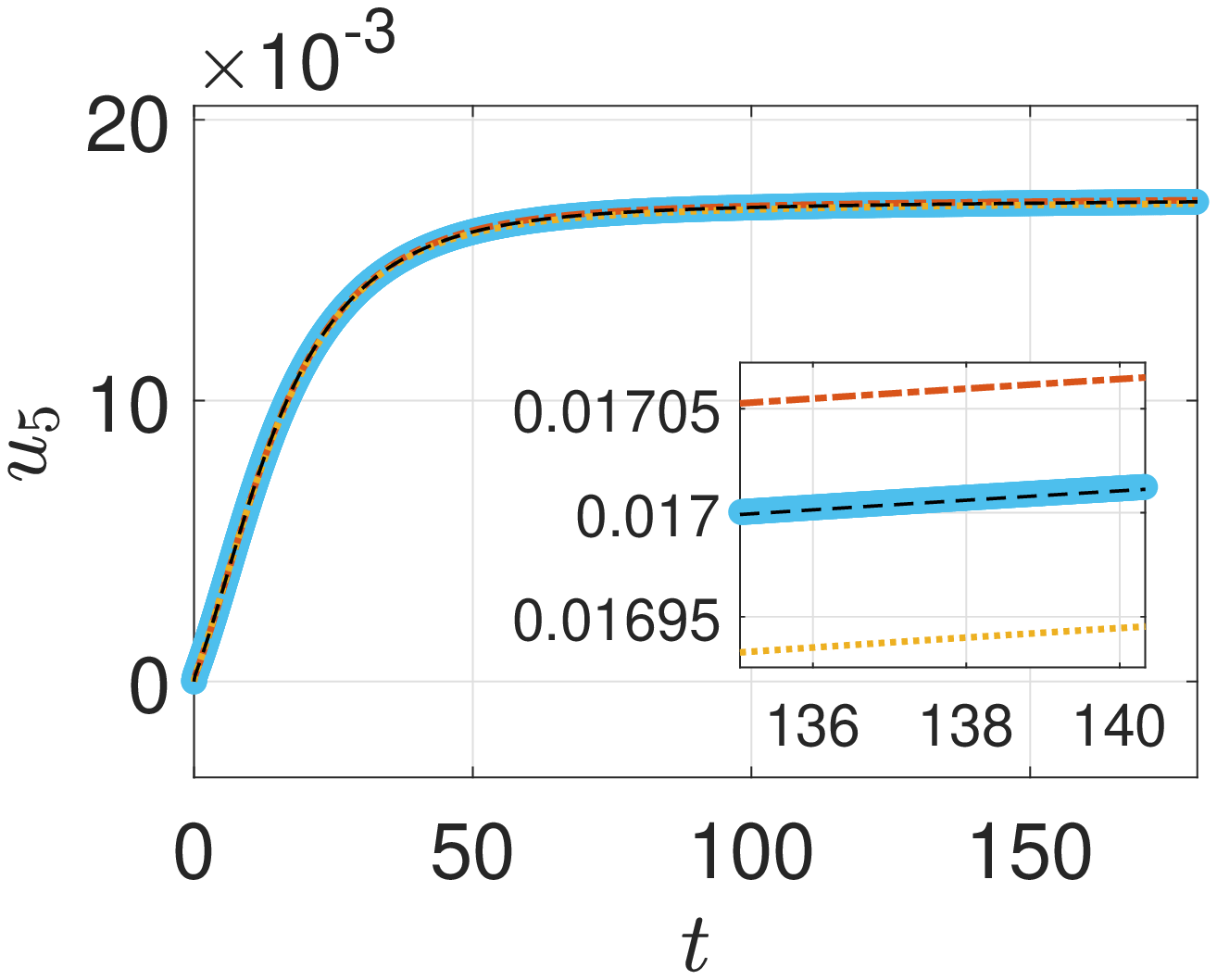}}
    \subfigure[]{\includegraphics[width=0.3\textwidth]{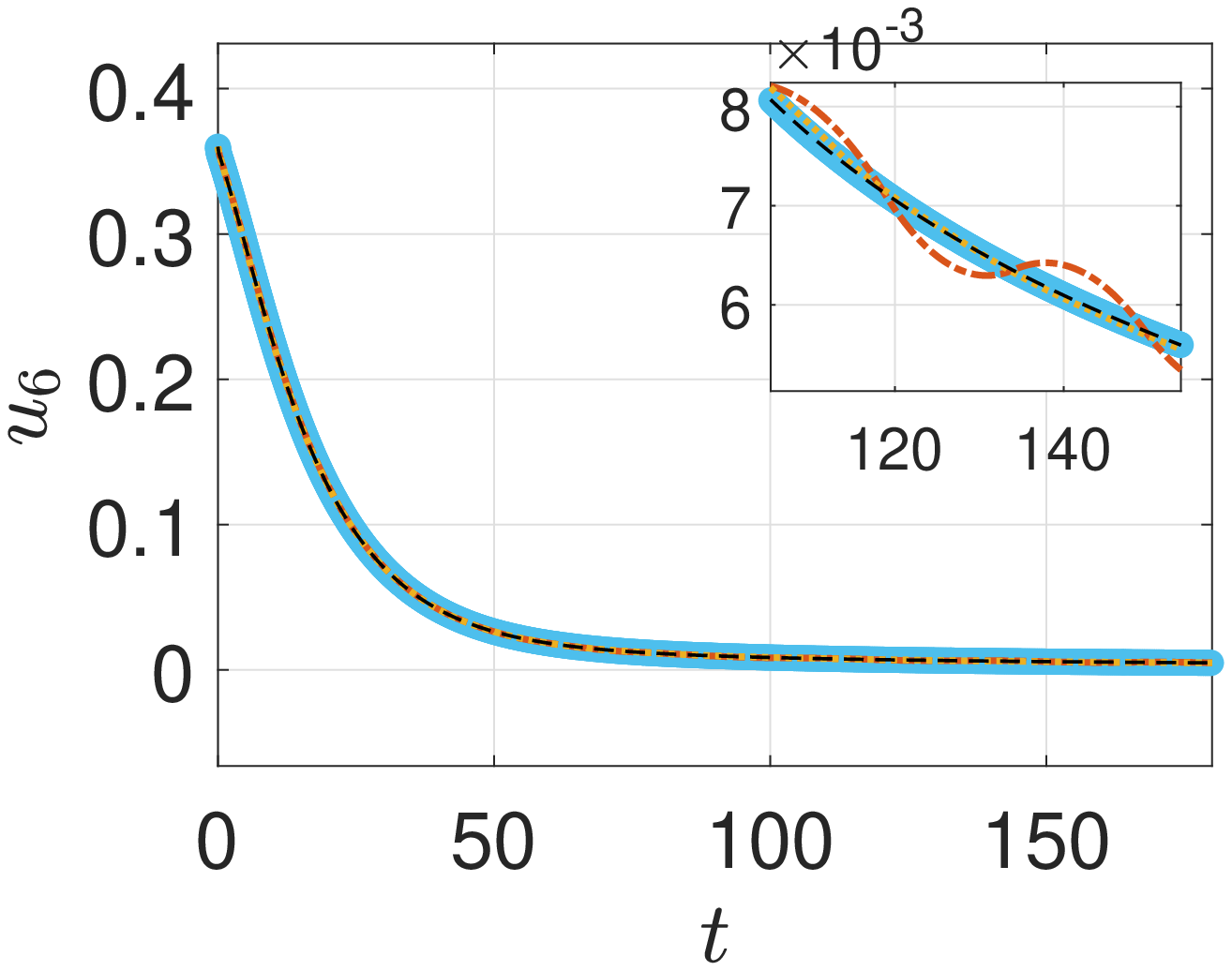}}
    \caption{The chemical Akzo Nobel DAE problem \eqref{eq:chemakzo}. Approximate solutions computed in the interval $[0 \quad 180]$ with both absolute and relative tolerances set to~1e$-$03. Insets depict zooms around the reference solution.
    \label{fig:chemakzo_solutions}}
\end{figure}
As it is shown in Figure~\ref{fig:chemakzo_solutions}, for tolerances 1e$-$03, the proposed scheme achieves more accurate solutions than \texttt{ode23t} and \texttt{ode15s}.
\begin{table}[ht]
\begin{center}
\caption{Chemical Akzo Nobel problem \eqref{eq:chemakzo}. Computational times in seconds (median, minimum and maximum over $10$ runs) and Number of points required in the interval $[0, 180]$ by PIRPNN, \texttt{ode23t} and  \texttt{ode15s} with tolerances~1e$-$03 and~1e$-$06. The reference solution was computed by \texttt{ode15s} with tolerances equal to 1e$-$14.}
    {\footnotesize
    \setlength{\tabcolsep}{3pt}
    \begin{tabular}{|l |l l l r |l l l r|}
        \hline
        & \multicolumn{4}{c|}{$tol=$ 1e$-$03} & \multicolumn{4}{c|}{$tol=$ 1e$-$06} \\
        \cline{2-9}
        & median & min & max & \multicolumn{1}{l|}{\# pts} & median & min & max & \multicolumn{1}{l|}{\# pts}\\
        \hline
        \rowcolor{LightCyan}
        PIRPNN & 2.90e$-$02 & 2.61e$-$02 & 4.82e$-$02 &  192  &  2.70e$-$02 & 2.58e$-$02 & 4.21e$-$02 &  204\\
        \texttt{ode23s}  & 2.02e$-$03 & 1.90e$-$03 & 3.21e$-$03 &  31  & 6.28e$-$03 & 5.84e$-$03 & 2.23e$-$02 &  195\\
        \texttt{ode15s} & 3.54e$-$03 & 3.34e$-$03 & 3.73e$-$03 &  34  & 4.41e$-$03 & 4.17e$-$03 & 1.19e$-$02 &  107\\
        reference & 3.85e$-$02 & 3.71e$-$02 & 4.48e$-$02  &  1426 & 3.85e$-$02 & 3.71e$-$02 & 4.48e$-$02 & 3195\\
        \hline
    \end{tabular}
    }
\end{center}
\label{tab:chemakzo_time_points}
\end{table}
In Table~\ref{tab:chemakzo_time_points}, we report the number of points and the computational times required by each method, including the time for computing the reference solution.
As shown, the corresponding total number of points required by the proposed scheme is comparable with the ones required by \texttt{ode23t} and \texttt{15s} and significantly smaller than the number of points required by the reference solution.
Furthermore, the computational times of the proposed method are comparable with the ones required by the \texttt{ode23t} and \texttt{15s}, thus outperforming the ones required for computing the reference solution.

\subsection{Case Study 5: The Belousov–Zhabotinsky stiff ODEs}
The Belousov–Zhabotinsky chemical reactions model \cite{belusov1959periodically,zhabotinsky1964periodical} 
is given by the following system of seven ODEs:
\begin{equation}
    \begin{aligned}
    &\dfrac{dA}{dt}= -k_1AY, \\
    &\dfrac{dY}{dt}= -k_1AY - k_2XY + k_5Z,\\
    &\dfrac{dX}{dt}= k_1AY - k_2XY + k_3BX - 2k_4X^2,\\ &\dfrac{dP}{dt}= k_2XY, \\
    &\dfrac{dB}{dt}= -k_3BX, \\
    &\dfrac{dZ}{dt}= k_3BX - k_5Z, \\ 
    &\dfrac{dQ}{dt}= k_4X^2.
    \end{aligned}
    \label{eq:BZ}
\end{equation}
$A,B,P,Q,X,Y,Z$ are the concentrations of chemical species and $k_1 = 4.72,$ $k_2 = 3 \times 10^9,$ $k_3 = 1.5 \times 10^4,$ $k_4 = 4 \times 10^7,$ $k_5 = 1$ are the reaction coefficients.
The initial conditions are set as $A(0) = B(0) = 0.066,$ $Y(0) = X(0) = P(0) = Q(0) = 0,$ $Z(0) = 0.002$.
This is a very stiff problem, due to the different scales of the reaction coefficients, thus exhibiting very sharp gradients. Indeed, for an acceptable solution is needed a tolerance at least of 1e$-$07 as also reported in \cite{shulyk2008numerical}. Here, we have compared the solution obtained with the proposed PIRPNN, with the stiff solvers \texttt{ode23s} and \texttt{ode15s}.
\begin{table}[ht]
\begin{center}
\caption{The Belousov–Zhabotinsky stiff ODEs problem \eqref{eq:BZ} in the time interval $[0 \quad 40]$. $l^2$, $l^{\infty}$ and mean absolute (MAE) approximation errors obtained with both absolute and relative tolerances set to 1e$-$07 and 1e$-$08.
\label{tab:BZ_accuracy}}
    {\footnotesize
    \begin{tabular}{|l|l |l l l |l l l|}
        \hline
\multicolumn{2}{|c|}{} & \multicolumn{3}{c|}{$tol=$ 1e$-$07} & \multicolumn{3}{c|}{$tol=$ 1e$-$08} \\
\cline{3-8}
        \multicolumn{2}{|c|}{} & $l^2$ & $l^{\infty}$ & MAE & $l^2$ & $l^{\infty}$ & MAE\\
        \hline
        \rowcolor{LightCyan}
        & PIRPNN & 3.23e$-$04 & 5.87e$-$06 & 7.24e$-$07 & 6.92e$-$06 & 1.27e$-$07 & 1.54e$-$08\\
        $A$ & \texttt{ode23s}  & 2.18e$-$02 & 4.40e$-$04 & 4.06e$-$05 & 2.44e$-$04 & 4.67e$-$06 & 6.28e$-$07\\
        & \texttt{ode15s} &  1.30e$-$01 & 1.37e$-$03 & 4.11e$-$04 & 8.29e$-$02 & 1.25e$-$03 & 1.69e$-$04 \\
        \hline
        \rowcolor{LightCyan}
        & PIRPNN & 8.27e$-$04 & 4.09e$-$05 & 1.34e$-$06 & 1.78e$-$05 & 9.05e$-$07 & 2.86e$-$08\\
        $Y$ & \texttt{ode23s}  & 4.22e$-$02 & 1.19e$-$03 & 7.29e$-$05  & 6.69e$-$04 & 3.51e$-$05 & 1.14e$-$06\\
        & \texttt{ode15s} & 1.06e$-$01 & 1.31e$-$03 & 3.28e$-$04 & 5.88e$-$02 & 1.19e$-$03 & 1.12e$-$04 \\
        \hline
        \rowcolor{LightCyan}
        & PIRPNN & 4.77e$-$05 & 1.04e$-$05 & 1.29e$-$08 & 1.94e$-$06 & 9.43e$-$07 & 2.79e$-$10\\
        $X$ & \texttt{ode23s}  & 2.60e$-$04 & 1.23e$-$05 & 1.87e$-$07 & 4.88e$-$05 & 1.16e$-$05 & 1.09e$-$08\\
        & \texttt{ode15s} & 3.21e$-$04 & 1.23e$-$05 & 2.71e$-$07 & 2.26e$-$04 & 1.23e$-$05 & 1.46e$-$07 \\
        \hline
        \rowcolor{LightCyan}
        & PIRPNN & 7.08e$-$04 & 4.15e$-$05 & 1.08e$-$06 & 1.53e$-$05 & 9.05e$-$07 & 2.32e$-$08\\
        $P$ & \texttt{ode23s}  & 3.27e$-$02 & 8.16e$-$04 & 6.07e$-$05 & 5.54e$-$04 & 3.48e$-$05 & 9.69e$-$07 \\
        & \texttt{ode15s} &  1.98e$-$01 & 2.03e$-$03 & 6.49e$-$04 & 1.35e$-$01 & 1.87e$-$03 & 2.95e$-$04 \\
        \hline
        \rowcolor{LightCyan}
        & PIRPNN & 2.99e$-$03 & 1.76e$-$04 & 1.83e$-$06 &  6.48e$-$05 & 3.86e$-$06 & 3.93e$-$08\\
        $B$ & \texttt{ode23s}  & 1.09e$-$01 & 3.17e$-$03 & 1.09e$-$04 & 2.41e$-$03 & 1.49e$-$04 & 2.04e$-$06\\
        & \texttt{ode15s} & 4.05e$-$01 & 3.52e$-$03 & 1.25e$-$03 & 2.85e$-$01 & 3.18e$-$03 & 6.49e$-$04 \\
        \hline
        \rowcolor{LightCyan}
        & PIRPNN & 2.73e$-$03 & 1.73e$-$04 & 2.93e$-$06 &  5.91e$-$05 & 3.84e$-$06 & 6.27e$-$08\\
        $Z$ & \texttt{ode23s}  & 7.85e$-$02 & 2.56e$-$03 & 1.12e$-$04 & 2.21e$-$03 & 1.48e$-$04 & 2.46e$-$06\\
        & \texttt{ode15s} & 1.24e$-$01 & 2.84e$-$03 & 2.55e$-$04 & 6.88e$-$02 & 2.56e$-$03 & 9.36e$-$05 \\
        \hline
        \rowcolor{LightCyan}
        & PIRPNN & 1.26e$-$03 & 8.64e$-$05 & 7.36e$-$07  & 2.73e$-$05 & 1.92e$-$06 & 1.58e$-$08 \\
        $Q$ & \texttt{ode23s}  & 4.38e$-$02 & 1.28e$-$03 & 4.34e$-$05 & 1.01e$-$03 & 7.36e$-$05 & 7.90e$-$07\\
        & \texttt{ode15s} & 1.64e$-$01 & 1.42e$-$03 & 5.03e$-$04 & 1.15e$-$01 & 1.28e$-$03 & 2.62e$-$04 \\
        \hline
    \end{tabular}
    }
\end{center}
\end{table}
Table~\ref{tab:BZ_accuracy} summarizes the $l^2$, $l^{\infty}$ and mean absolute (MAE) approximation errors with respect to the reference solution in $40,000$ equally spaced grid points in the interval $[0 \quad 40]$. As shown, for the given tolerances, the proposed method outperforms \texttt{ode15s} and \texttt{ode23s} in all metrics.
\begin{figure}
    \centering
    \subfigure[]{\includegraphics[width=0.3 \textwidth]{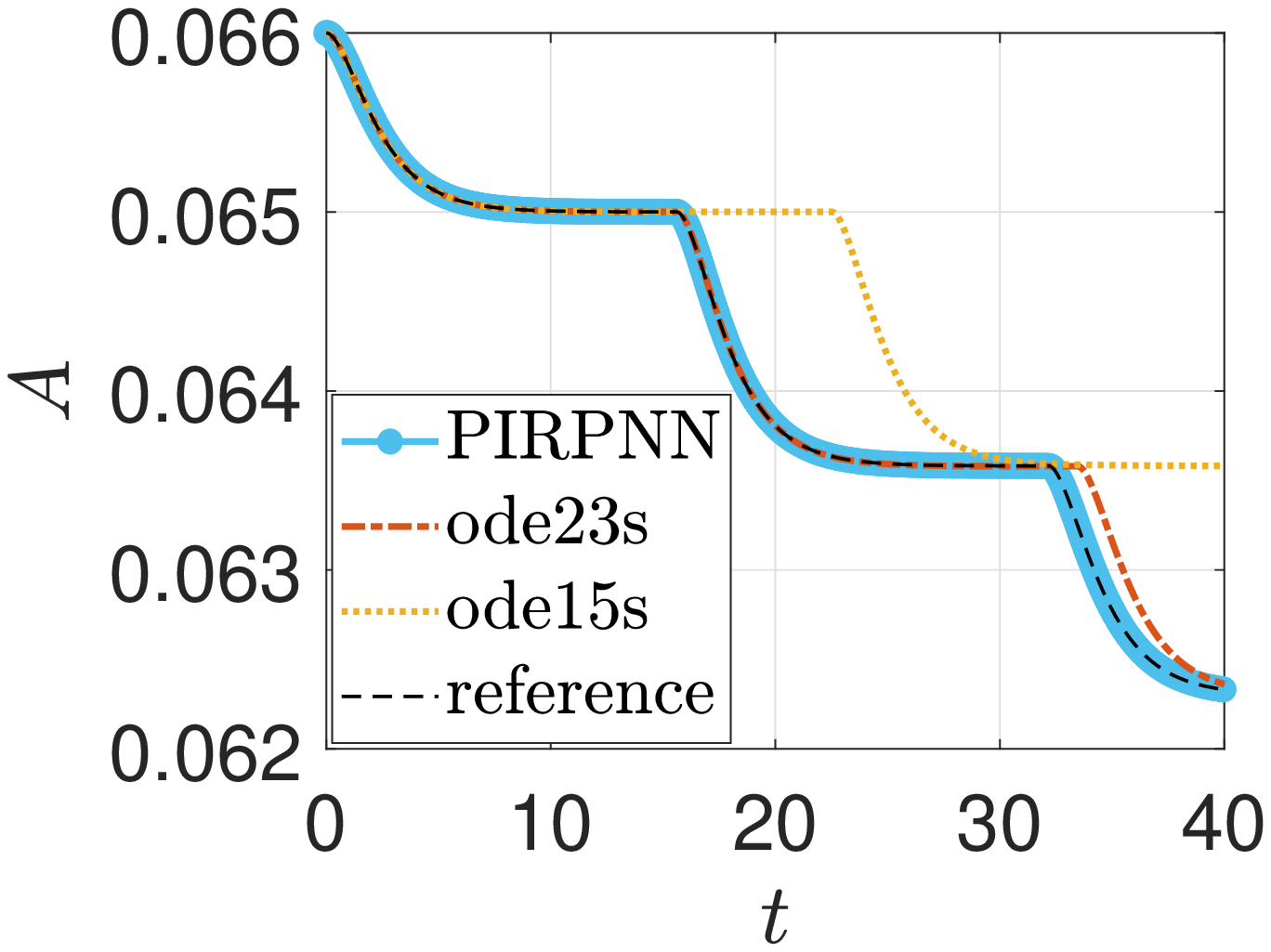}}
    \subfigure[]{\includegraphics[width=0.3 \textwidth]{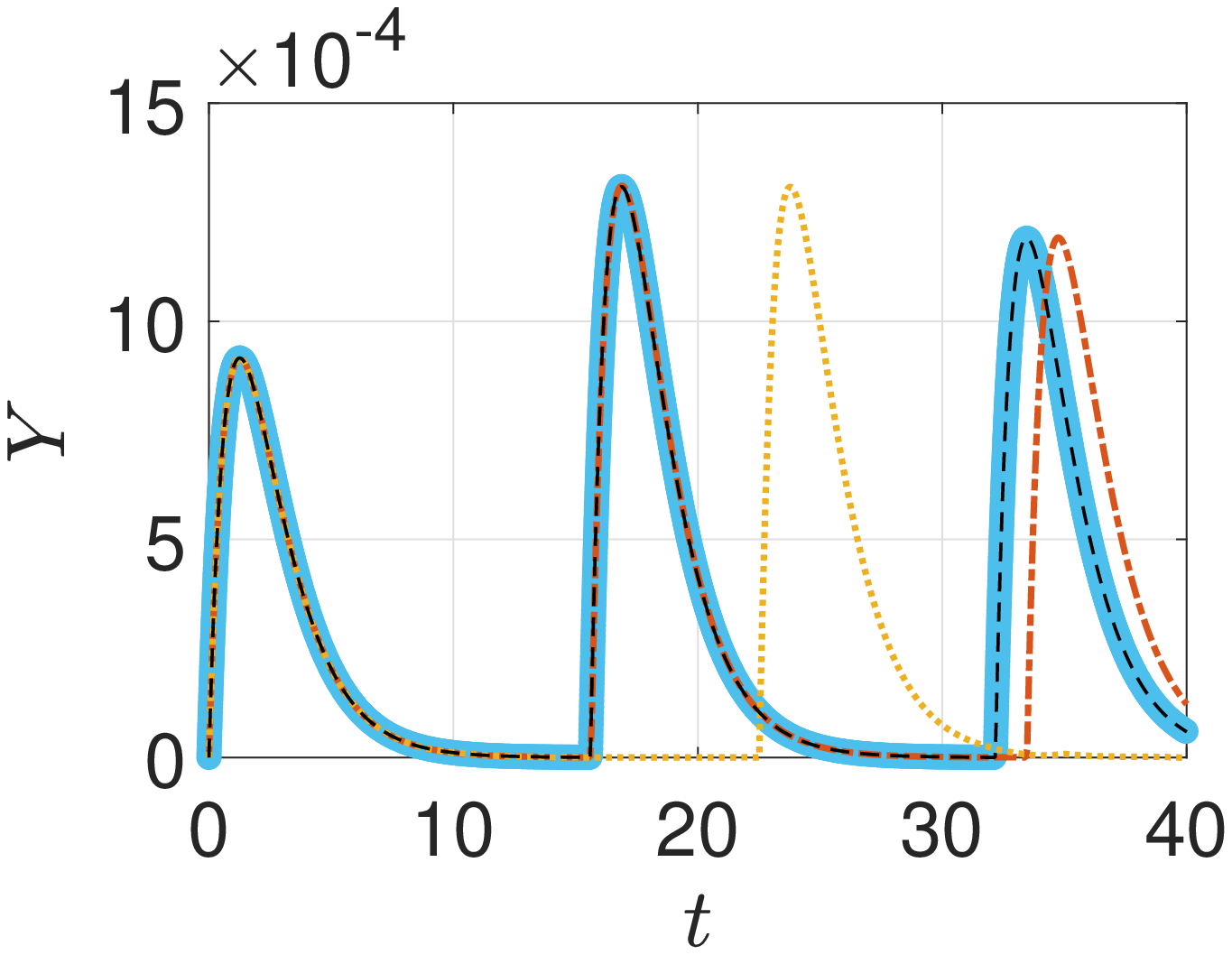}}
    \subfigure[]{\includegraphics[width=0.3 \textwidth]{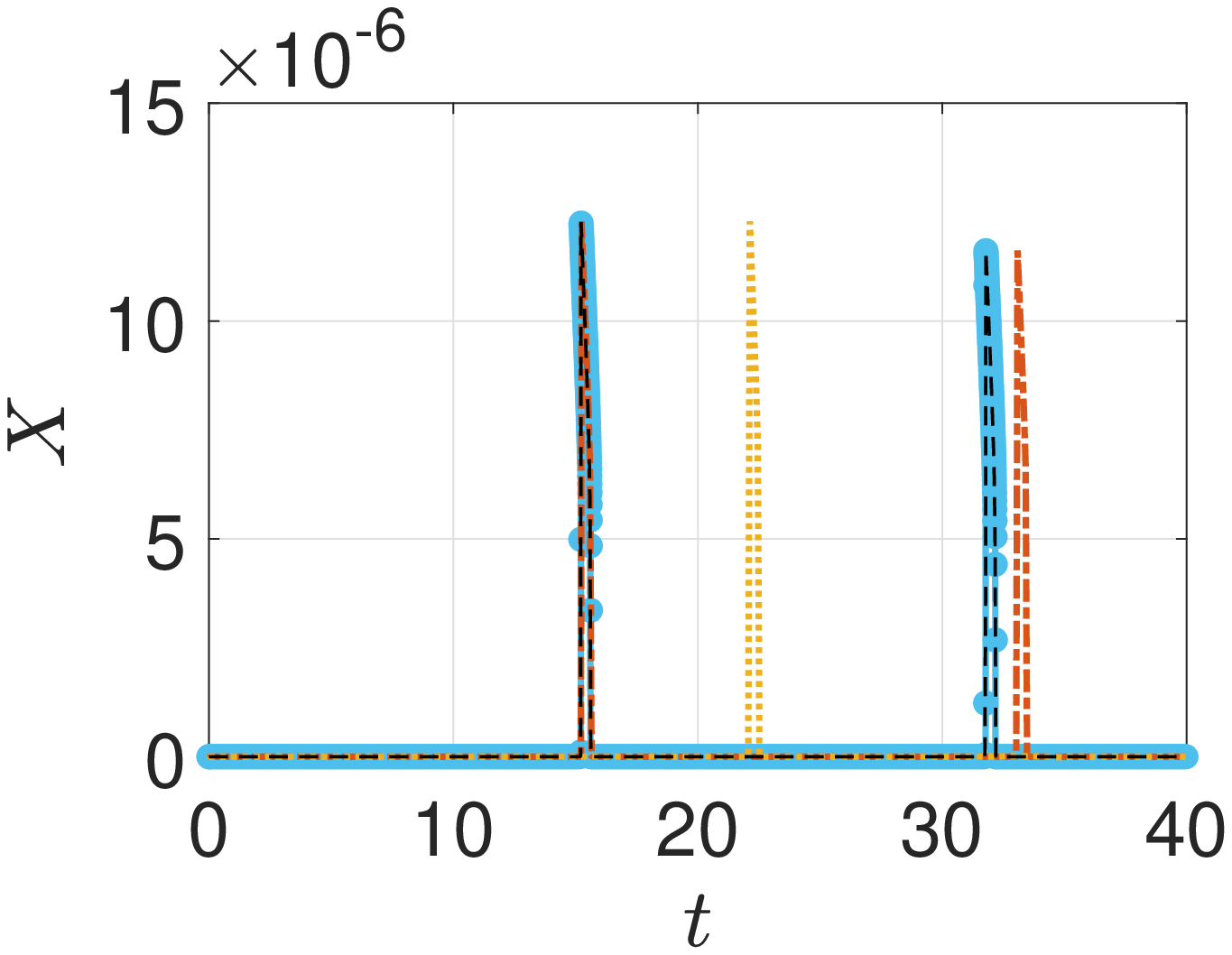}}
    \subfigure[]{\includegraphics[width=0.3 \textwidth]{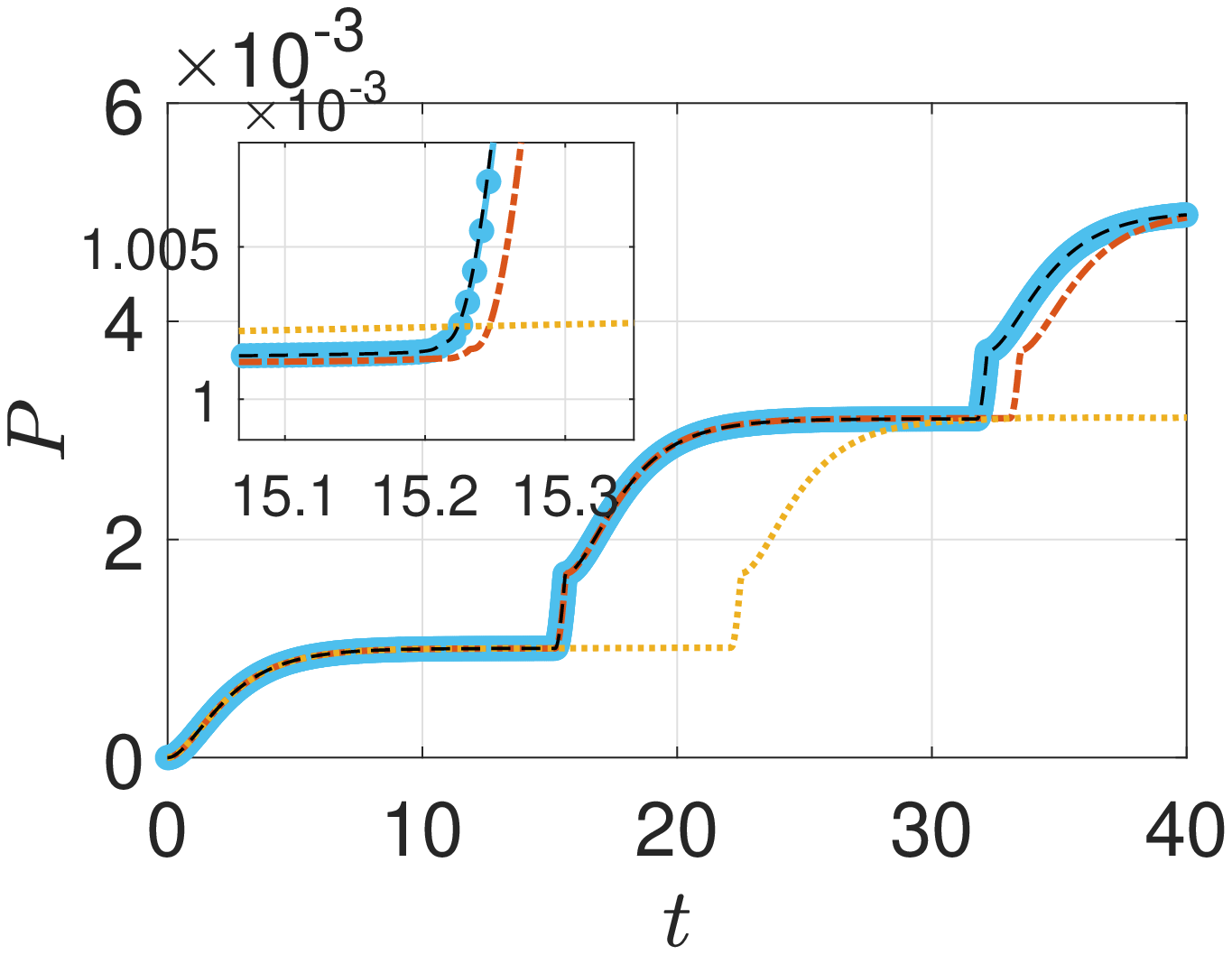}}
    \subfigure[]{\includegraphics[width=0.3 \textwidth]{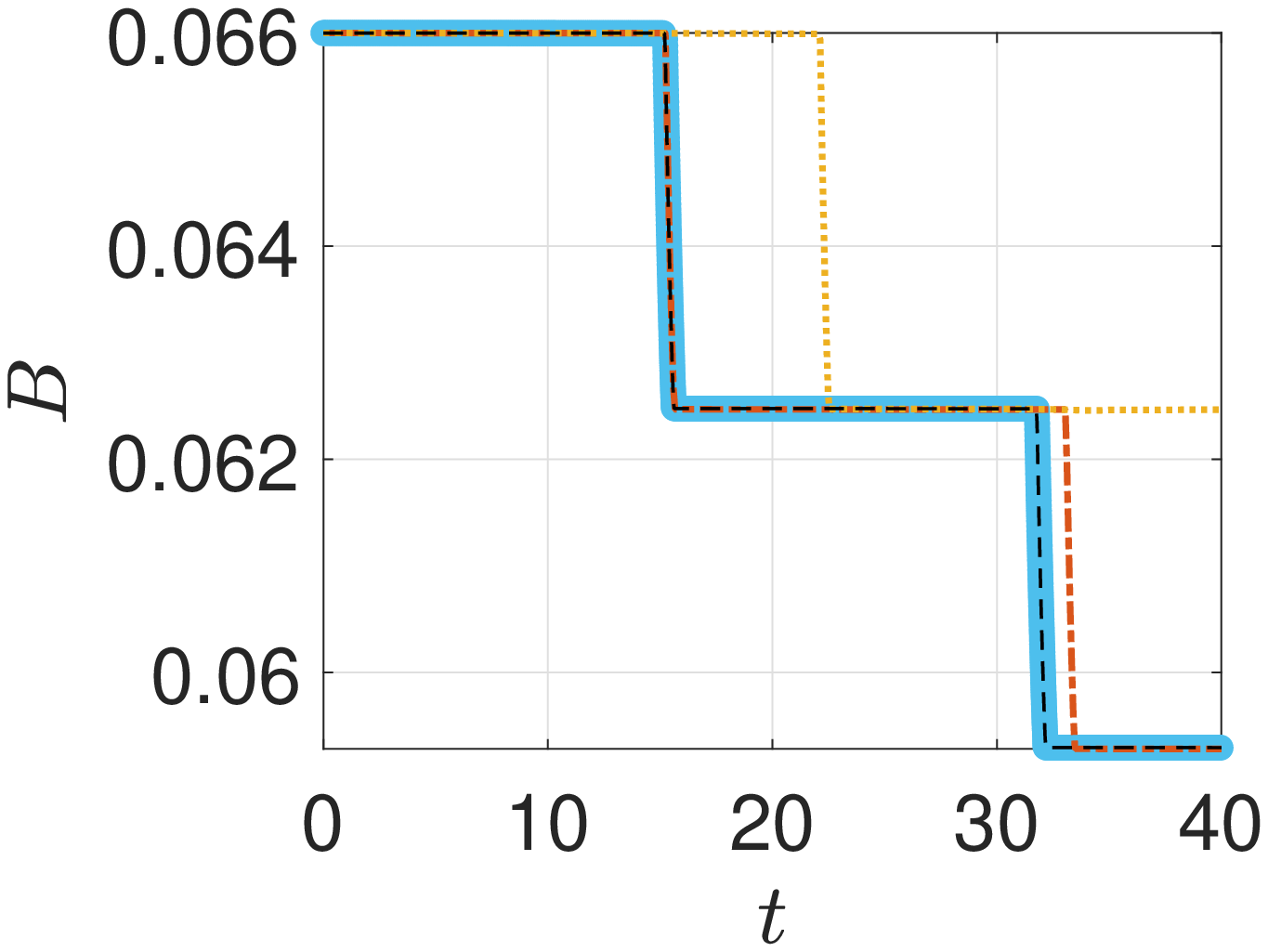}}
    \subfigure[]{\includegraphics[width=0.3 \textwidth]{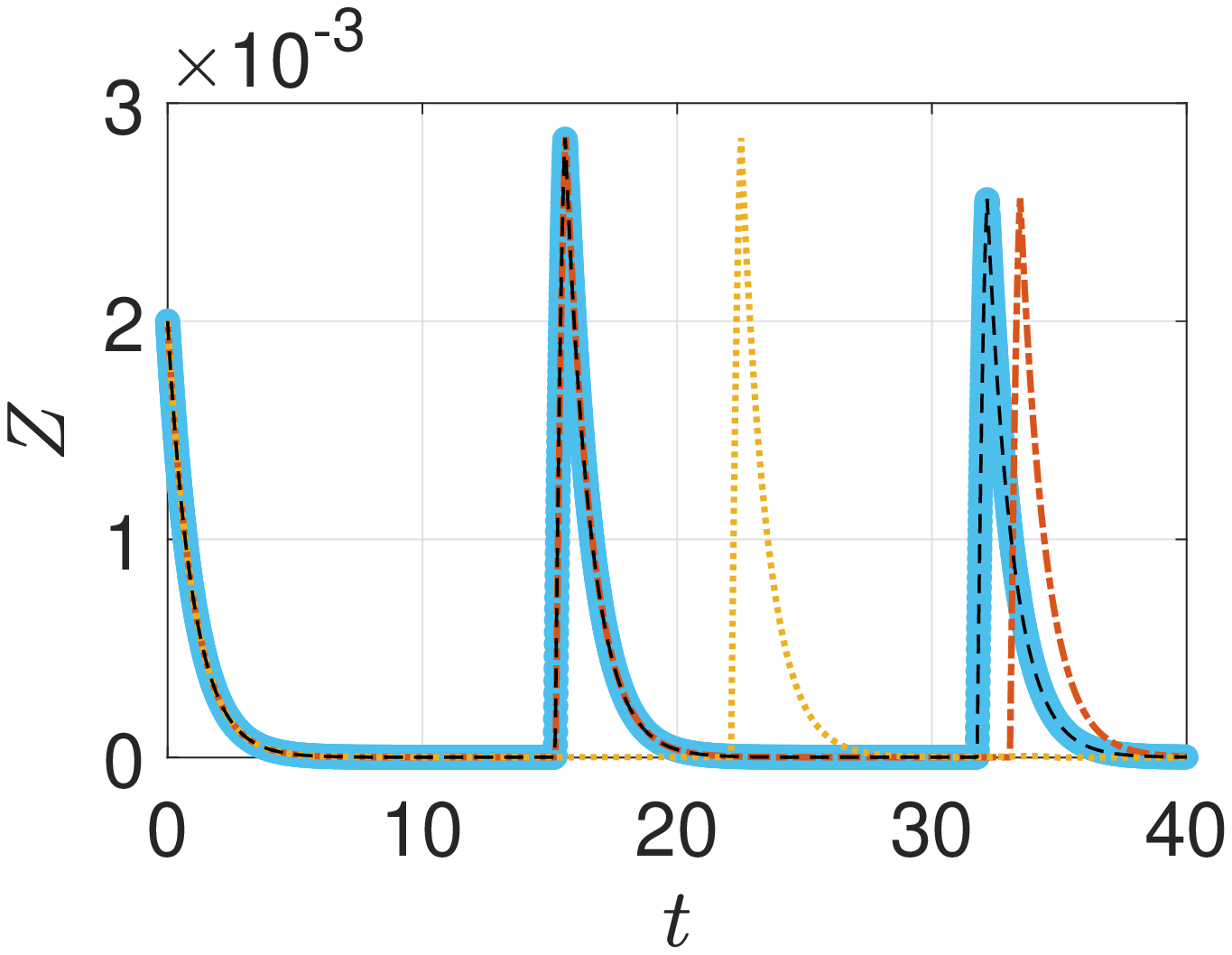}}
    \caption{The Belousov–Zhabotinsky stiff ODEs problem \eqref{eq:BZ} in the time interval $[0 \quad 40]$. Approximate solutions of 6 chemical species (out of 7) computed by the PIRPNN, ode23s, ode15s with tolerances set to~1e$-$07. 
    \label{fig:BZ_solutions}}
\end{figure}
As it is shown in Figure~\ref{fig:BZ_solutions}, for tolerances 1e$-07$, the proposed scheme achieves more accurate solutions than \texttt{ode23s} and \texttt{ode15s}. Actually, as shown, for the given tolerances, both \texttt{ode23s} and \texttt{ode15s} fail to converge resulting in time shifted jumps, while also \texttt{ode15s} fails to approximate the solution up to the final time.
\begin{table}[ht]
\begin{center}
\caption{Belousov-Zhabotinsky stiff ODEs problem \eqref{eq:BZ}. Computational times (s) (median, minimum and maximum over $10$ runs) and number of points required in the interval $[0 \quad 40]$ by the PIRPNN, \texttt{ode23s} and \texttt{ode15s} with both absolute and relative tolerances~1e$-$07 and~1e$-$08.}
    {\footnotesize
    \setlength{\tabcolsep}{3pt}
    \begin{tabular}{|l |l l l r |l l l r|}
        \hline
        & \multicolumn{4}{c|}{$tol=$ 1e$-$07} & \multicolumn{4}{c|}{$tol=$ 1e$-$08} \\
        \cline{2-9}
        & median & min & max & \multicolumn{1}{l|}{\# pts} & median & min & max & \multicolumn{1}{l|}{\# pts}\\
        \hline
        \rowcolor{LightCyan}
        PIRPNN & 4.62e$-$01 & 4.13e$-$01 & 5.88e$-$01 &  1624  & 5.54e$-$01 & 4.98e$-$01 & 6.19e$-$01 &  2056\\
        \texttt{ode23s}  & 1.04e$-$02 & 9.96e$-$03 & 1.57e$-$02 &  255  & 1.83e$-$02 &  1.77e$-$02  & 3.36e$-$02 &  480\\
        \texttt{ode15s} & 9.72e$-$03 & 9.44e$-$03 & 1.82e$-$02 &  251  & 1.37e$-$02 & 1.33e$-$02  & 2.52e$-$02 &  362\\
        reference & 8.79e$-$02 & 8.63e$-$02 & 1.59e$-$01  &  3195 & 8.79e$-$02  & 8.63e$-$02  & 1.59e$-$01 & 3195\\
        \hline
    \end{tabular}
    }
\end{center}
\label{tab:BZ_time_points}
\end{table}
In Table~\ref{tab:BZ_time_points}, we report the computational times and number of points required by each method, including the ones required for computing the reference solution. As shown, the corresponding total number of points required by the proposed scheme is larger than the ones required by \texttt{ode23t} and \texttt{15s} (which however fail), but yet significantly less than the number of points required by the reference solution. Besides, the computational times required by the proposed method are for any practical purposes comparable with the ones required for computing the reference solution.

\subsection{Case Study 6: The Allen-Cahn phase-field PDE}
The Allen-Cahn equation is a famous reaction-diffusion PDE that was proposed in \cite{allen1979microscopic} as a phase-field model for describing the dynamics of the mean curvature flow. Here, for our illustrations, we considered a one-dimensional formulation given by \cite{trefethen2000spectral}:
\begin{equation}
\begin{aligned}
& \dfrac{\partial u}{\partial t}=\nu\dfrac{\partial^2 u}{\partial x^2}+u-u^3, \qquad x \in [-1 \quad 1], \qquad
& u(-1,t)=-1, \quad u(1,t)=1,
\end{aligned}
\label{eq:Allen_Cahn}
\end{equation}
with initial condition $u(x,0) = 0.53 \, x + 0.47 \, \sin(-1.5\, \pi \, x)$.
For $\nu=0.01$, the solution is stiff \cite{trefethen2000spectral}, thus exhibiting a metastable behavior with an initial two-hill configuration that disappears close to $t=40$ with a fast transition to a one-hill stable solution. Here, we integrate this until $t=70$.
To solve the Allen-Cahn PDE, we used an equally spaced grid using (as in \cite{trefethen2000spectral}) $102$ points $x_0,x_1,\dots,x_{100},x_{101}$ and second order centered finite differences. Hence, \eqref{eq:Allen_Cahn} becomes a system of $100$ ODEs in $u_i(t)=u(x_i,t),$ $i=1,\dots,100$:
\begin{equation}
\begin{aligned}
& \dfrac{\partial u_i}{\partial t}=\nu\dfrac{(u_{i+1}-2u_i+u_{i-1})}{dx^2}+u_i-u_i^3, \quad
& u_0=-1, \quad u_{101}=1.
\end{aligned}
\label{eq:Allen_Cahn_FD}
\end{equation}
Here, for our computations, we have used a sparse QR decomposition as implemented in the SuiteSparseQR \cite{davis2009user,davis2011algorithm}.
\begin{table}[ht]
\begin{center}
\caption{Allen-Cahn phase-field PDE \eqref{eq:Allen_Cahn_FD} with $\nu =0.01$ in $[0 \quad 70]$. $l^{2}$, $l^{\infty}$ and mean absolute approximation (MAE) errors obtained with both absolute and relative tolerances set to 1e$-$03 and 1e$-$06.}
{\footnotesize
\begin{tabular}{|l|lll|lll|}
\hline
& \multicolumn{3}{c|}{$tol=$ 1e$-$03} & \multicolumn{3}{c|}{$tol=$ 1e$-$06} \\
\cline{2-7}
& $l^2$ & $l^{\infty}$ & MAE & $l^2$ & $l^{\infty}$ & MAE\\
\hline
\rowcolor{LightCyan}
PIRPNN & 6.36e$-$03 & 8.01e$-$05 & 2.19e$-$06 & 2.07e$-$05 & 1.43e$-$07 & 8.36e$-$09\\
\texttt{ode23s}  & 8.98e$-$01 & 1.12e$-$02 & 3.15e$-$04 & 2.46e$-$03 & 2.69e$-$05 & 1.16e$-$06\\
\texttt{ode15s} & 3.74e$+$00 & 4.50e$-$02 & 1.17e$-$03 & 1.84e$-$03 & 2.33e$-$05 & 6.23e$-$07 \\
\hline
\end{tabular}
}
\end{center}
\label{tab:AC_accuracy}
\end{table}
Table~\ref{tab:AC_accuracy} summarizes the $l^2$, $l^{\infty}$ and mean absolute (MAE) approximation errors with respect to the reference solution in $7\,000 \times 102$ equally spaced grid points in the time interval $[0 \quad 70]$ and in the space interval $[-1 \quad 1]$, respectively. As shown, for the given tolerances, the proposed method outperforms \texttt{ode15s} and \texttt{ode23s} in all metrics.
\begin{figure}
    \centering
    \subfigure[]{
    \includegraphics[width=0.45 \textwidth]{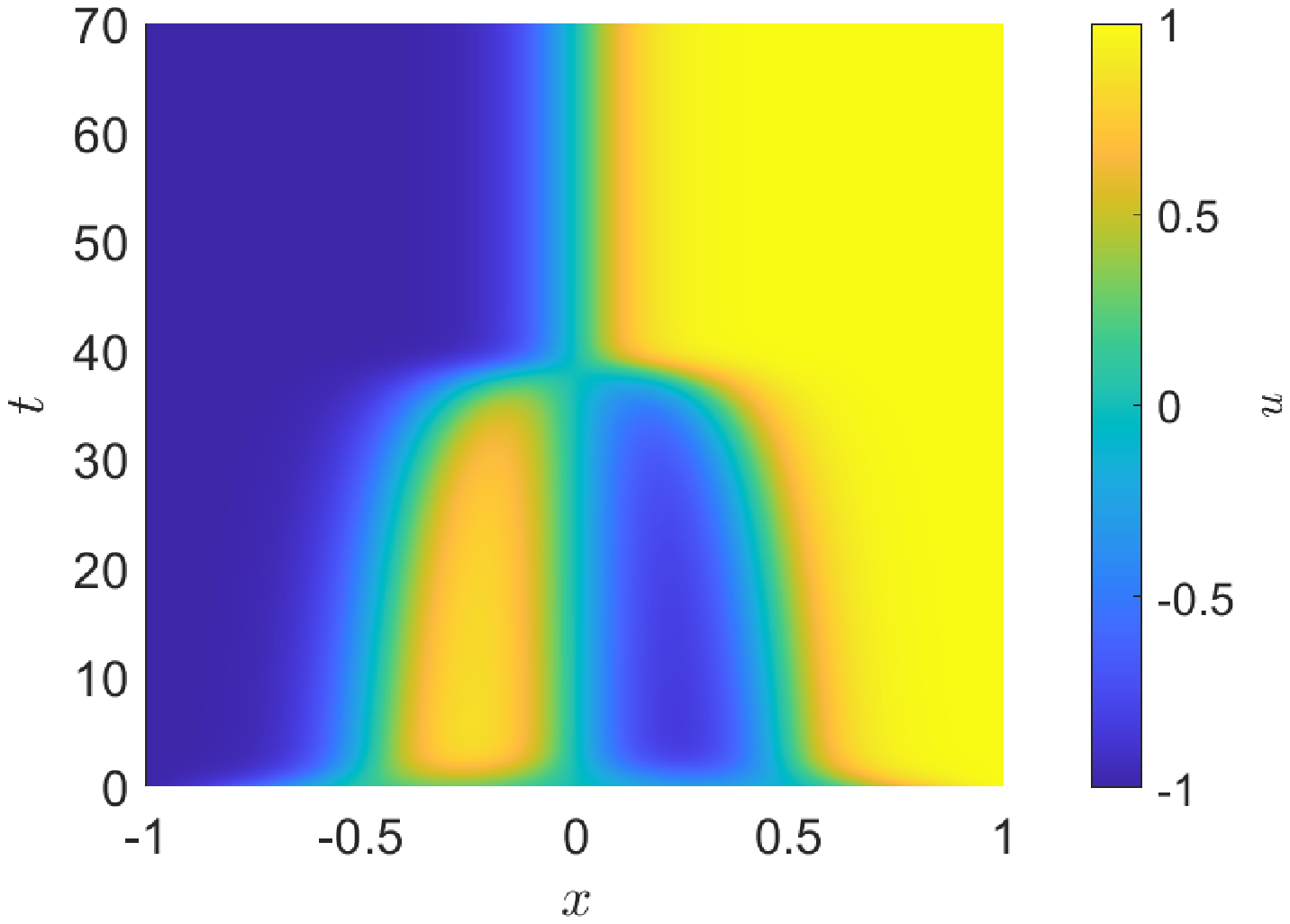}
    }
    \subfigure[]{
    \includegraphics[width=0.45 \textwidth]{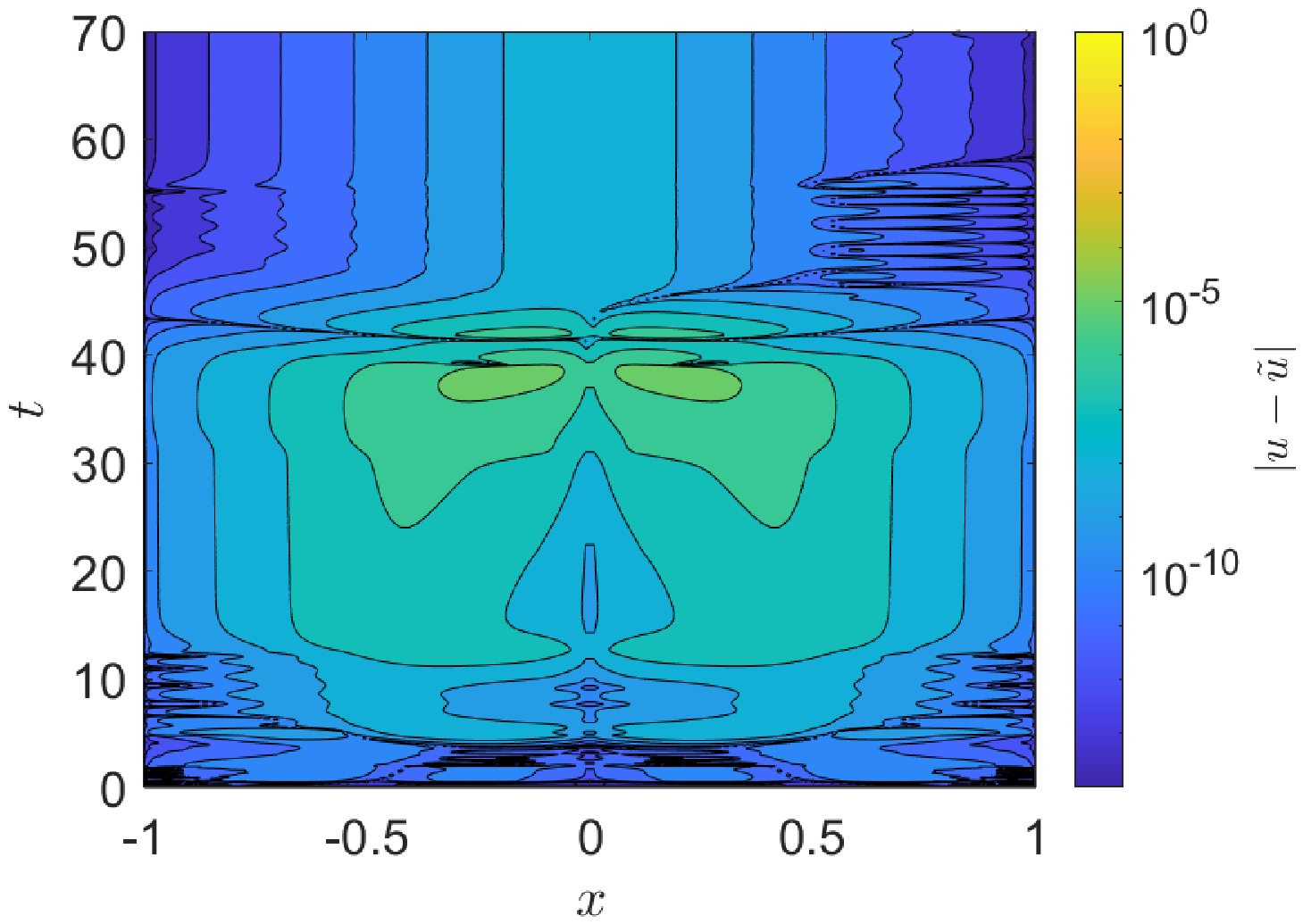}
    }
    \subfigure[]{
    \includegraphics[width=0.45 \textwidth]{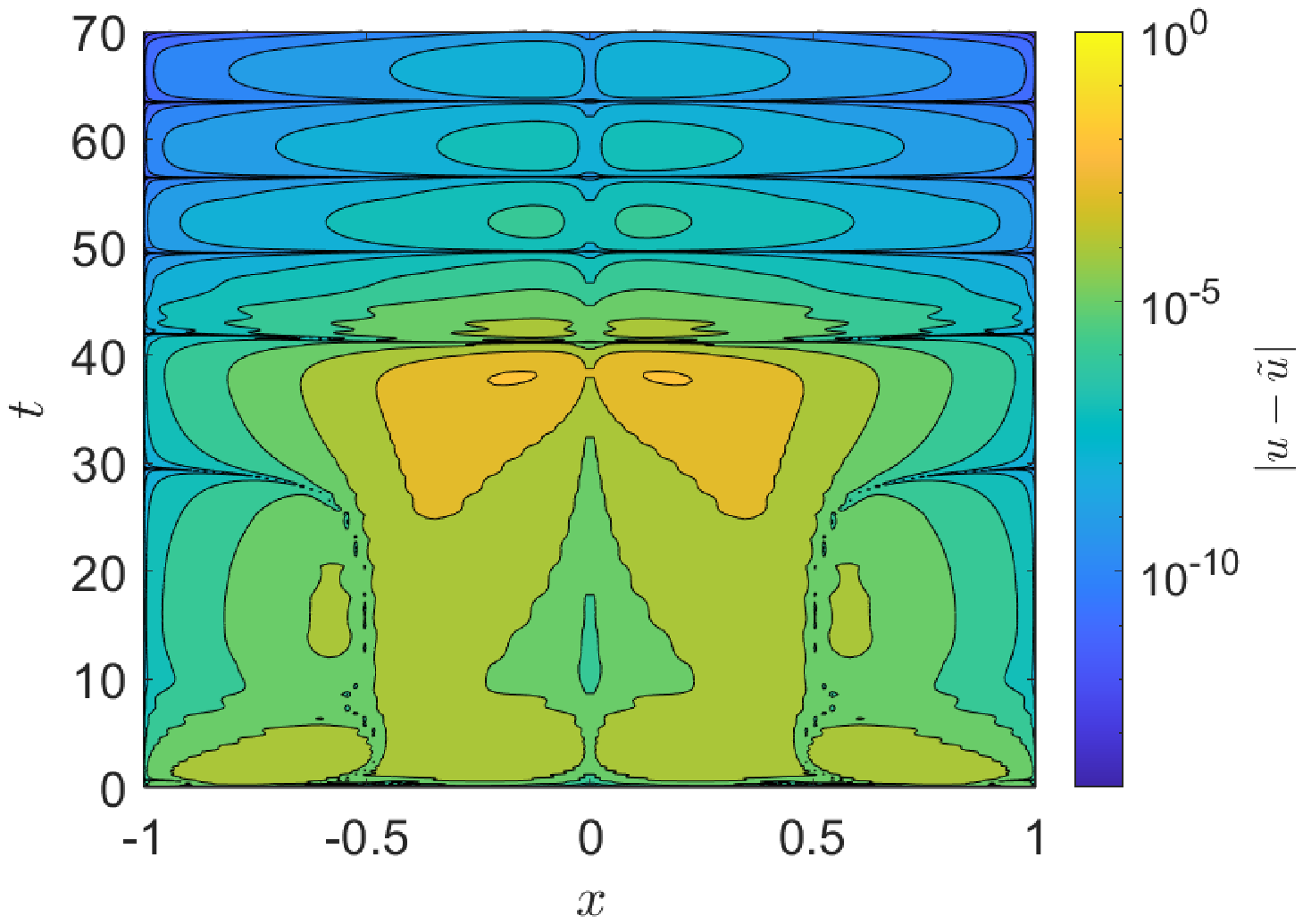}
    }
    \subfigure[]{
    \includegraphics[width=0.45 \textwidth]{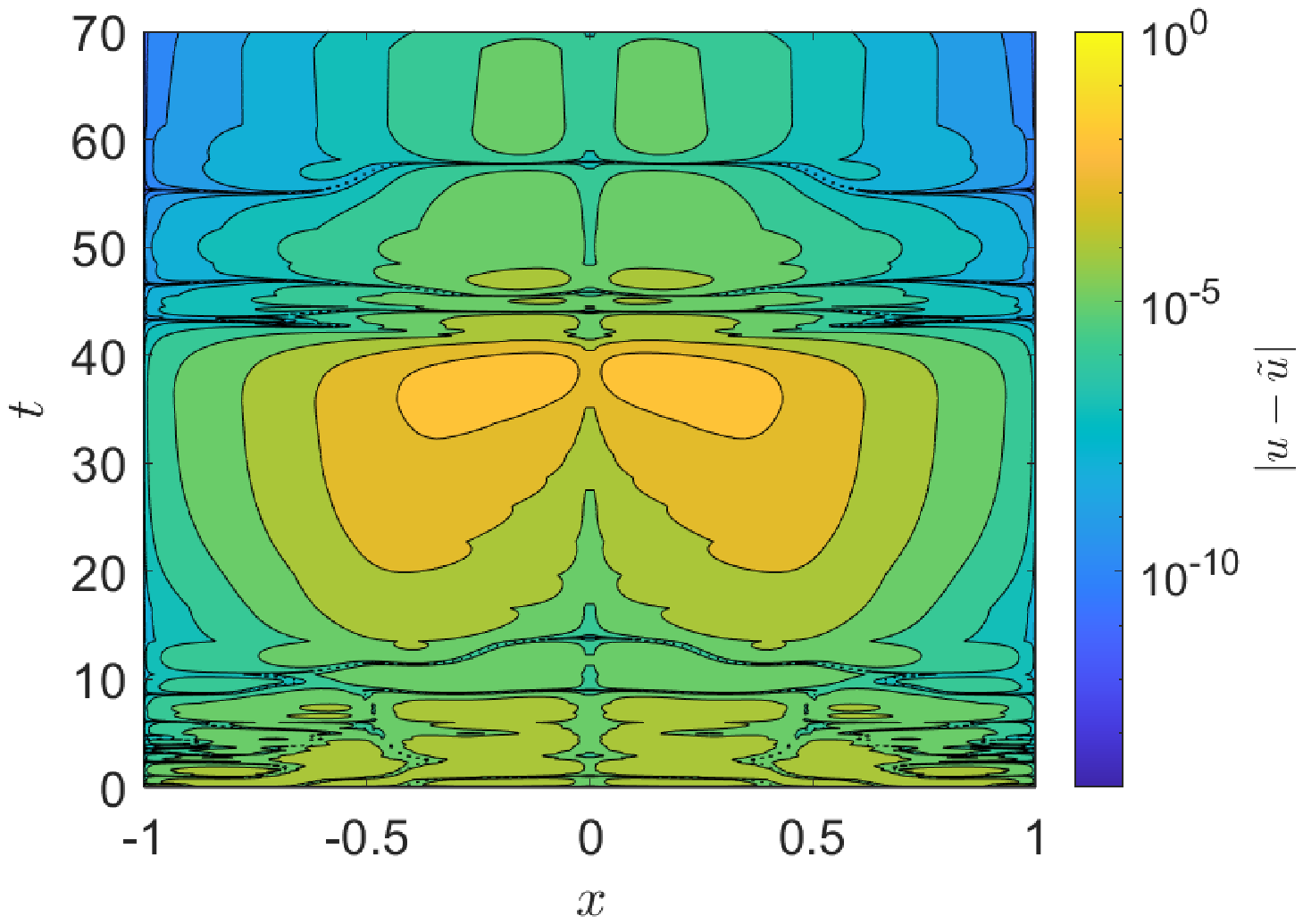}
    }
    \caption{Allen-Cahn phase-field PDE  discretized with FD \eqref{eq:Allen_Cahn_FD} with $nu=0.01$ in the time interval $[0 \quad 70]$. Contour plots of absolute point errors computed with tolerance 1e$-$03: a) reference solution computed with \texttt{ode15s} with tol=1e$-$14, b) PIRPNN absolute errors c) \texttt{ode23s} absolute errors d) \texttt{ode15s} absolute errors. \label{fig:AC_abs_err}
    }
\end{figure}
As it is shown in Figure~\ref{fig:AC_abs_err}, for tolerances 1e$-$03, the proposed scheme achieves more accurate solutions than \texttt{ode23s} and \texttt{ode15s}.
\begin{table}[ht]
\begin{center}
\caption{Allen-Cahn PDE phase-field PDE with $\nu=0.01$ in the time interval $[0 \quad 70]$. Computational times (s) (median, minimum and maximum over $10$ runs) and number of points required by PIRPNN, \texttt{ode23s} and \texttt{ode15s} with both absolute and relative tolerances set to 1e$-$03 and 1e$-$06. We note that most of the computational time for computing the solution with PIRPNN as shown, is due to the time required for the construction of the Jacobian matrix.}
    {\footnotesize
    \setlength{\tabcolsep}{3pt}
    \begin{tabular}{|l |l l l r |l l l r|}
        \hline
        & \multicolumn{4}{c|}{$tol=$ 1e$-$03} & \multicolumn{4}{c|}{$tol=$ 1e$-$06} \\
        \cline{2-9}
        & median & min & max & \multicolumn{1}{l|}{\# pts} & median & min & max & \multicolumn{1}{l|}{\# pts}\\
        \hline
        \rowcolor{LightCyan}
        PIRPNN & 5.49e$+$00 & 5.19e$+$00 & 5.68e$+$00 &  330  & 6.37e$+$00 & 5.80e$+$00 & 7.53e$+$00 &  380\\
        \rowcolor{LightRed}
        PIRPNN jac & 4.41e$+$00 & 4.38e$+$00 & 4.45e$+$00 &    & 5.05e$+$00 & 4.44e$+$00 & 5.35e$+$00 &  \\
        \texttt{ode23s}  & 1.03e$-$02 & 8.94e$-$03 & 1.88e$-$02 &  54  & 8.64e$-$02 & 8.35e$-$02 & 9.15e$-$02 &  529\\
        \texttt{ode15s} & 8.46e$-$03 & 7.87e$-$03 & 2.39e$-$02 &  68  & 1.67e$-$02 & 1.62e$-$02 & 2.32e$-$02 &  223\\
        reference & 2.35e$-$01 & 2.32e$-$01 & 2.56e$-$01 &  3706 & 2.41e$-$01 & 2.35e$-$01 & 2.76e$-$01 &  3706\\
        \hline
    \end{tabular}
    }
\end{center}
\label{tab:AC_time_points}
\end{table}
In Table~\ref{tab:AC_time_points}, we report computational times and number of points required by each method, including the ones required for computing the reference solution. As shown, the corresponding total number of points required by the proposed scheme is comparable with the ones required by \texttt{ode23s} and \texttt{ode15s} and significantly less than the number of points required by the reference solution.
On the other hand, the computing times of the proposed method are significantly larger than the ones required by \texttt{ode23s} and \texttt{ode15s} (Table \ref{tab:AC_time_points}) and also with the ones required by the reference solution. As reported, the higher computational cost is due to the time required for the construction of the Jacobian matrix, which even if it is sparse has a complex structure that make difficult to assemble it in an efficient way. This task is beyond the scope of this paper. However, in a subsequent work, we aim at implementing matrix-free methods in the Krylov subspace \cite{brown1994using,kelley1999iterative} such as Newton-GMRES for the solution of such large-scale problems.

\subsection{Case Study 7: The Kuramoto-Sivashinsky PDE}
The Kuramoto--Sivanshinsky (KS) \cite{sivashinsky1977nonlinear,trefethen2000spectral} equation is a celebrated fourth-order nonlinear PDE which exhibits deterministic chaos.
Here, we consider a one-dimensional formulation as proposed in \cite{trefethen2000spectral}:
\begin{equation}
\begin{aligned}
    &\frac{\partial u}{\partial t}=-uu_x-u_{xx}-u_{xxxx}, \qquad x \in [0,32\pi]
\end{aligned}
\label{eq:Kuramoto_Sivanshinsky}
\end{equation}
with periodic boundary condition and initial condition given by:
\begin{equation}
    u(x,0)=cos\biggl(\frac{1}{16}x\biggr)\biggl(1+sin(\frac{1}{16}x\biggr).
\end{equation}
The integration time is $[0 \quad 100]$. To solve the KS PDE, we have used an equispaced grid in space using 201 points $x_1,x_2,\dots,x_{200},x_{201}$, i.e., with a space step $dx=32\pi/201$, and second order central finite difference, so that \eqref{eq:Kuramoto_Sivanshinsky} became a system of 200 ODEs in the variables $u_i(t)=u(x_i,t)$:
\begin{equation}
\begin{aligned}
    &\frac{\partial u_i}{\partial t}=-u_i\frac{u_{i+1}-u_{i-1}}{2dx}-\frac{u_{i+1}-2u_{i}+u_{i-1}}{dx^2}-\frac{u_{i+2}-4u_{i+1}+6u_{i}-4u_{i-1}+u_{i-2}}{dx^4}\\
    &i=1,\dots,200.
\end{aligned}
\label{eq:Kuramoto_Sivanshinsky_FD}
\end{equation}
Here, as for the Allen-Cahn PDE, dealing with a resulting high dimensional Jacobian of size $4000\times4000$, we have used a sparse QR decomposition as implemented in the SuiteSparseQR \cite{davis2009user,davis2011algorithm}.
\begin{table}[ht]
\begin{center}
\caption{Kuramoto-Shivasinsky PDE discretized with central FD \eqref{eq:Kuramoto_Sivanshinsky_FD} in the time interval $[0 \quad 100]$. Absolute errors ($l^{2}$-norm, $l^{\infty}$-norm and MAE) for the solutions computed with both absolute and relative tolerances set to 1e$-$03 and 1e$-$06. The reference solution was computed with \texttt{ode15s} with tolerances equal to 1e$-$14.}
{\footnotesize
\begin{tabular}{|l|lll|lll|}
\hline
& \multicolumn{3}{c|}{$tol=$ 1e$-$03} & \multicolumn{3}{c|}{$tol=$ 1e$-$06} \\
\cline{2-7}
& $l^2$ & $l^{\infty}$ & MAE & $l^2$ & $l^{\infty}$ & MAE\\
\hline
\rowcolor{LightCyan}
RPNN   & 5.87e$+$01 & 5.77e$+$03 & 1.09e$+$04  & 8.20e$-$04 & 8.90e$-$06 & 1.29e$-$07\\
\texttt{ode23s}  & 9.40e$+$02 & 5.20e$+$00 & 2.40e$+$01 & 1.36e$+$01 & 1.39e$-$01 & 2.61e$-$03\\
\texttt{ode15s} & 7.10e$+$01 & 4.87e$+$01 & 2.30e$+$02 & 2.38e$-$01 & 1.88e$-$03 & 7.69e$-$05 \\
\hline
\end{tabular}
}
\end{center}
\label{tab:KS_accuracy}
\end{table}
Table~\ref{tab:KS_accuracy} summarizes the approximation errors, in terms of $l^2$-norm and $l^{\infty}$-norm errors and MAE, with respect to the reference solution in $100,000 \times 201$ equally spaced grid points in the time interval $[0 \quad 100]$ and in the space interval $[0 \quad 32\pi]$, respectively. As shown, for the given tolerances, the proposed method outperforms \texttt{ode15s} and \texttt{ode23s} in all metrics.
\begin{figure}
    \centering
    \subfigure[]{\includegraphics[width=0.45 \textwidth]{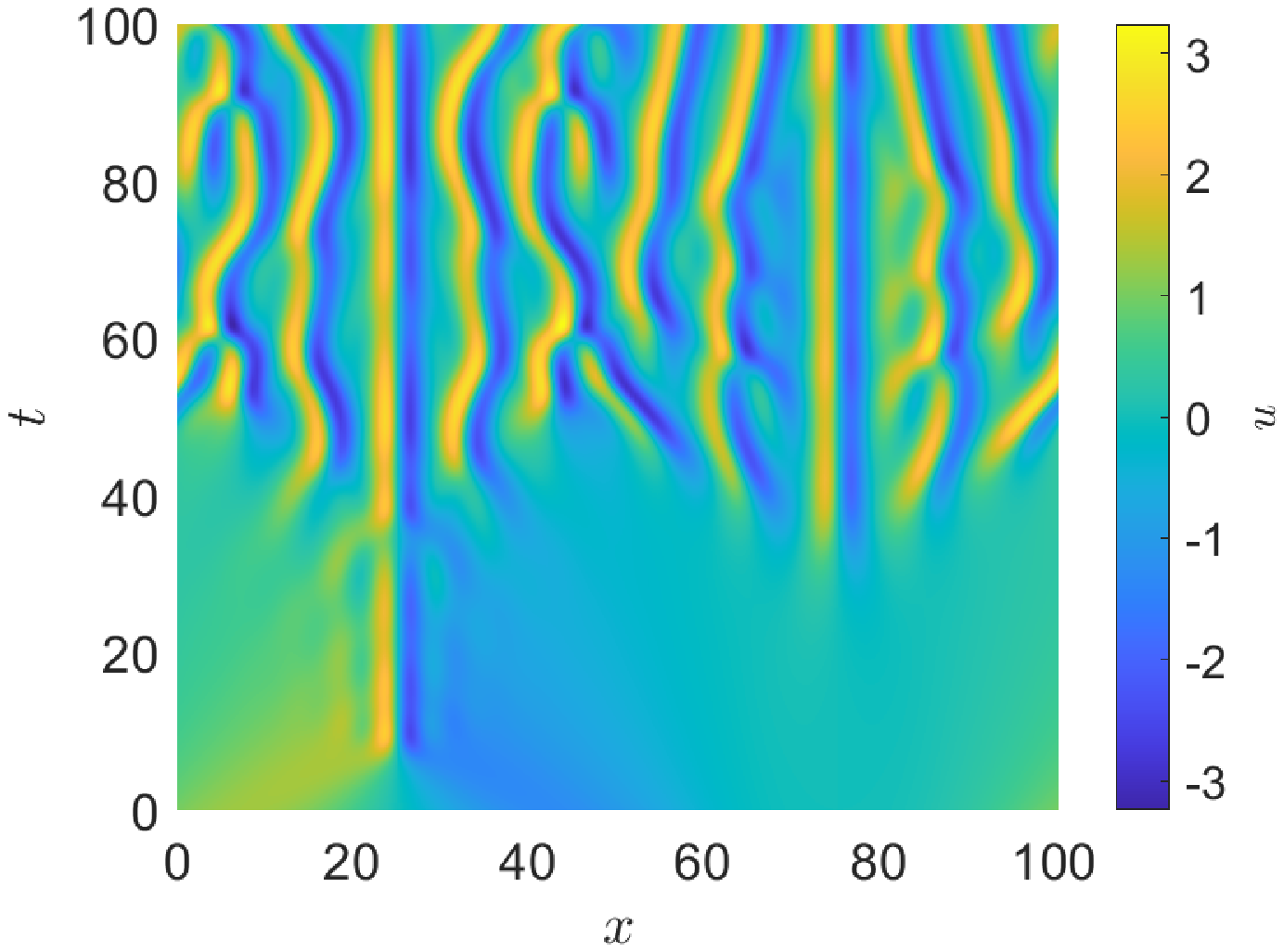}}
    \subfigure[]{\includegraphics[width=0.45 \textwidth]{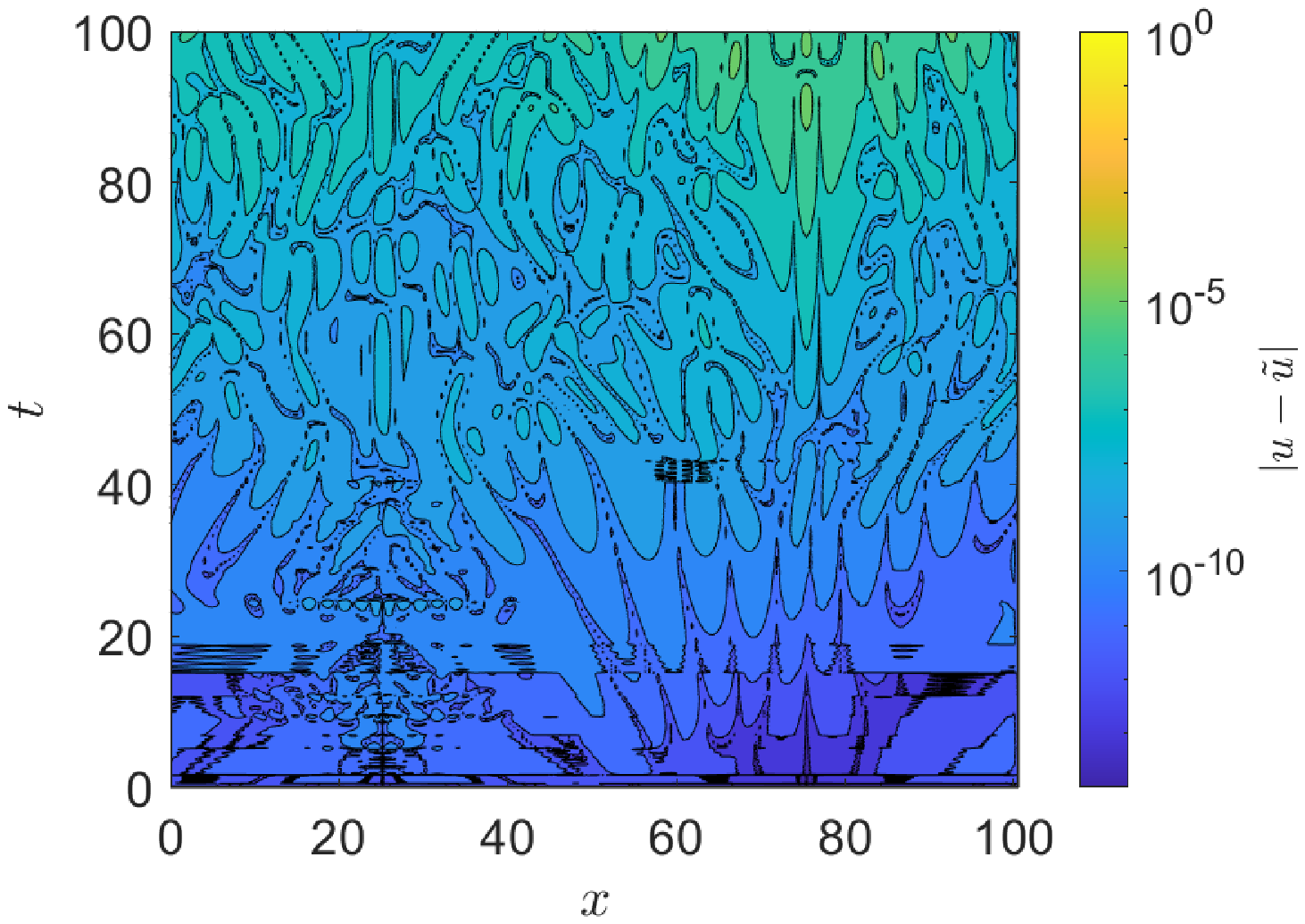}}
    \subfigure[]{\includegraphics[width=0.45 \textwidth]{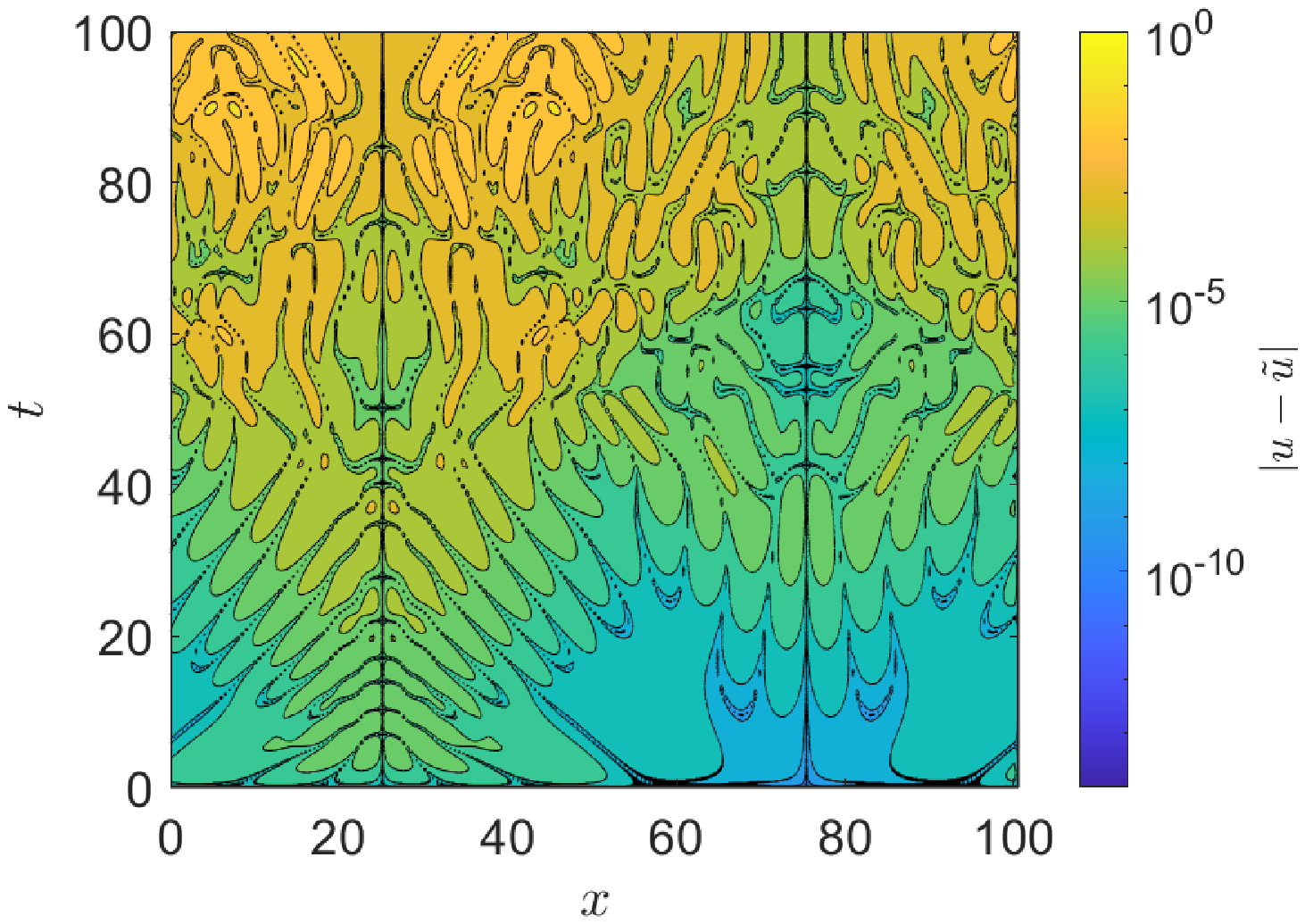}}
    \subfigure[]{\includegraphics[width=0.45 \textwidth]{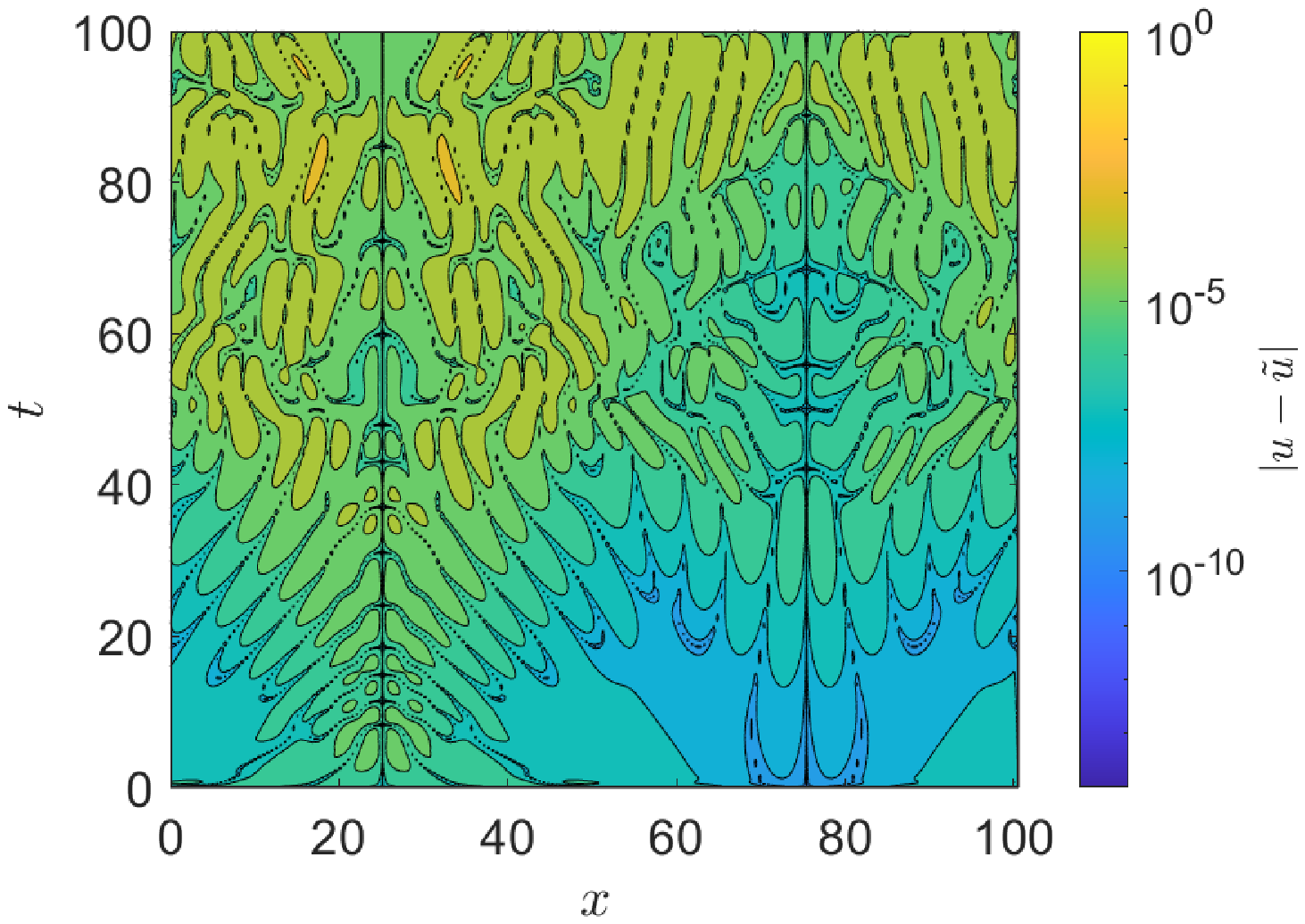}}
    \caption{Kuramoto-Sivashinsky PDE discretized with FD \eqref{eq:Kuramoto_Sivanshinsky_FD} in the time interval $[0 \quad 100]$. Contour plots of absolute point error computed with tolerance 1e$-$06: a) reference solution computed with \texttt{ode15s} with tol=1e$-$14, b) PIRPNN absolute errors c) \texttt{ode23s} absolute errors d) \texttt{ode15s} absolute errors. \label{fig:KS_abs_err}}
\end{figure}
As it is shown in Figure~\ref{fig:KS_abs_err}, for tolerances 1e$-$06, the proposed scheme achieves more accurate solutions than \texttt{ode23s} and \texttt{ode15s}.
\begin{table}[ht]
\begin{center}
\caption{Kuramoto-Shivasinsky PDE \eqref{eq:Kuramoto_Sivanshinsky_FD}. Computational times in seconds (median, minimum and maximum over 10 runs) and number of points required in the time interval $[0 \quad 100]$ by PIRPNN, \texttt{ode23s} and  \texttt{ode15s} with both absolute and relative tolerances set to 1e$-$03 and 1e$-$06. The reference solution was computed by \texttt{ode15s} with tolerances equal to 1e$-$14. We remark that most of the computational time for computing the solution with, reported in the row, is due to the time needed for the construction of the Jacobian, reported in the PIRPNN jac row, that for this problem has a dimension $4000\times4000$.}
    {\footnotesize
    \setlength{\tabcolsep}{3pt}
    \begin{tabular}{|l |l l l r |l l l r|}
        \hline
        & \multicolumn{4}{c|}{$tol=$ 1e$-$03} & \multicolumn{4}{c|}{$tol=$ 1e$-$06} \\
        \cline{2-9}
        & median & min & max & \multicolumn{1}{l|}{\# pts} & median & min & max & \multicolumn{1}{l|}{\# pts}\\
        \hline
        \rowcolor{LightCyan}
        PIRPNN & 8.81e$+$01  & 8.32e$+$01  & 9.47e$+$01 &  900  & 9.67e$+$01 & 9.13e$+$01 & 1.03e$+$02 &  1133 \\
        \rowcolor{LightRed}
        PIRPNN jac & 8.29e$+$01 & 6.90e$+$01 & 8.66e$+$01 &    & 7.71e$+$01 & 7.59e$+$01 & 8.53e$+$01 &  \\
        \texttt{ode23s}  & 1.38e$+$01 & 1.37e$+$01 & 1.65e$+$01 &  351   & 1.63e$+$00 & 1.59e$+$00 & 1.64e$+$00 &  4403\\
        \texttt{ode15s} & 3.22e$+$02 & 3.21e$+$02 & 4.45e$+$02 &  247  & 8.15e$-$02 & 7.43e$-$02 & 8.72e$-$02 &  749\\
        reference & 1.24e$+$00 & 1.21e$+$00 & 1.32e$+$00 &  14889 & 1.24e$+$00 & 1.21e$+$00 & 1.32e$+$00 &  14889 \\
        \hline
    \end{tabular}
    }
\end{center}
\label{tab:KS_time_points}
\end{table}
In Table~\ref{tab:KS_time_points}, we report computational times and number of points required by each method, including the ones required for computing the reference solution. As shown, the corresponding total number of points required by the proposed scheme is comparable with the ones required by \texttt{ode23s} and \texttt{ode15s} and significantly less than the number of points required by the reference solution.
On the other hand, the computing times of the proposed method are significantly larger than the ones concerning \texttt{ode23s} and \texttt{ode15s} (Table \ref{tab:KS_time_points}) and also with the ones required by the reference solution. As reported, the higher computational cost is due to the time required for the construction of the Jacobian matrix, which even if it is sparse has a complex structure that make difficult to assemble it. An efficient construction/assembly of the sparse Jacobian matrix is beyond the scope of this paper. However, in a subsequent work, we aim at implementing matrix-free methods in the Krylov subspace \cite{brown1994using,kelley1999iterative} such as Newton-GMRES for the solution of such large-scale problems.


\section{Discussion\label{sec:discussion}}
We proposed a physics-informed machine learning scheme based on the concept of random projections for the solution of IVPs of nonlinear ODEs and index-1 DAEs. The only unknowns are the weights from the hidden to the output layer which are estimated using Newton iterations.  To deal with the ill-posedness least-squares problem, we used  SVD decomposition when dealing with low-dimensional systems and sparse QR factorization with regularization when dealing with large-dimensional systems as for example those that arise from the discretization in space of PDEs. The hyper-parameters of the scheme, i.e., the bounds of the uniform distribution from which the values of the shape parameters of the Gaussian kernels are drawn and the interval of integration are parsimoniously chosen, based on the bias-variance trade-off concept and a variable step size scheme based on the elementary local error control algorithm. Furthermore, to facilitate the convergence of the scheme, we address a natural continuation method for providing good initial guesses for the Newton iterations.

The efficiency of the proposed scheme was assessed both in terms of numerical approximation accuracy and computation cost considering seven benchmark problems, namely the index-1 DAE Robertson model, a non autonomous index-1 DAEs mechanics problem, a non autonomous index-1 DAEs power discharge control problem, the chemical Akzo Nobel problem, the Belousov-Zhabotinsky ODEs model, the one-dimensional Allen-Cahn phase-field PDE and the one-dimensional Kuramoto-Sivashinky PDE. In addition, the performance of the scheme was compared against three stiff solvers of the MATLAB ODE suite, namely the \texttt{ode15s}, \texttt{ode23s} and \texttt{ode23t}. The results suggest that proposed scheme arises an alternative method to well established traditional solvers.

Future work is focused on the further development and application of the scheme for solving very large scale stiff and DAE problems (also of index higher than one) arising in many problems of contemporary interest, thus considering and integrating ideas from other methods such as DASSL \cite{petzold1982description}, CSP \cite{hadjinicolaou1998asymptotic} and matrix-free methods in the Krylov-subspace \cite{brown1994using,kelley1999iterative} in order to speed up computations for high-dimensional systems.

\section*{Acknowledgments}
This work was supported by the Italian program ``Fondo Integrativo Speciale per la Ricerca (FISR)'' - FISR2020IP 02893/ B55F20002320001.
G.F. is supported by a 4-year scholarship from the Scuola Superiore Meridionale, Universit\`a degli Studi di Napoli Federico II, Italy. E.G. was supported by a 3-year scholarship from the Universit\`a degli Studi di Napoli Federico II, Italy. 



\end{document}